\documentclass[11pt, reqno, dvipsnames]{amsart}
\usepackage[utf8]{inputenc}
\usepackage[english]{babel}
\usepackage{color}
\usepackage{amsmath,amsthm,amssymb,amsfonts}
\usepackage{mathrsfs}
\usepackage{tikz-cd}
\usepackage[linktocpage=true]{hyperref}
\usepackage{enumerate}
\usepackage{geometry}

\hypersetup{
	colorlinks = true,
	linkcolor = {ForestGreen},
	citecolor = {RoyalBlue},
	urlcolor = {Black}
}

\newcommand{\Z}{{\mathbb Z}}
\newcommand{\Q}{{\mathbb Q}}

\newcommand{\N}{{\mathbb N}}

\newcommand{\f}{{\mathcal F}}

\newcommand{\Gal}{\operatorname{Gal}}

\newcommand{\Spec}{\operatorname{Spec}}
\newcommand{\Sp}{\operatorname{Sp}}
\newcommand{\Spa}{\operatorname{Spa}}
\newcommand{\Shv}{\operatorname{Shv}}
\newcommand{\Spf}{\operatorname{Spf}}
\newcommand{\rg}{\operatorname{R\Gamma}}

\newcommand{\colim}{\operatorname{colim}}
\newcommand{\an}{\operatorname{an}}

\newcommand{\solid}{\operatorname{solid}}
\newcommand{\Mod}{\operatorname{Mod}}
\newcommand{\eh}{\operatorname{\acute{e}h}}
\newcommand{\hk}{\operatorname{R\Gamma_{HK}}}
\newcommand{\dr}{\operatorname{R\Gamma_{dR}}}
\newcommand{\cris}{\operatorname{R\Gamma_{cr}}}
\newcommand{\et}{\operatorname{\acute{e}t}}
\newcommand{\ket}{\operatorname{k\acute{e}t}}
\newcommand{\proet}{\operatorname{pro\acute{e}t}}

\newcommand{\hkc}{\operatorname{R\Gamma_{HK,c}}}
\newcommand{\drc}{\operatorname{R\Gamma_{dR,c}}}

\newcommand{\etc}{\operatorname{R\Gamma_{\acute{e}t,c}}}

\newcommand{\syn}{\operatorname{R\Gamma_{syn}}}

\newcommand{\fil}{\operatorname{Fil}}

\numberwithin{equation}{section}
\newtheorem{theorem}{Theorem}
\numberwithin{theorem}{section}

\newtheorem{thm}[theorem]{Theorem}
\newtheorem{lem}[theorem]{Lemma}

\newtheorem{prop}[theorem]{Proposition}
\newtheorem{conj}[theorem]{Conjecture}

\theoremstyle{definition}
\newtheorem{defn}[theorem]{Definition}
\newtheorem{rem}[theorem]{Remark}
\newtheorem{ex}[theorem]{Example}

\AtBeginDocument{%
	\def\MR#1{}
}

\numberwithin{equation}{section}
\begin{document}
	
	\title{Hyodo-Kato cohomology in rigid geometry: some foundational results}
	
	\author{Xinyu Shao}
	\address{Institut de Mathématiques de Jussieu-Paris Rive Gauche, Sorbonne Universit\'e, 4 place Jussieu, 75005 Paris, France}
	\email{xshao@imj-prg.fr}
	
	\maketitle

	\begin{abstract}
		By exploring the geometric properties of Hyodo-Kato cohomology in rigid geometry, we establish several foundational results, including the semistable conjecture for \'etale cohomology of almost proper rigid analytic varieties, and GAGA (comparison between algebraic and analytic) for Hyodo-Kato cohomology. A central component of our approach is the Gysin sequence for Hyodo-Kato cohomology, which we construct using the open-closed exact sequence for compactly supported Hyodo-Kato cohomology and Poincaré duality.
	\end{abstract}
	
	\tableofcontents

	\section{Introduction}
	
	This article is devoted to prove geometric properties for Hyodo-Kato cohomology in rigid geometry. We build geometric properties for Hyodo-Kato cohomology of dagger varieties, including the Mayer-Vietoris property, Poincar\'e duality and the Gysin isomorphism. Then by using geometric properties of Hyodo-Kato cohomology, we establish a comparison between algebraic and analytic Hyodo-Kato cohomology (GAGA) and the semistable conjecture for \'etale cohomology of almost proper rigid analytic varieties.
	
	Let $\mathscr{O}_K$ be a complete discrete valuation ting with fraction field $K$ of characteristic 0 and with perfect residue field $k$ of characteristic $p$. Let $W(k)$ be the ring of Witt vectors of $k$ with fraction field $F$ (therefore $W({k})=\mathscr{O}_F$). Let $\bar{K}$ be an algebraic closure of $K$ and $C$ be its $p$-adic completion, and let $\mathscr{O}_{\bar{K}}$ denote the integer closure of $\mathscr{O}_K$ in $\bar{K}$. Let $W(\bar{k})$ be the ring of Witt vectors of $k$ with fraction field $\breve{F}$ (therefore$\breve{F}$ is the $p$-adic completion of $F^{\operatorname{nr}}$, the maximal unramified extension of $K$ in $\bar{K}$ and $W(\bar{k})=\mathscr{O}_{\breve{F}}$) and let $\phi$ be the absolute Frobnius on $W(\bar{k})$. Set $\mathscr{G}=\operatorname{Gal}(\bar{K}/K)$.
	
	We will denote by $\mathcal{O}_K,\mathcal{O}_K^{\times},\mathcal{O}_K^0$,  depending on the context, the scheme $\Spec(\mathcal{O}_K)$ or the formal scheme $\Spf(\mathcal{O}_K)$ with the trivial log structure, the canonical (i.e., associated to the closed point) log structure, and the log structure induced by $\N \to \mathcal{O}_K, 1 \to 0.$
	
	\subsection{Semistable conjecture for open varieties}
	
	The Hyodo-Kato cohomology for algebraic varieties first appeared in Fontaine-Jannsen conjecture (\cite{jannsen1989cohomology} and \cite{fontaine1994semistable}), also known as the semistable conjecture. This conjecture suggests the existence of a ``new crystalline cohomology group" $H^i_{\rm HK}$, which should compare with \'etale cohomology for $X/\mathcal O_K$ proper scheme with semistable reductions. Moreover, this group should serve as a deformation of de Rham cohomology group, i.e. there should exist a Hyodo-Kato morphism $\iota_{\operatorname{HK}}:H^i_{\rm HK}(X_K)\to H^i_{\operatorname{dR}}(X_K)$ such that the base change $$\iota_{\operatorname{HK}}\otimes_FK:H^i_{\rm HK}(X_K)\otimes_FK\xrightarrow{\simeq} H^i_{\operatorname{dR}}(X_K)$$is an isomorphism. Moreover, the group $H^i_{\rm HK}(X_K)$ is endowed with a Frobenius $\phi$ and a monodromy operator $N$. The conjecture can be formulated as follows.
	
	\begin{conj}
		Let $X$ be a proper and smooth algebraic variety over $K$ admitting
		a semistable model over $\mathscr{O}_K$. Let $i\geq 0$. We have a functorial $\mathscr{G}_K$-equivariant ${ B}_{{\rm st}}$-linear isomorphism commuting with $\varphi$ and $N$
		$$H^i_{\operatorname{\acute{e}t}}(X_C,\Q_p)\otimes_{\Q_p}{{ B}_{{\rm st}}}\simeq H^i_{\operatorname{HK}}(X)\otimes_F{{B}_{{\rm st}}},$$
		compatible with the de Rham period morphism, via the natural injection
		${{ B}_{{\rm st}}}{\subset} {{ B}_{{\rm dR}}}$, and the Hyodo-Kato morphism $\iota_{\operatorname{HK}}:H^i_{\rm HK}\to H^i_{\operatorname{dR}}$.
	\end{conj}
	
	The Hyodo-Kato cohomology was defined in \cite{hyodo1991derhamwitt} for varieties with semistable reductions, by introducing a modified de Rham Witt complex. Later in \cite{hyodokato1994}, Osamu Hyodo and Kazuya Kato introduced log crystalline cohomology, and used it to give another construction for varieties with log-smooth models. The most remarkable results are in \cite{beilinson2013crys}, where Beilinson constructed the Hyodo-Kato cohomology for algebraic varieties by h-descent \cite{voevodsky1996h}, and proved the semistable conjecture without any smooth, proper or reduction hypothesis.
	
	For proper smooth rigid analytic varieties, the above conjecture (without semistable reduction hypothesis) also holds by \cite{CN5}. The proof relies on the geometrization of the syntomic cohomology and the theory of Banach-Colmez spaces.
	
	Like the algebraic case, it is natural to think what will happen for non proper or non smooth varieties. The situation is more complicated, especially if one wants to recover the \'etale cohomology. For pro-\'etale cohomology, various results have been obtained in \cite{CDNstein}, where Pierre Colmez, Gabriel Dospinescu, and Wiesława Nizioł calculated the $p$-adic \'{e}tale cohomology and pro-\'{e}tale cohomology for the $p$-adic Drinfeld half plane. In the pro-\'{e}tale case, they proved, in general, that (combined with the result in \cite{CN4.3}) for $r \geq 0$ and $X$ a smooth Stein space over $C$, we have the following commutative Galois equivariant diagram, which can be regarded as pro-\'etale version of semistable conjecture:
	$$\begin{tikzcd}
		0 \arrow[r] & \Omega^{r-1}(X)/\ker d \arrow[r] \arrow[d, equal] & H_{\mathrm{pro\acute{e}t}}^r\left(X, \mathbb{Q}_p(r)\right) \arrow[r] \arrow[d] & \left(H_{\mathrm{HK},\breve{F}}^r(X) \widehat{\otimes}_{\breve{F}} {{B}_{{\rm st}}}\right)^{N=0, \varphi=p^r} \arrow[r] \arrow[d] & 0  \\
		0 \arrow[r] & \Omega^{r-1}(X)/\ker d \arrow[r] & \Omega^{r}(X)^{d=0} \arrow[r] & H^r_{\operatorname{dR}}(X) \arrow[r] & 0.
	\end{tikzcd}$$
	In \cite{CN5}, there is also a semistable conjecture which generalizes the above result for arbitrary quasi-compact dagger varieties. However, this conjecture remains open.
	
	In this article, we consider the \'etale version of semistable conjecture for open varieties. We show the semistable conjecture is true for almost proper rigid analytic varieties.
	
	\begin{theorem} (Theorem \ref{ss})
		Suppose $X$ is a proper smooth rigid analytic variety over $C$, $Z \subset X$ is a strictly normal crossing divisor, and $U=X-Z$.
		
		(1) We have a $B_{\mathrm{st}}$-linear functorial isomorphism commuting with $\varphi$ and $N$
		$$\alpha^i_{\operatorname{st}}(U):H^i_{\operatorname{\acute{e}t}}(U,\Q_p)\otimes_{\Q_p}{B_{{\rm st}}}\simeq H^i_{\operatorname{HK}}(U)\otimes_{F^{\operatorname{nr}}}{B_{{\rm st}}},$$ that induces a $B_{\mathrm{dR}}$-linear filtered isomorphism $$H^i_{\et}(U,\Q_p) \otimes_{\Q_p}B_{\operatorname{dR}} \simeq H^i_{B_{\operatorname{dR}}^+}(X) \otimes_{B_{\operatorname{dR}}^+}B_{\operatorname{dR}}.$$ Here, $H^i_{B_{\operatorname{dR}}^+}(X)$ is the logarithmic $B_{\operatorname{dR}}^+$-cohomology introduced in \cite{xsperiod}, where the log structure of $X$ is induced by $Z$, and the filtration on $H^i_{B_{\operatorname{dR}}^+}(X)$ is defined by $$\fil^{\star}H^i_{B_{\operatorname{dR}}^+}(X):=\operatorname{Im}(H^i(\fil^{\star}\rg_{B_{\operatorname{dR}}^+}(X))\to H^i_{B_{\operatorname{dR}}^+}(X)).$$
		
		(2) Let $i \leq r$. Then we have an exact sequence
		\begin{equation} \label{1.1}
			0 \xrightarrow{}H^i_{\et}(U,\Q_p(r))\xrightarrow{}(H_{\mathrm{HK}}^i(U) {\otimes}_{F^{\operatorname{nr}}} {B^+_{{\rm st}}})^{N=0, \varphi=p^r}\xrightarrow{}H^i_{B_{\operatorname{dR}}^+}(X)/F^r\xrightarrow{}0.
		\end{equation}
		
		Moreover, when $X$ descends to a rigid analytic variety over $K$, statements in (1) and (2) are Galois equivariant.	
	\end{theorem}
	
	\begin{rem}
		(1) Unlike the pro-\'etale semistable conjecture, $H^i_{\operatorname{\acute{e}t}}(U,\Q_p)$ and $H^i_{\operatorname{HK}}(U)$ are finite dimensional vector spaces.
		
		(2) We have a natural isomorphism $H^i_{B_{\operatorname{dR}}^+}(X)\simeq H^i_{{\operatorname{dR}}}(U/B_{\operatorname{dR}}^+)$, where $H^i_{{\operatorname{dR}}}(U/B_{\operatorname{dR}}^+)$ is the $B_{\operatorname{dR}}^+$-cohomology introduced in \cite{CN4.3}. But the two cohomology groups have different filtrations.
		
		(3) when $X$ descends to a rigid analytic variety $X_0$ over $K$, we have a natural filtered isomorphism $$H^i_{B_{\operatorname{dR}}^+}(X)= H^i_{\operatorname{logdR}}(X_0)\otimes_KB_{\operatorname{dR}}^+.$$ 
		
		(4) The short exact sequence (\ref{1.1}) is used in \cite{ertl2024vpicardgroupsteinspaces} to study the image of Hodge-Tate log map.
		
		(5) In a follow-up article \cite{xslogsyn}, we will prove an \'etale version of semistable conjecture for quasi-compact log rigid analytic varieties by introducing logarithmic syntomic cohomology. In particular, after extending the cohomology groups to the category of Vector Space, we can give another (more conceptual) proof of the above theorem.
	\end{rem}
	
	We will discuss the proof of semistable conjecture for \'etale cohomology in Section \ref{cc}, aftering introducing the compactly supported Hyodo-Kato cohomology.

	\subsection{Algebraic and analytic Hyodo-Kato cohomology}
	
	Another foundational question on Hyodo-Kato cohomology is if we can compare the algebraic and analytic Hyodo-Kato cohomology, i.e. if we have GAGA for Hyodo-Kato cohomology. 
	
	The main difficulty there is the different constructions for algebraic and analytic Hyodo-Kato cohomology. Briefly, the approach of defining Hyodo-Kato cohomology for rigid analytic varieties in \cite{CN4.3} is similar to the algebraic one defined in \cite{beilinson2013crys}, by using the fact that for smooth rigid analytic varieties \'etale locally there exist semistable formal models \cite{temkin2017alteration}. However, unlike the algebraic case, in general there is no ``compatification" for rigid analytic varieties, and we need to define Hyodo-Kato cohomology (locally) without adding horizontal divisors. This is the main problem when comparing the analytic and algebraic Hyodo-Kato cohomology, even if we have the Hyodo-Kato isomorphism in the geometric case.
	
	Let us give an example to show the difference between the definition of algebraic and analytic Hyodo-Kato cohomology, that is, we want to compute the Hyodo-Kato cohomology for $\mathbb A^1_K$ and $\mathbb A^{1,\an}_K$. For the algebraic affine line, let $(\mathbb P^1_{\mathcal{O}_K},\mathcal M)$ be the log scheme $\mathbb P^1_{\mathcal{O}_K}$ with log structure induced by the closed immersion $\{\infty\} \cup \mathbb P^1_k \hookrightarrow \mathbb P^1_{\mathcal{O}_K}.$ Denote by $(\mathbb P^1_{k},\mathcal M')$ be the log scheme over $\Spec(k)$ induced by base change $\Spec(k)\hookrightarrow\Spec(\mathcal{O}_K)$. Then the algebraic Hyodo-Kato cohomology of $\mathbb A^1_K$ can be computed by $$\hk(\mathbb A^1_K)=\cris((\mathbb P^1_{k},\mathcal M')^0/\mathcal{O}_K^0)_{\Q_p},$$ where $\mathcal{O}_K^0$ is the log scheme $\mathcal{O}_K$ with log structure induced by $\N\to \mathcal{O}_K, 1\mapsto 0.$
	
	On the other hand, to compute $\hk(\mathbb A^{1,\an}_K)$, by definition one has to cover it by rigid generic fiber of semistable formal models, then one uses descent. For example, let $\mathfrak X$ be the standard semistable formal model of $\mathbb A^{1,\an}_K$, i.e., coming from the Bruhat-Tits building of rigid analytic projective line. Let $\mathfrak X_n \subset \mathfrak X$ be the induced semistable formal model of $X_n:=\Spf(K\langle \omega^{n}x_1,\cdots,\omega^{n}x_N\rangle)$, the closed disk of radius $\omega^{-n}$, where $\omega$ is an uniformizer of $K$. Then the Hyodo-Kato cohomology $\hk(\mathbb A^{1,\an}_K)$ is computed by $$\hk(\mathbb A^{1,\an}_K)=R\lim_n\cris(\mathcal X_{n,k}^0/\mathcal{O}_K^0)_{\Q_p}.$$ It seems we don't even have a direct way to construct a morphism $\hk(\mathbb A^1_K)\to \hk(\mathbb A^{1,\an}_K),$ although one can show that both sides should be vanishing for nonzero degrees in this case.
	
	In this article, by studying the geometric properties of algebraic and analytic Hyodo-Kato cohomology, we will show that we indeed have GAGA for Hyodo-Kato cohomology:
	
	\begin{thm}  (Theorem \ref{hkgaga})
		(1) Let $X$ be an algebraic variety over $K$, and $X^{\operatorname{an}}$ be its analytification. Then there exists a natural quasi-isomorphism in $D(\Mod^{\mathrm{solid}}_{F^{\mathrm nr}}):$
		$$\hk (X) \xrightarrow{\simeq} \hk(X^{\an}),$$
		which is compatible with Frobenius, monodromy, and the GAGA morphism for de Rham cohomology, i.e. we have the following commutative square:
		$$\begin{tikzcd}
			\hk(X) \arrow[r,"\iota_{\operatorname{HK}}"] \arrow[d,"\simeq"] & \dr(X) \arrow[d,"\simeq"] \\ 
			\hk(X^{\an}) \arrow[r,"\iota_{\operatorname{HK}}"] & \dr(X^{\an}),
		\end{tikzcd}$$
		where the horizon maps are Hyodo-Kato morphisms.		
		
		(2) Let $X$ be an algebraic variety over $\overline K$, and $X_C^{\operatorname{an}}$ be the analytification of $X_C:=X \times_{\overline K}C$. Then there exist a natural quasi-isomorphism in $D(\Mod^{\mathrm{solid}}_{F^{\operatorname{nr}}})$:
		$$\hk (X) \xrightarrow{\simeq} \hk(X_C^{\an}),$$
		which is compatible with Galois action, Frobenius, monodromy, and the GAGA morphism for de Rham cohomology.	
	\end{thm}
	
	\begin{rem}
		This theorem is also used in \cite{ertl2024vpicardgroupsteinspaces} to study the image of Hodge-Tate log map.
	\end{rem}

	\subsection{Cohomology with compact support} 
	
	To construct the comparison map between algebraic and analytic Hyodo-Kato cohomology, firstly we note that, if $\mathcal X$ is a projective semistable scheme over $\mathcal O_K,$ then the local-global compatibility assures a natural equality $$\hk(\mathcal X_K)=\hk(\mathcal X^{\an}_K)\simeq\cris(\mathcal X_{k}^0/\mathcal{O}_K^0).$$ We would like to reduce the proof of GAGA to the projective case (with a semistable model). As the example described in last section, in general one needs to work on the a strictly semistable pair. Briefly, a semistable pair $(U,\overline U)$ over a field $K$ is an open embedding $j:U\hookrightarrow \overline U$ with dense image of a $K$-variety $U$ into a proper flat regular $\mathcal{O}_K$-scheme $\overline U$, such that $\overline U-U$  is a regular divisor with normal crossings on $\overline U$, the irreducible components of $\overline U-U$ is regular, and the closed fiber $\overline U_k$ is reduced. Let $D=\overline U_K-U.$ If we have the compatible Gysin sequences $$\hk(D)\{-1\}[-2]\simeq [\hk(\overline U_K) \to \hk(U)],$$$$\hk(D^{\an})\{-1\}[-2]\simeq [\hk(\overline U_K^{\an}) \to \hk(U^{\an})]$$ where $D$ is smooth, then since $\overline U_K$ and $D$ are proper, we should be able to have the GAGA comparison for $U$, and therefore the general case.
	
	The algebraic Gysin sequence is true, since in \cite{nizioldeglise2018syntomic} they construct the realization functor for (geometric) Hyodo-Kato cohomology. However, it is still unclear how to construct the Gysin morphism directly in rigid geometry. On the other hand, the Gysin morphism can be deduced from the Poincar\'e duality and open-closed exact sequence of compactly supported Hyodo-Kato cohomology. We will review the definition of compactly supported Hyodo-Kato cohomology, which is developed in \cite{AGNcompact}.
	
	Therefore, the proof of GAGA heavily relies on the following propositions.	
	
	\begin{prop}[open-closed exact sequence] (Proposition \ref{open-closed})
		If $X$ is a smooth dagger variety over $L=K$ or $C$, $Z \subset X$ is a smooth divisor, $U=X-Z.$ Then we have
		$$\hkc(U)\simeq  [\hk(X) \xrightarrow{} \hk(Z)].$$
	\end{prop}
	
	\begin{theorem} (\cite[Theorem 5.30]{AGNcompact} or Theorem \ref{hkrigidpoincare})
		Let $Y$ be a partially proper smooth rigid analytic varieties or a quasi-compact smooth dagger variety over $C$ of dimension $d$. Then:
		
		(i) There is a natural trace map in $D(\Mod_{F^{\operatorname{nr}}}^{\operatorname{solid}}):$ $${\operatorname{tr_{HK}}}:\hkc(Y)[2d] \to F^{\operatorname{nr}}\{-d\},$$ compatible with the Hyodo-Kato morphism.
		
		(ii) The pairing $$\hk(Y) \otimes^{L_{\blacksquare}}\hkc(Y)[2d] \xrightarrow{} \hkc(Y)[2d]\to F^{\operatorname{nr}}\{-d\}$$ is a perfect duality in $D(\Mod_{F^{\operatorname{nr}}}^{\operatorname{solid}})$, i.e. we have the induced quasi-isomorphism in $D(\Mod_{F^{\operatorname{nr}}}^{\operatorname{solid}}):$ $$\hk(Y) \simeq R\operatorname{Hom}_{F^{\operatorname{nr}}}(\hkc(Y),F^{\operatorname{nr}}\{-d\}).$$
	\end{theorem}
	
	\subsection{The proof of semistable conjecture} \label{cc}
	
	 In this article, similar to the proof of GAGA, our strategy is to reduce the proof of semistable conjecture for almost proper rigid analytic varieties to the proper ones. The proof includes the following steps:
	 
	 \begin{itemize}
	 	\item Construction of the period morphism: in \cite[Theorem 6.9]{CN4.3}, for $r \geq 0,$ Pierre Colmez and Wiesława Nizioł construct a natural quasi-isomorphism $$\tau^{\leq r}\rg_{\mathrm{syn}}(U,\Q_p(r)) \xrightarrow{\simeq} \tau^{\leq r}\rg_{\proet}(U,\Q_p(r))$$ for any partially proper rigid analytic variety, and in our case the syntomic cohomology $\rg_{\mathrm{syn}}(U,\Q_p(r))$ can be described as 
	 	\begin{equation} \label{syn}
	 		\rg_{\mathrm{syn}}(U,\Q_p(r)):=[[\hk(U)\otimes_{F^{\mathrm{nr}}}B_{\mathrm{st}}]^{N=0, \varphi=p^r}\to \rg_{{\mathrm{dR}}}(U/B_{\mathrm{dR}}^+)/F^r],
	 	\end{equation}
	 	which induces a map for $r>2d$:
	 	\begin{align*}
	 		\tau^{\leq 2d}\rg_{\proet}(U,\Q_p(r)) \xrightarrow{\simeq} \tau^{\leq 2d}\rg_{\operatorname{syn}}(U,r) &\to [\hk(U) {\otimes}_{F^{\operatorname{nr}}} {B_{{\rm st}}}]^{N=0, \varphi=p^r} \\ & \xrightarrow{p^{-r}} \hk(U) {\otimes}_{F^{\operatorname{nr}}} {B_{{\rm st}}}.
	 	\end{align*}
	 	After imposing $B_{{\rm st}}$-linearity and twisting, this allows us to define the period morphism $$\alpha_{\operatorname{st}}(U):\rg_{\operatorname{\acute{e}t}}(U,\Q_p)\otimes_{\Q_p}{B_{{\rm st}}}\to \hk(U)\otimes_{F^{\operatorname{nr}}}{B_{{\rm st}}},$$ by simply applying the natural morphism $\rg_{\operatorname{\acute{e}t}}(U,\Q_p) \to \rg_{\proet}(U,\Q_p)$. This construction might seem strange, as $\rg_{{\et}}(U_C,\Q_p)$ and $\rg_{{\proet}}(U_C,\Q_p)$ are generally quite different. For example, when $U=\mathbb A^1,$ the analytic affine line, $H^1_{\et}(A^1_C,\Q_p)=0$ but $H^1_{\proet}(A^1_C,\Q_p)$ is very large, as indicated by the main theorem of \cite{CDNstein}. Nevertheless, this construction yields precisely the period morphism required. The same method also applies to define $\alpha_{\operatorname{st}}(X)$: in this case where $X$ is proper smooth, $\rg_{\operatorname{\acute{e}t}}(X,\Q_p) = \rg_{\proet}(X,\Q_p)$, and Pierre Colmez and Wiesława Nizioł prove that $\alpha_{\operatorname{st}}(X)$ is a quasi-isomorphism in \cite{CN5}, by using the theory of Banach-Colmez spaces.
	 	\item Reducing to the proper case: we can reduce to the case where $Z$ is smooth, and we check that $\alpha_{\operatorname{st}}$ is compatible with the Gysin sequence. Then we can apply the Gysin sequence to show that $\alpha_{\operatorname{st}}(U)$ is a natural quasi-isomorphism. The key point is that our construction for $\alpha_{\operatorname{st}}$ is natural with respect to $U \to X$, so it suffices to check that $\alpha_{\operatorname{st}}$ is compatible with the Gysin morphism, that the following square 
	 	$$\begin{tikzcd}
	 		\rg_{\et} (Z,\Q_p(-1))[-2 ]\otimes_{\Q_p}B_{\operatorname{st}} \arrow[r,"g_{\et}"] \arrow[d,"\alpha_{\operatorname{st}}(Z)(-1){[-2]}"] & \rg_{\et}(X,\Q_p)\otimes_{\Q_p}B_{\operatorname{st}}  \arrow[d,"\alpha_{\operatorname{st}}(X)"]  \\
	 		\hk (Z)\{-1\}[-2] \otimes_{F^{\operatorname{nr}}}B_{\operatorname{st}} \arrow[r,"g_{\operatorname{HK}}"] & \hk (X) \otimes_{F^{\operatorname{nr}}}B_{\operatorname{st}} ,
	 	\end{tikzcd}$$
	 	is commutative, where the horizontal maps are the Gysin morphisms. However, the Gysin morphisms are deduced from the Poincar\'e duality and open-closed exact sequence, and we can reduce to checking that $\alpha_{\operatorname{st}}$ is compatible with the Poincar\'e duality, that is to say, checking that $\alpha_{\operatorname{st}}$ is compatible with the pairing and the trace map, but both are automatic due to our construction of Poincar\'e duality for pro-\'etale cohomology (and hence for étale cohomology, since $X$ and $Z$ are proper smooth): in fact, the pairing and the trace map for pro-\'etale cohomology is induced from the one for Hyodo-Kato cohomology, by using the distinguished triangle (\ref{syn}).
	 	\item Construction of the short exact sequence (\ref{1.1}): we construct (\ref{1.1}) by using \cite[Remark 5.16]{CN5}, which is an extension of the standard result for admissible filtered $(\varphi,N)$-modules: once we are able to show that $\alpha_{\operatorname{st}}(U)$ is a natural quasi-isomorphism, we can show that for $i\leq r$, we have that $$(D,D^+_{\operatorname{dR}}):=(H^i_{\operatorname{HK}}(U),H^i_{\et}(U,\Q_p)\otimes_{\Q_p}B_{\operatorname{dR}}^+)$$ is a weakly admissible filtered $(\varphi,N)$-module over $C$ (in the sense of \cite[Definition 5.3.1]{CN5}) with $\varphi$-slopes in $[0,r]$. The key observation here is that we have $$(D \otimes_{F^{\operatorname{nr}}}B^+_{\operatorname{dR}})/t^rD^+_{\operatorname{dR}} \simeq H^i_{\operatorname{dR}}(U/B^+_{\operatorname{dR}})/t^rD^+_{\operatorname{dR}}\simeq H^i_{B_{\operatorname{dR}}^+}(X)/F^r,$$ where the first isomorphism is the geometric Hyodo-Kato isomorphism \cite[Theorem 4.6]{CN4.3}, and the second isomorphism follows from \cite[Proposition 5.17]{xsperiod} and \cite[Theorem 1.4]{xsperiod}. Then \cite[Remark 5.16]{CN5} gives the desired short exact sequence (\ref{1.1}).
	 	\item Checking that our construction of (\ref{1.1}) is Galois equivariant: however, the previous step intertwines different constructions of period morphisms, which could pose a problem when considering the Galois action when $X$ descends to a rigid analytic variety over $K$. To resolve this issue, we need to show that our period map $\alpha_{\operatorname{st}}(U)$, after tensoring with $B_{\mathrm{dR}}$, agrees with the de Rham period map of \cite[Theorem 1.4]{xsperiod}. Once again, we can reduce to $Z$ smooth, and by using the Gysin sequence, we are reduced to checking that the maps $\alpha_{\operatorname{st}}(X)\otimes_{B_{\operatorname{st}}}B_{\operatorname{dR}}$ and $\alpha_{\operatorname{st}}(Z)\otimes_{B_{\operatorname{st}}}B_{\operatorname{dR}}$ agree with the de Rham period morphism for proper smooth rigid analytic varieties constructed by Peter Scholze in \cite{scholze2013p}, but this compatibility has been established by Sally Gilles in  \cite{gilles2023period}.
	 \end{itemize}
	 
	 We refer the reader to Theorem \ref{ss} for a detailed proof.
	
	\subsection{Structure of the article}
	
	In Section 2, we review the h-topology from \cite{beilinson2012derham} and the \'eh-topology from \cite{guo2019hodgetate}, along with the (log-)de Rham comparison theorem. These tools are fundamental for defining Hyodo-Kato cohomology and proving the semistable conjecture.
	
	In Section 3, we present some key aspects of both algebraic and analytic Hyodo-Kato cohomology, including their compactly supported versions. We define the algebraic Hyodo-Kato cohomology with compact support, which agrees with the construction via the Hyodo-Kato realization functor from \cite{nizioldeglise2018syntomic}. For the analytic case, we review the compactly supported analytic Hyodo-Kato cohomology as defined in \cite{AGNcompact}.
	
	Section 4 focuses on the geometric properties of compactly supported Hyodo-Kato cohomology for analytic varieties. We establish the Mayer-Vietoris properties for Hyodo-Kato cohomology (and related theories) in the context of abstract blow-up squares. In the absence of six-functor formalism within the rigid analytic framework (which is still under development), we demonstrate Poincaré duality and the open-closed exact sequence for compactly supported Hyodo-Kato cohomology as defined here. This will also yield the Gysin sequence.
	
	In Section 5, we apply the geometric properties developed in Section 4 to establish key results in Hyodo-Kato cohomology. As outlined in the introduction, we compare algebraic and analytic Hyodo-Kato cohomology and establish the semistable conjecture for almost proper rigid analytic varieties. We will also examine the construction of log-crystalline cohomology with compact support by Tsuji in \cite{tsuji1999poincare}, showing that our definition of compactly supported Hyodo-Kato cohomology is consistent with the one developed in this paper.
	
	Finally, in the appendix, we verify that the overconvergent de Rham cohomology defined for (non-smooth) dagger varieties of \cite{bosco2023rational} agrees with the definition in \cite{grosse2004derham}. This result is used in Section 4 and enables us to confirm the finiteness of overconvergent Hyodo-Kato cohomology in broader cases.
	
	\subsection*{Acknowledgments}
	
	The author would like to express his sincere gratitude to Wiesława Nizioł for suggesting the problem, for many valuable discussions, and for carefully reading this paper. He is deeply indebted to her for identifying subtle errors and for strengthening the overall rigor of the manuscript. The author also thanks Yicheng Zhou and Zhenghui Li for helpful discussions. This article is part of the author’s Ph.D. thesis.

	\subsection*{Notation and conventions}
	
	In this article, we use the language of adic spaces developed in \cite{huber2013etale}. A rigid analytic variety is a quasi-separated adic space locally of finite type over $\Spa(L,\mathcal O_L)$ for a $p$-adic field $L$. All rigid analytic spaces considered will be over $K$ or $C$. We will also use the notion of dagger varieties \cite{gross2000overconvergent}. We assume all rigid analytic varieties and dagger varieties are separated, taut, and countable at infinity. We denote $\operatorname{Rig}_L$ (or $\operatorname{Rig}^{\dagger}_L$) the category of rigid analytic varieties (or dagger varieties) over $L$, and we denote $\operatorname{Sm}_L$ (or $\operatorname{Sm}^{\dagger}_L$) the category of smooth rigid analytic varieties (or smooth dagger varieties) over $L$.
	
	If $\mathcal A$ is an abelian category, unless otherwise stated, we always work with derived stable $\infty$-category $D(\mathcal{A})$.
	
	We will freely use the language of condensed mathematics developed in \cite{cscondensed}. In fact, most of the cohomology groups appeared in this article is finite dimension vector space, so this will not cause much trouble.
	
	We will also use the theory of log adic space. For definitions and properties of log adic space, see \cite{dllz2023logadic}. A (pre)log-structure on a condensed ring is simply a (pre)log-structure on the underlying ring.

	\section{Preliminaries}
	
	In this section, we review some preliminaries which are essential to define Hyodo-Kato cohomology and prove the semistable conjecture for \'etale cohomology of almost proper rigid analytic varieties.
	
	\subsection{The \'eh-topology and h-topology} \label{model}
	
	The \'eh-topology is used to define the Hyodo-Kato cohomology for singular rigid analytic varieties, which has a basis of semistable formal models by $\eta$-\'etale descent. On the algebraic side, there is a similar definition of \'eh-topology, but if we want to get a basis of semistable models in the algebraic setup, we need to refine the \'eh-topology to h-topology. We will recall the definition of the \'eh-topology for singular analytic varieties and h-topology for algebraic varieties. We begin with a generalization of Verdier’s criterion.
	
	\subsubsection{Beilinson basis} In \cite{beilinson2012derham}, Beilinson extended Verdier’s well-known criterion \cite[4.1]{SGA4I}, which provides conditions for changing sites while preserving their associated topoi. We will review the proposition quickly.
	
	Let $\mathcal C$ be a essentially small site, we denote $\operatorname{Sh}(\mathcal C)$ the category of sheaves of sets on $\mathcal C.$ The (Beilinson) basis is defined as follows.
	
	\begin{defn}
		A basis of $\mathcal C$ is a pair $(\mathcal B, \phi)$, where $\mathcal B$ is a a essentially small subcategory of $\mathcal C$ and $\phi:\mathcal B \to \mathcal C$ is a faithful functor such that for  $C\in\mathcal C$ and a finite family of pairs $(B_{\alpha},f_{\alpha}), B_{\alpha}\in \mathcal B, f_{\alpha}: V\to F(B_{\alpha}),\alpha \in I$ there exists a set of objects $B^{\prime}_{\beta}\in \mathcal B$ and a covering family $\{\phi(B^{\prime}_{\beta})\to C\}$ such that each composition $\phi(B^{\prime}_{\beta})\to C\to \phi(B_{\alpha})$ lies in $\operatorname{Hom}(B_{\beta}^{\prime},B_{\alpha})\subset \operatorname{Hom}(\phi(B^{\prime}_{\beta}),\phi(B_{\alpha}))$.
	\end{defn}	
	
	\begin{rem}
		The Verdier’s criterion requires that $\phi$ is fully faithful, where here we only require $\phi$ to be faithful. This is useful when we consider a basis of semistable models (where $\phi$ is the (rigid) generic fiber).
	\end{rem}
	
	Define a covering sieve in $\mathcal B$ as a sieve whose image by $\phi$ is a covering sieve in $\mathcal C$. We have the following theorem by \cite{beilinson2012derham}.
	
	\begin{thm}
		If $(\mathcal B, \phi)$ is a basis of $\mathcal C$, then
		
		(i) Covering sieves in $\mathcal B$ forms a Grothendieck topology in $\mathcal B.$
		
		(ii) The functor $\phi:\mathcal B \to \mathcal C$ is continuous.
		
		(iii) $\phi$ induces an equivalence of topoi $\operatorname{Sh}(\mathcal{B}) \xrightarrow{\simeq}\operatorname{Sh}(\mathcal{C}).$
	\end{thm}
	
	\subsubsection{h-topology for algebraic varieties} We begin with the definition of h-topology of \cite{voevodsky1996h}.
	
	\begin{defn}
		For a field $L$, the h-topology on $\operatorname{Var}_K$ is generated by the pretopology whose coverings are finite families of maps $\{Y_i \to X\}_{i\in I}$, such that $Y:=\coprod_{i \in I}Y_i\to X$ is a universal topological epimorphism, i.e. a subset $U$ of $X$ is Zariski open if and only the preimage of $U$ in $Y$ is open, and the same is true for any base change over $X$.
	\end{defn}
	
	\begin{rem}
		The h-topology is stronger than the proper topology and \'etale topology, but it is weaker than the v-topology.
	\end{rem}
	
	A pair $(U,\overline U)$ over a field $K$ is an open embedding $j:U\hookrightarrow \overline U$ with dense image of a $K$-variety $U$ into a reduced proper flat $\mathcal{O}_K$-scheme $\overline U$. For a field $K$, We denote by $\operatorname{Var}^{\operatorname{ss}}_{K}$ the category of strictly semistable pairs, i.e. pairs $(U,\overline U)$ such that $\overline U$ is regular, $\overline U-U$  is a divisor with normal crossings on $\overline U$, the irreducible components of $\overline U-U$ is regular, and the special fiber $\overline U_k$ is reduced. 
	
	In the geometric setup, we denote by $\operatorname{Var}^{\operatorname{ss},b}_{\overline{K}}$ the category of basic semistable $\mathcal{O}_{\overline{K}}$-pairs, i.e. for a semistable $\mathcal{O}_{\overline{K}}$-pairs $(U,\overline U)$, there exists a semistable scheme $(U',\overline U')$ over $\mathcal{O}_E$, where $E$ is a finite field extension of $K$, such that $(U,\overline U)$ is isomorphism to the base change $(U_{\overline{K}},\overline U_{\mathcal{O}_{\overline{K}}}).$ 
	
	By \cite[Proposition 2.5]{beilinson2012derham}, we have
	
	\begin{prop} \label{modelalgebraic}
		(i) $\operatorname{Var}^{\operatorname{ss}}_{K}$ forms a basis of $\operatorname{Var}_{K,\operatorname{h}}.$
		
		(ii) $\operatorname{Var}^{\operatorname{ss},b}_{\overline{K}}$ or $\operatorname{Var}^{\operatorname{ss}}_{\overline{K}}$ forms a basis of $\operatorname{Var}_{\overline{K},\operatorname{h}}.$
	\end{prop}
	
	We will often use the following observation.
	
	\begin{lem} \label{nclemma}
		Suppose $(U,\overline U)$ is a strictly semistable pair, let $Z:=\overline U_{\eta}-U.$ If $Z$ is smooth, denote by $\overline{Z}$ the closure of $Z$ in $\overline{U}$, then $(Z,\overline{Z})$ is also in $\operatorname{Var}^{\operatorname{ss}}_K$.
	\end{lem}
	
	\begin{proof}
		In fact, since $(U, \overline{U})$ is a strictly semistable pair, we can write $\overline{U}-U$ as the union of irreducible components $\overline{Z} \cup D_1 ... \cup D_m$, and for $D_J=\bigcap_{j \in J}D_j$ where $J \subset \{1,2,...,m\}$ is a finite set, $\overline{Z} \cap D_J$ is a regular scheme each of whose irreducible components has codimension $|J|$ in $\overline{Z}$. Since we have assumed $Z$ is smooth and $\overline{U}$ is proper, $D_i$ are in the special fiber of $\overline{U}$ for all $i$, $\overline{Z}-Z=\bigcup_{i}(Z\cap D_i)$ and each $Z\cap D_j$ is smooth over the special fiber, hence by further writing $Z\cap D_i$ as disjoint union of smooth irreducible components $\cup_k D_{ik}$,  for $D_J'=\bigcap_{jl \in J'}D_{jl}$ where indexing $J'$ is a finite set, $\cap D_J'$ is a regular scheme each of whose irreducible components has codimension $|J'|$ in $\overline{Z}$, hence $\overline{Z}-Z$ is strictly normal crossing divisor in $\overline{Z}$.
	\end{proof}
	
	\subsubsection{\'eh-topology for rigid analytic varieties}
	
	We will briefly summarize \cite[2]{guo2019hodgetate}.
	
	\begin{defn}
		For a field $L$, the \'eh-topology on $\operatorname{Rig}_L$ is generated by the pretopology whose coverings can be refined by \'etale coverings, universal homeomorphisms, and morphisms $$\operatorname{Bl}_Z(X)\sqcup Z \to X,$$ where $Z$ is a closed analytic subset of $X$.
	\end{defn}

	For a field $K$, We denote by $\mathscr{M}^{\operatorname{ss}}_K$ the category of semistable formal $\mathcal{O}_K$-models, i.e. it is semistable over $\mathcal{O}_E$ for a finite field extension $E$ of $K$. We also denote by $\mathscr{M}^{\operatorname{ss},b}_C$ the category of semistable formal $\mathcal{O}_C$-models, i.e. a formal $\mathcal{O}_C$-model $\mathscr{X}$ such that there exists a semistable model $\mathscr{X}'$ over $\mathcal{O}_E$, where $E$ is a finite field extension of $K$, such that $\mathscr{X}$ is isomorphism to the base change $\mathscr{X}'_{\mathcal{O}_C}.$ By \cite[Proposition 2.8]{CNderham} and resolution of singularities, we have the following proposition.
	
	\begin{prop} \label{modelrigid}
		(i) $\operatorname{Sm}_K$ or $\mathscr{M}^{\operatorname{ss}}_K$ forms a basis of $\operatorname{Rig}_{K,\operatorname{\acute{e}h}}.$
		
		(ii) Over $C$, $\operatorname{Sm}_C$ or $\mathscr{M}^{\operatorname{ss},b}_C$ or $\mathscr{M}^{\operatorname{ss}}_K$ forms a basis of $\operatorname{Rig}_{C,\operatorname{\acute{e}h}}.$
	\end{prop}
	
	We have analogy results for dagger varieties, We denote by $\mathscr{M}^{\operatorname{wss}}_L$ the category of weakly semistable formal $\mathcal{O}_L$-models. We also denote by $\mathscr{M}^{\operatorname{wss},b}_C$ the category of weakly basic semistable formal $\mathcal{O}_C$-models. By \cite[Proposition 2.14]{bosco2023rational}, we have the following proposition.
	
	\begin{prop} \label{modeldagger}
		(i) $\operatorname{Sm}^{\dagger}_K$ or $\mathscr{M}^{\operatorname{wss}}_K$ forms a basis of $\operatorname{Rig}^{\dagger}_{K,\operatorname{\acute{e}h}}.$
		
		(ii) Over $C$, $\operatorname{Sm}^{\dagger}_C$ or $\mathscr{M}^{\operatorname{wss},b}_C$ or $\mathscr{M}^{\operatorname{wss}}_K$ forms a basis of $\operatorname{Rig}^{\dagger}_{C,\operatorname{\acute{e}h}}.$
	\end{prop}
	
	\subsection{de Rham cohomology and $B^+_{\operatorname{dR}}$-cohomology} 
	
	For $i\geq 0$, let $\Omega^i_{\operatorname{\acute{e}h}}$ be the \'eh-sheafification of the presheaf $X\mapsto \Omega^i(X)$ for $X\in \operatorname{Sm}_L$ or $\operatorname{Sm}^{\dagger}_L$. We have the following \'eh-descent theorem, by \cite[Theorem 4.0.2]{guo2019hodgetate}.
	
	\begin{thm} \label{ehdescentrigid}
		If $X\in \operatorname{Sm}_L$ or $\operatorname{Sm}^{\dagger}_L$, then for any $i \in\N$, we have $$R\pi_{X,*}\Omega^i_{\operatorname{\acute{e}h}}=\Omega^i_{X/L}.$$
	\end{thm}
	
	With the \'eh-topology, we can define the de Rham cohomology for rigid analytic varieties as follows.
	
	\begin{defn}
		Let $X$ be a rigid analytic variety or a dagger variety over $L$, the de Rham cohomology of $X$ is defined to be $$\dr(X):=\rg(X,\Omega^{\bullet}_{\operatorname{\acute{e}h}}).$$
	\end{defn}
	
	For $X$ defined over $C$, to have a meaningful de Rham-\'etale comparison theorem, it is necessary to introduce the $B^+_{\operatorname{dR}}$-cohomology, which is a deformation of de Rham cohomology, i.e. we have $$H^i_{\operatorname{dR}}(X/B_{\operatorname{dR}}^+)\otimes_{B_{\operatorname{dR}}^+}C\simeq H^i_{\operatorname{dR}}(X).$$ The $B^+_{\operatorname{dR}}$-cohomology has various constructions, we refer the reader to \cite{bms1}, \cite{guo2019hodgetate}, \cite{CN4.3} and \cite{bosco2023rational} for different constructions.
	
	The following theorem is due to \cite[Theorem 1.1.4]{guo2019hodgetate} and \cite[Theorem 7.4]{bosco2023rational}.
	
	\begin{thm} \label{ehcomparison}
		Let $X$ be a proper rigid-analytic variety, for each $i \geq 0$, 
		
		(i) If $X$ is defined over $K$, we have a natural isomorphism $$H^i_{\et}(X,\Q_p) \otimes_{\Q_p}B_{\operatorname{dR}} \simeq H^i_{\operatorname{dR}}(X) \otimes_{K}B_{\operatorname{dR}},$$ compatible with filtrations.
		
		(i) If $X$ is defined over $C$, we have a natural isomorphism $$H^i_{\et}(X,\Q_p) \otimes_{\Q_p}B_{\operatorname{dR}} \simeq H^i_{\operatorname{dR}}(X/B_{\operatorname{dR}}^+) \otimes_{B_{\operatorname{dR}}^+}B_{\operatorname{dR}},$$ compatible with filtrations. When $X$ can be descent to $X_0$ over $K$, this isomorphism agrees with the comparison isomorphism in (i), under a canonical identification $$H^i_{\operatorname{dR}}(X/B_{\operatorname{dR}}^+)= H^i_{\operatorname{dR}}(X_0)\otimes_KB_{\operatorname{dR}}^+.$$
	\end{thm}
	
	There is a logarithmic analogy of the above theorem for proper smooth rigid analytic varieties as follows, which is a generalization of \cite[Theorem 1.7]{bms1}.
	
	\begin{theorem} \cite{xsperiod} \label{logcdr}
		Let $X$ is a proper smooth rigid analytic varieties over $C$,  $D \subset X$ is a strictly normal crossing divisor, and denote by $U=X-D.$ We endow $X$ with the logarithmic structure induced by the divisor $D$. Then there are cohomology groups $H^i_{B_{\operatorname{dR}}^+}(X)$ which come with a canonical filtered isomorphism $$H^i_{\ket}(X,\Q_p) \otimes_{\Q_p}B_{\operatorname{dR}} \simeq H^i_{B_{\operatorname{dR}}^+}(X) \otimes_{B_{\operatorname{dR}}^+}B_{\operatorname{dR}}.$$ When $X$ and $Z$ come from $X_0$ and $Z_0$ over a discrete valued field $K$, this isomorphism agrees with the comparison isomorphism (see \cite[Theorem 1.1]{dllz2023logrh}) $$H^i_{\ket}(X,\Q_p) \otimes_{\Q_p}B_{\operatorname{dR}} \simeq H^i_{\operatorname{logdR}}(X_0) \otimes_{K}B_{\operatorname{dR}},$$under a canonical identification $$H^i_{B_{\operatorname{dR}}^+}(X)= H^i_{\operatorname{logdR}}(X_0)\otimes_KB_{\operatorname{dR}}^+.$$ Moreover, if $X$ comes from $X_0$ over a discrete valued field $K$ or admitting a proper semistable formal model over $\operatorname{Spf}(\mathcal{O}_C)$, $H^i_{B_{\operatorname{dR}}^+}(X)$ is a finite free $B_{\operatorname{dR}}^+$-module, and we have
		
		(i) The Hodge–de Rham spectral sequence $$E^{ij}_1:=H^j(X_{\operatorname{log}},\Omega^{\log,i}_{X})\Rightarrow H^{i+j}_{\operatorname{logdR}}(X)$$ degenerates at $E_1$.
		
		(ii) The Hodge–Tate spectral sequence
		$$E^{ij}_2:=H^j(X_{\operatorname{log}},\Omega^{\log,i}_{X})(-j)\Rightarrow H^{i+j}_{\ket}(X,\Q_p)\otimes_{\Q_p}C$$ degenerates at $E_2$.
	\end{theorem}	
	
	\subsection{Cohomology for dagger varieties}
	
	Following \cite{CNderham} and \cite{bosco2023rational}, we briefly review a construction that allows us to canonically define a cohomology theory on $\mathrm{Rig}_L^\dagger$ from a cohomology theory on $\mathrm{Rig}_L$.
	
	Let $\mathcal D$ be a presentable $\infty$-category. The continuous functor $$l:\mathrm{Rig}_{L,\eh}^\dagger \to \mathrm{Rig}_{L,\eh}: X\mapsto \widehat X$$ given by sending a dagger variety $X$ to its completion $\widehat X$ induces an adjunction $$l_*:\Shv^{\mathrm{hyp}}(\mathrm{Rig}_{L, \eh}^\dagger, \mathcal D)\rightleftarrows \Shv^{\mathrm{hyp}}(\mathrm{Rig}_{L, \eh}, \mathcal D):l^{* \mathrm{hyp}},$$ where $l^{* \mathrm{hyp}}$ is the hypercompletion of $l^*.$ For $\f\in \Shv^{\mathrm{hyp}}(\mathrm{Rig}_{L, \eh}, \mathcal D)$, denote by $$\f^\dagger:=l^{* \mathrm{hyp}}\f\in \Shv^{\mathrm{hyp}}(\mathrm{Rig}_{L, \eh}^\dagger, \mathcal D).$$
	
	Therefore, in the following context, once we introduce a cohomology theory (without compact support) for rigid analytic varieties (e.g., de Rham, Hyodo-Kato, \'etale, pro-\'etale, syntomic), we will automatically apply the above construction to get the associated definition for dagger varieties (by taking $\mathcal D=D(\Mod^{\mathrm{solid}}_\Q)$).

	We also note that the \'etale cohomology for a dagger variety $X$ is the same as the \'etale cohomology for $\widehat X,$ which is implied by \cite[Corollary A.28]{vezzani2018monskywashnitzer}.
	
	\section{Hyodo-Kato cohomology (with compact support)}
	
	
	
	In this section we review the definition of Hyodo-Kato cohomology for algebraic varieties of \cite{beilinson2013crys}, and Hyodo-Kato cohomology for rigid analytic varieties in \cite{CN4.3}. We will also define the Hyodo-Kato cohomology with compact support, and we will show that our definition agrees with Tsuji's log-crystalline cohomology with compact support in \cite{tsuji1999poincare} in the following chapters.
	
	\subsection{Hyodo-Kato cohomology for algebraic varieties}
	
	We review the Hyodo-Kato Cohomology (with compact support) for algebraic varieties. In particular, thanks to the Hyodo-Kato realization functor, we have the Poincar\'e duality for Hyodo-Kato Cohomology.
	
	\subsubsection{Definition and properties} We begin with recalling the definition and basic properties of algebraic Hyodo-Kato Cohomology of \cite{beilinson2013crys} and \cite{nizioldeglise2018syntomic}.
	
	(1) Arithmetic setup: We define here Hyodo-Kato cohomology of algebraic varieties over $K$. Suppose $(U, \overline{U})\in \operatorname{Var}^{\operatorname{ss}}_K$ is a strict semistable pair, which is defined in Definition \ref{modelalgebraic}. Let $\mathscr A_{\operatorname{HK}}$ be the h-sheafification of the presheaf $$(U, \overline{U})\mapsto \hk((U, \overline{U})_0):=\rg_{\operatorname{cris}}((U, \overline{U})_0/\mathcal O^0_{F_L})_{\Q_p}, $$ where we endow $(U, \overline{U})$ with the log structure defined by the compactifing log structure of the open immersion $U\hookrightarrow \overline U.$ The sheaf $\mathscr A_{\operatorname{HK}}$ is a h-sheaf of dg solid $F$-algebra on $\operatorname{Alg}_K$ equipped with a Frobenius operator $\phi$ and a monodromy operator $N$ such that $N\phi=p\phi N$. For an algebraic variety $X \in \operatorname{Alg}_K,$ we define the Hyodo-Kato cohomology of $X$ by $\hk(X):=\rg(X,\mathscr A_{\operatorname{HK}}).$ By \cite{janniziol2016hk}, there exist a Hyodo-Kato morphism $$\iota_{\operatorname{HK}}:\hk(X)\to \dr(X),$$ which is constructed locally for a semistable pair and then descend.
	
	(2) Geometric setup: Suppose now $(U, \overline{U})\in \operatorname{Var}^{\operatorname{b,ss}}_{\overline K}$ be a basic semistable model. Endow $(U, \overline{U})$ with the log structure as above, then by definition, $f:(U, \overline{U}) \to \Spec (\mathcal O_{\overline K})^{\times}$ is the base change of a semistable pair $\theta_L:(U_L, \overline{U_{\mathcal O_L}})\to \Spec (\mathcal O_{L})^{\times}$ for a finite field extension $L/K.$ The data $(L',\theta_L,U_L, \overline{U_{\mathcal O_L}})$ form a filtered set $\Sigma$. Note that for a morphism $(L,\theta_L,U_L, \overline{U_{\mathcal O_L}}) \to (L',\theta_{L'},U'_{L'}, \overline{U'_{\mathcal O_{L'}}})$ in $\Sigma$, the base change theorem tells us $$\hk((U'_{L'}, \overline{U'_{\mathcal O_{L'}}})_0)\otimes_{F_L'}F_L \simeq \hk((U_{L}, \overline{U_{\mathcal O_{L}}})_0).$$
	
	Let $\mathscr A_{\operatorname{HK}}$ be the h-sheafification of the presheaf $$(U, \overline{U})\mapsto \hk((U, \overline{U})_0):=\varinjlim_L\hk((U_L, \overline{U_{\mathcal O_L}})_0),$$ where the filtered limit is indexed by $\Sigma.$ This is a h-sheaf of dg solid $F^{\operatorname{nr}}$-algebra on $\operatorname{Alg}_{\overline K}$ equipped with a Frobenius operator $\phi$ and a monodromy operator $N$ such that $N\phi=p\phi N$. For an algebraic variety $X \in \operatorname{Alg}_{\overline K},$ we define the Hyodo-Kato cohomology of $X$ by $\hk(X):=\rg(X,\mathscr A_{\operatorname{HK}}).$ By a more careful data keeping, Beilinson proved in \cite{beilinson2013crys} that we have a natural Hyodo-Kato isomorphism (by fixing a pseudo-uniformizer): $$\iota_{\operatorname{HK}}\otimes_{F^{\operatorname{nr}}} \overline K:\hk(X)\otimes^L_{F^{\operatorname{nr}}}\overline K\xrightarrow{\simeq} \dr(X).$$
	
	We have the following local-global compatibility.
	
	\begin{prop} \label{hklocalglobal}
		(i)\cite[2.5]{beilinson2013crys} For any $(V,\overline{V}) \in \operatorname{Var}_{\overline{K}}^{\operatorname{ss}},$ the natural map $$\hk(V,\overline{V})\to \hk(V)$$ is a quasi-isomorphism.
		
		(ii)\cite[Theorem 3.18]{janniziol2016hk} For any $(V,\overline{V}) \in \operatorname{Var}_{K}^{\operatorname{ss}},$ the natural map $$\hk(V,\overline{V})\to \hk(V)$$ is a quasi-isomorphism.
	\end{prop}
	
	\begin{rem}
		According to \cite[Theorem 3.18]{janniziol2016hk}, the property also holds for log smooth schemes over $\mathcal O_K^{\times},$ we will not use this fact in the following chapters.
	\end{rem}
	
	\subsubsection{Algebraic Hyodo-Kato Cohomology with compact support} 
	
	We define the Hyodo-Kato Cohomology with compact support in this section.
	
	We first recall some basic definition and properties of compactification. For a morphism $X \xrightarrow{} Y$ of schemes, we say $X$ has a compactification over $Y$ if there exists a quasi-compact open immersion $X \to \overline{X}$ such that $\overline X$ is proper over $Y$. For $X \xrightarrow{} Y$ separated and of finite type, such a compactification always exists. The category of compactifications of $X$ over $Y$ is cofiltered, which means for two compactifications $X\to \overline{X_1}$ and $X\to \overline{X_2}$, there exists a compactification $X\to \overline{X}$ and morphisms $\overline{X_1} \xleftarrow{} \overline{X} \xrightarrow{} \overline{X_2}$ over $Y$ such that both morphisms are isomorphisms on the open subset $X$. Moreover, suppose we have a morphism $X \xrightarrow{} Y$ of algebraic varieties over a field $L$, and a compactification $X \to \overline{X}$ over $L$, then one can construct a compactification $X \to \overline{X'}$ over $Y$, by taking $\overline{X'}$ to be the scheme theoretic image of $X \to \overline{X} \times Y.$
	
	\begin{defn} \label{hkcalgebraic}
		Let $X$ be an algebraic variety, and $X \hookrightarrow \overline{X}$ be a compactification of $X$ over $L$. We define the Hyodo-Kato cohomology with compact support of $X$ to be
		$$\hkc(X):=[\hk(\overline{X})\to \hk(\overline{X}-X)].$$
	\end{defn}	
	
	Since the Hyodo-Kato Cohomology satisfies h-descent, one checks easily that $\hkc(X)$ is well-defined: indeed, suppose we have two compactifications of $X$ in $L:$ $X\hookrightarrow \overline{X_1}$ and $X\hookrightarrow \overline{X_2}$, then we have a compactification $X\to \overline{X}$ and morphisms $\overline{X_1} \xleftarrow{} \overline{X} \xrightarrow{} \overline{X_2}$. Then we have a abstract blow up square:
	$$\begin{tikzcd}
		\overline X-X \arrow[r] \arrow[d] & \overline X \arrow[d] \\
		\overline{X_i}-X \arrow[r] & \overline{X_i}
	\end{tikzcd}$$
	for $i=1,2.$ This is in fact an \'eh-covering in the sense of \cite{thomas2006eh}.

	\begin{rem}
		For a general cohomological theory satisfying \'eh-descent, one can define similarly the cohomology with compact support.
	\end{rem}
	
	\begin{prop}
		Suppose $X$ be an algebraic variety, $U \hookrightarrow X$ be an open immersion. Then we have
		$$\hkc(U)\simeq  [\hkc(X) \xrightarrow{} \hkc(X-U)].$$
	\end{prop}
	
	\begin{proof}
		Choose a compactification of $X$: $X\hookrightarrow \overline{X}$, then $\overline{X}$ is also a compactification of $U,$ and $\overline X-U$ is a compactification of $X-U.$ Then the proposition follows from the definition of compactly supported cohomology for $X,U$ and $X-U.$
	\end{proof}
	
	\begin{prop}
		The Hyodo-Kato Cohomology with compact support is contravariant for proper morphisms and covariant for open immersions.
	\end{prop}
	
	\begin{proof}
		If $U \hookrightarrow X$ is an open immersion, as above we choose a compactification $X\hookrightarrow \overline{X}$ of $X$. Then the map $\hk(\overline X-X) \to \hk(\overline X-U)$ induces $\hkc(U) \to \hkc(X)$.
		
		If $Y \to X$ is a proper morphism, then choose a compactification $X\hookrightarrow \overline{X}$ of $X$ and  a compactification $Y\hookrightarrow \overline{Y'}$ of $Y$. Denote by $\overline{Y}$ the closure of the image of the canonical morphism $Y \to \overline{Y'} \times \overline{X}$. We have a Cartesian square $$\begin{tikzcd}
			Y \arrow[r,hookrightarrow] \arrow[d] & \overline Y \arrow[d] \\
			{X} \arrow[r,hookrightarrow] & \overline{X},
		\end{tikzcd}$$
		because the map from $Y$ to the preimage of $X$ in $\overline{Y}$ is proper and open, so it is identity. Then $\hkc(X) \to \hkc(Y)$ is given by the following commutative square 
		$$\begin{tikzcd}
			\overline Y-Y \arrow[r,hookrightarrow] \arrow[d] & \overline Y \arrow[d] \\
			\overline{X}-X \arrow[r,hookrightarrow] & \overline{X}.
		\end{tikzcd}$$
	\end{proof} 
	
	\subsubsection{The Hyodo-Kato realization functor}
	
	The Hyodo-Kato cohomology with compact support is also defined in \cite{nizioldeglise2018syntomic}, by constructing the realization functor of Hyodo-Kato cohomology, and then applying the six functor formalism for Voevodsky's motive. In this part we will briefly summarize \cite{nizioldeglise2018syntomic}. This will give us the pairing and the Poincar\'e duality.
	
	Let $DM(\overline{K},\Q_p)$ be the triangulate category of Voevodsky’s mixed motives with rational coefficients, and $M_{gm}$ be the natural functor $$M_{gm}:\operatorname{Sch}/\overline{K} \to DM(\overline{K},\Q_p).$$ We refer the reader to \cite{voevodsky2000motive} for the definition and properties of Voevodsky’s motives. According to \cite{nizioldeglise2018syntomic}, there exists Hyodo-Kato realization functor $$\hk:DM(\overline{K},\Q_p) \to D(\Mod^{\solid}_{F^{\mathrm{nr}}}),$$ such that its composition with $M_{gm}$ gives the geometric Hyodo-Kato cohomology of Beilinson, i.e. for $X\in\operatorname{Sch}/\overline{K}, $ we have $\hk(M_{gm}(X))\simeq\hk(X).$ It is worth to mention that the realization functor also gives us the Gysin sequence for smooth pairs.
	
	\begin{prop}\label{gysinalg}
		If $X$ is a smooth variety over $\overline{K}$ and $Z \subset X$ is a smooth closed subvariety, denote by $U:=X-Z.$ We have $$\hk(Z)\{-1\}[-2]\simeq [\hk(X) \to \hk(U)].$$
	\end{prop}
	
	\begin{proof}
		This follows from the distinguished triangle of \cite[Proposition 3.5.4]{voevodsky2000motive}.
	\end{proof}
	
	There is also a functor $$M_{gm}^c:\operatorname{Sch}/\overline{K} \to DM(\overline{K},\Q_p)$$ from the category of schemes of finite type over $\overline{K}$ to $DM(\overline{K},\Q_p)$, which sends $X \in \operatorname{Sch}/\overline{K}$ to its associated motives with compact support. We define the Hyodo-Kato cohomology with compact support of $X$ by $$\hkc(X):=\hk(M_{gm}^c(X)).$$

	\begin{lem}
		The definition of Hyodo-Kato cohomology with compact support in \cite{nizioldeglise2018syntomic} agrees with the Definition \ref{hkcalgebraic}.
	\end{lem}
	
	\begin{proof}
		This follows from the distinguished triangle of \cite[Proposition 4.1.5]{voevodsky2000motive}.
	\end{proof}
	
	Thanks to the six functor formalism for Voevodsky's mixed motives (see for example \cite{deglise2019motive}), one can construct a pairing for smooth varieties $X$: $$\hk(X)\otimes \hkc(X)\to \hkc(X),$$ and the Hyodo-Kato realization functor gives us the Poincar\'e duality for algebraic varieties as follows.
	
	\begin{theorem} \label{algpoincare}
		Let $X$ be a smooth algebraic varieties over $\overline{K}$ of dimension $d$, then:
		
		(i) There is a natural trace map $${\operatorname{tr_{HK}}}:\hkc(X) \to F^{\operatorname{nr}}\{-d\},$$ compatible with the Hyodo-Kato morphism.
		
		(ii) The pairing $$\hk(X) \otimes^L\hkc(X)[2d] \xrightarrow{} \hkc(X)[2d]\to F^{\mathrm{nr}}\{-d\}$$ is a perfect duality of solid $L_F$-vector spaces, i.e., we have induced isomorphisms in $D(\Mod^{\solid}_{F^{\mathrm{nr}}}):$
		$$H^i_{\operatorname{HK}}(X) \simeq \operatorname{Hom}_{F^{\operatorname{nr}}}(H^{2d-i}_{\operatorname{HK,c}}(X),F^{\operatorname{nr}}\{-d\}),$$
		$$H^i_{\operatorname{HK,c}}(X) \simeq \operatorname{Hom}_{F^{\operatorname{nr}}}(H^{2d-i}_{\operatorname{HK}}(X),F^{\operatorname{nr}}\{-d\}).$$
	\end{theorem}	
	
	\subsection{Hyodo-Kato cohomology for analytic varieties}
	
	We review the definitions and properties for Hyodo-Kato Cohomology (with compact support) of analytic varieties.
	
	\subsubsection{Definition and properties} We briefly summarize the definition and basic properties of Hyodo-Kato Cohomology discussed in \cite{CNderham} and \cite{CN4.3}.
	
	(1) Arithmetic setup: We define here Hyodo-Kato cohomology of smooth rigid analytic varieties over $K$. Suppose $\mathscr X\in \mathscr M^{\operatorname{ss}}_K$, which is defined in Definition \ref{modelrigid}. Let $\mathscr A_{\operatorname{HK}}$ be the \'eh-sheafification of the presheaf $$\mathscr X\mapsto \hk(\mathscr X_0):=\rg_{\operatorname{cris}}(\mathscr X_0/\mathcal O^0_{F_L})_{\Q_p}$$ with value in $D(\Mod_{F^{\operatorname{nr}}}^{\operatorname{solid}})$. Here the condensed structure of $\rg_{\operatorname{cris}}(\mathscr X_0/\mathcal O^0_{F_L})$ comes from the inverse limit of the condensed abelian groups $\varprojlim \rg_{\operatorname{cris}}(\mathscr X_0/\mathcal O^0_{F_L,n}).$ The sheaf $\mathscr A_{\operatorname{HK}}$ is a \'eh-sheaf of dg solid $F$-algebra on $\mathscr M_K$ equipped with a Frobenius operator $\phi$ and a monodromy operator $N$ such that $N\phi=p\phi N$. For a rigid analytic variety $X \in \operatorname{Rig}_K,$ we define the Hyodo-Kato cohomology of $X$ by $\hk(X):=\rg(X,\mathscr A_{\operatorname{HK}}).$ By \cite{CNderham}, there exist a Hyodo-Kato morphism $$\iota_{\operatorname{HK}}:\hk(X)\to \dr(X).$$
	
	(2) Geometric setup: Suppose now $\mathscr X\in \mathscr M^{\operatorname{b,ss}}_{C}$ be a basic semistable formal model. By definition, $f:\mathscr X \to \Spf (O_C)^{\times}$ is the base change of a semistable pair $\theta_L:\mathscr X_L\to \Spf (\mathcal O_{L})^{\times}$ for a finite field extension $L/K.$ The data $(L,\theta_L,\mathscr X_L)$ form a filtered set $\Sigma$. Note that for a morphism $(L,\theta_L,\mathscr X_L) \to (L',\theta_{L'},\mathscr X'_{L'})$ in $\Sigma$, the base change theorem tells us $$\hk(\mathscr X'_{L',0})\otimes_{F_L'}F_L \simeq \hk(\mathscr X_{L,0}).$$
	
	Let $\mathscr A_{\operatorname{HK}}$ be the \'eh-sheafification of the presheaf $$\mathscr X\mapsto \hk(\mathscr X_0):=\varinjlim_L\hk(\mathscr X_{L,0}),$$ where the filtered limit is indexed by $\Sigma.$ This is a \'eh-sheaf of dg solid $F^{\operatorname{nr}}$-algebra on $\operatorname{Rig}_{C}$ equipped with a Frobenius operator $\phi$ and a monodromy operator $N$ such that $N\phi=p\phi N$. For a rigid analytic variety $X \in \operatorname{Rig}_{C},$ we define the Hyodo-Kato cohomology of $X$ by $\hk(X):=\rg(X,\mathscr A_{\operatorname{HK}}).$ 
	
	We have the following local-global compatibility.
	
	\begin{prop} \cite[Proposition 4.11 and Proposition 4.23]{CNderham} \label{hkanalyticlocalglobal}
		(i) For a semistable formal scheme $\mathscr{X}$ over $\mathcal{O}_K,$ the natural morphism $$\hk(\mathscr{X}_1)\to \hk(\mathscr{X}_K)$$ in $D(\Mod_{F}^{\operatorname{solid}})$ is a strict quasi-isomorphism.
		
		(i) For $\mathscr{X} \in \mathscr{M}^{\operatorname{ss,b}}_C,$ the natural morphism $$\hk(\mathscr{X}_1)\to \hk(\mathscr{X}_C)$$ in $D(\Mod_{F^{\operatorname{nr}}}^{\operatorname{solid}})$ is a strict quasi-isomorphism.
	\end{prop}
	
	In the geometric setup, we have the Hyodo-Kato isomorphism. 
	
	The following Hyodo-Kato isomorphism follows from \cite[Theorem 4.6]{CN4.3} and \cite[Theorem 2.3.14]{bosco2023rational}.
	
	\begin{prop} \label{hkisomorphism}
		Suppose $X$ is a dagger variety over $C$, we have a natural isomorphism in $D(\operatorname{Mod}_{C}^{\mathrm{solid}}):$ $$\rg_{\operatorname{HK}}(X)\otimes_{{F}^{\mathrm{nr}}}^{\operatorname{L}_\blacksquare}C \xrightarrow{\simeq} \dr(X).$$
	\end{prop}
	
	
	\subsubsection{Compactly supported Hyodo-Kato cohomology for rigid analytic varieties} \label{compact}
	
	We review the definition of compactly supported Hyodo-Kato cohomology for rigid analytic varieties, which is defined in \cite[4.3]{AGNcompact}.
	
	\begin{defn}
		Let $X$ be a rigid analytic varieties over $L=K$ or $C$, the compactly supported Hyodo-Kato cohomology for $X$ is defined to be $$\hkc(X):=\varinjlim_{U\subset X}[\hk(X)\to\hk(X-U)]$$ in $D(\Mod_{L_F}^{\operatorname{solid}})$, where the filtered colimit is indexed by quasi-compact opens $U$ in $X$, and $L_F=F$ when $L=K$ and $L_F=F^{\operatorname{nr}}$ when $L=C$. 
	\end{defn}
	
	We have similar definitions for compactly supported de Rham, $B_{\operatorname{dR}}^+$, \'etale, syntomic and pro-\'etale cohomology for rigid analytic varieties. All these cohomology groups with compact support are contravariant with proper morphisms and covariant with open immersions.
	
	\begin{rem}
		One should be careful with the terminology used. When $X$ is partially proper, the definition of compactly supported pro-\'etale cohomology coincidences with Huber's definition of compactly supported \'etale cohomology in \cite{huber2013etale} (and \cite{huber1998comparison} for $p-$adic cohomology), see \cite[Section 2]{AGNcompact} for more details.
	\end{rem}
	
	\subsubsection{Compactly supported Hyodo-Kato cohomology for dagger analytic varieties}
	
	Now suppose $X$ is a dagger variety over $L=K$ or $C$. We firstly assume that $X$ is dagger affinoid. According to \cite{vezzani2018monskywashnitzer} we can take a presentation $\{X_k\}_{i\geq 1}$ of $X$, where all $X_k$ are smooth affinoid rigid analytic varieties over $L$ such that $X_1 \Supset X_2 \Supset ...$ and there intersection is $X.$ Denote by $\mathscr A_{\operatorname{HK}}$ the analytic sheafification of the presheaf $$X \mapsto \rg_{\mathrm{HK,c}}^{\natural}(X):=\colim_k[\hk(X_k)\to \hk(X_k-X)].$$ For a general $X$, define the overconvergent compactly supported Hyodo-Kato cohomology by $$\hkc(X):=\rg(X,\mathscr A_{\operatorname{HK}}).$$ This definition satisfies local global compatibility.
	
	\begin{prop}\cite[Corollary 4.16]{AGNcompact}
		If $X$ is dagger affinoid, then we have a natural quasi-isomorphism $$\hkc(X)\xrightarrow{\simeq} \rg_{\mathrm{HK,c}}^{\natural}(X)$$
	\end{prop}
	
	We have similar definitions and properties for compactly supported de Rham, $B_{\operatorname{dR}}^+$, syntomic and pro-\'etale cohomology for dagger analytic varieties.
	
	When $X$ is partially proper, we can compare the two definitions as expected.
	
	\begin{prop}\cite[Proposition 4.18]{AGNcompact}
		Suppose $X$ is moreover partially proper, then we have a natural quasi-isomorphism $$\hkc(X)\xrightarrow{\simeq}\hkc(\widehat{X}).$$ The same proposition also holds for compactly supported de Rham, $B_{\operatorname{dR}}^+$,  \'etale and pro-\'etale cohomology.
	\end{prop}

	\subsubsection{Galois descent}
	
	In the arithmetic setup, one can recover the Hyodo-Kato by using Galois descent.
	
	\begin{prop} \label{hkgaloisdescent}
		Suppose $X$ is a  dagger variety over $K$, and $G:=\Gal(\overline{K}/K)$. Then the natural projection $X_C \to X$ induces a natural isomorphism in $D(\operatorname{Mod}_{\Q_p}^{\operatorname{solid}}):$
		$$\hk(X)\xrightarrow{\simeq}\hk(X_C)^{G_K},\hkc(X)\xrightarrow{\simeq}\hkc(X_C)^{G_K}.$$
	\end{prop}
	
	\begin{proof}
		We may assume $X$ is dagger affinoid, as solid $G$-cohomology commutes with limits and also filtered colimits since $G$ is profinite. For the Hyodo-Kato cohomology, this follows from \cite[Remark 5.21]{CNderham} and finiteness of Hyodo-Kato cohomology. The claim for Hyodo-Kato cohomology with compact support then follows.
	\end{proof}

	\section{Geometric Properties}
	
	As outlined in the introduction, this section is devoted to proving some geometric properties for analytic varieties, mainly for dagger varieties. This includes the Mayer-Vietoris property, the Poincar\'e duality, the open-closed exact sequence and the Gysin isomorphism.
	
	\subsection{Mayer-Vietoris property}
	
	Using the results from \cite{guo2019hodgetate}, we construct several Mayer-Vietoris exact sequences for rigid analytic and dagger varieties. We start with an abstract definition of the Mayer-Vietoris property.
	
	\begin{defn}
		Let $\mathcal{C}$ be a site, $F: \mathcal{C}^{\operatorname{op}} \to \mathcal{D}$ be a pre(co)sheaf valued in a stable $\infty$-category $\mathcal{D}$. For a commutative square
		$$\begin{tikzcd}
			Y' \arrow[r] \arrow[d] & X' \arrow[d] \\
			Y  \arrow[r] & X
		\end{tikzcd}$$
		we will say that $F$ satisfies the Mayer-Vietoris property for this square if $F$ sends it to a cartesian square. In this case, we will also call this square $F$-acyclic.
	\end{defn}
	
	\begin{ex}
		Let $\mathcal{C}=\operatorname{Sch}_{\tau}$ (or $\operatorname{Rig}_{\tau},\operatorname{Rig}_{\tau}^{\dag}$,...) with a Grothendieck topology $\tau$ (of suitably bounded cardinality if $\tau$ is large), $\mathcal{D}=\mathcal{D}(\operatorname{CondAb})$, $F=\operatorname{R\Gamma}(-,\f)$ for a condensed abelian sheaf $\f$ on $X_{\tau}$. Then the square
		$$\begin{tikzcd}
			U \cap V \arrow[r] \arrow[d] & V \arrow[d] \\
			U \arrow[r] & X
		\end{tikzcd}$$ is $F$-acyclic for any $U,V$ open in $X$ with $X=U \cup V$. This is the usual Mayer-Vietoris sequence. Similarly, for $\mathcal{C}=\operatorname{Sch}_{\tau}$ or $\operatorname{Rig}_{\tau}^{\dag}$, cosheaf $F=\operatorname{R\Gamma}_c(-,\f)$ is $F$-acyclic.
	\end{ex}
	
	\begin{ex}
		\cite{thomas2006eh} implies for a commutative ring $\Lambda$, $F=\operatorname{R\Gamma}_{\operatorname{\acute{e}t}}(-,\underline{\Lambda}):(\operatorname{Sch}/k)^{op} \to \mathcal{D}(\operatorname{Ab})$ has the Mayer-Vietoris property for abstract blow up squares. An abstract blow up square means $Y$ is a closed immersion of $X$, $f:X' \to X$ is proper surjective which induces an isomorphism of $X'-f^{-1}(Y) \to X -Y$, and $Y'=X' \times_X Y.$ We will prove similar results in the rigid analytic/dagger settings.
	\end{ex}
	
	\begin{ex}
		According to \cite{beilinson2012derham} and \cite{beilinson2013crys}, $F=\dr(-)$ or $F=\hk(-)$ has the Mayer-Vietoris property for proper surjective cartesian square, i.e. $Y'=X' \times_X Y$ and $X' \to X$ is proper surjective.
	\end{ex}
	
	Now we will state the main result of this section.
	
	\begin{prop} \label{mvmain}
		Suppose $\mathcal{C}=\operatorname{Rig}_{L}$ or $\operatorname{Rig}_{L}^{\dag}$, $L=K$ or $C$. Let $X \in \mathcal{C}, Y \subset X$ be a nowhere dense analytic closed subspace, $Y'=X' \times_X Y,$ Then any square 
		$$\begin{tikzcd}
			Y' \arrow[r] \arrow[d] & Y \arrow[d] \\
			X' \arrow[r] & X
		\end{tikzcd}$$
		satisfying one of the following two conditions
		\begin{enumerate}
			\item $X' \to X$ is a blowup along $Y$.
			\item $X$ is quasi-compact, $Y$ is an irreducible component of $X$, and $X'$ is the union of all the other irreducible components of $X$.
		\end{enumerate}   
		Then for the functors $F=\operatorname{R\Gamma}_{\rm pro\acute{e}t}(-,\Q_p)$ or $\operatorname{R\Gamma}_{\rm \acute{e}t}(-,\Q_p)$ or $\operatorname{R\Gamma}_ {\rm \acute{e}h}(-,\f)$ or $\operatorname{R\Gamma}_{\rm dR}(-)$ or $\operatorname{R\Gamma}(-/B^+_{\operatorname{dR}})$ or $\operatorname{R\Gamma}_{{\rm HK}}(-)$ (only for $\mathcal{C}=\operatorname{Rig}_{L}^{\dagger}$ in this case) from $\mathcal{C}^{\operatorname{op}} \to D(\Mod^{\solid}_{\Q_p})$, the above commutative square is $F$-acyclic. 	
	\end{prop}
	
	\begin{rem}
		After covering $X$ by open (dagger) affinoids, as limits and colimits of cartesian squares is still cartesian, we may assume $X$ is affinoid or dagger affinoid in each situation.
	\end{rem}
	
	\begin{proof}
		We assume first $X$ is affinoid.
		
		\begin{itemize}
			\item \'eh and de-Rham cohomology: This follows from \cite[Proposition 5.1.4]{guo2019hodgetate}.
			\item \'etale and pro-\'etale cohomology: The sheaf $X \to \operatorname{R\Gamma}(X_{\rm \acute{e}t},\Lambda)$ is a v-sheaf, hence an \'eh-sheaf over $\operatorname{Rig}_{\operatorname{\acute{e}h}}$. Therefore, $\operatorname{R\Gamma}(X_{\rm \acute{e}t},\Lambda)=\operatorname{R\Gamma}(X_{\rm \acute{e}h},\Lambda)$ for any commutative ring $\Lambda$. After taking limits we get the Mayer-Vietoris property for $\operatorname{R\Gamma}(X_{\rm \acute{e}t},\Z_p)$, tensoring $\Q_p$ we get the Mayer-Vietoris property for $\operatorname{R\Gamma}(X_{\rm \acute{e}t},\Q_p)$. Since $X$ is quasi-compact, $\operatorname{R\Gamma}(X_{\rm \acute{e}t},\Q_p)= \operatorname{R\Gamma}(X_{\rm pro\acute{e}t},\Q_p)$ and therefore the Mayer-Vietoris property also holds for $\operatorname{R\Gamma}(X_{\rm pro\acute{e}t},\Q_p)$.
		\end{itemize}
		
		Now we assume $X$ is dagger affinoid.
		
		\begin{itemize}
			\item \'eh and de-Rham cohomology: The proof is same as \cite[Proposition 5.1.4]{guo2019hodgetate} except that here we need to use overconvergent setup.
			\item \'etale cohomology: By \cite{vezzani2018monskywashnitzer}, $\operatorname{R\Gamma}(X_{\rm \acute{e}t},\Lambda) \simeq \operatorname{R\Gamma}(\widehat{X}_{\rm \acute{e}t},\Lambda)$ for any commutative ring $\Lambda$, and the Mayer-Vietoris property follows from the affinoid case.
			\item pro-\'etale cohomology: We use the presentation of an affinoid dagger variety. Take a presentation $\{X_h\}$ of $X$, and we may assume $X_1$ is small enough so that $Y$ can be extended to a closed subvariety $Y_1 \subset X_1$ (because $X$ is affinoid, $Y$ is defined by a finite number of formal power series which can be extended their definition a little bit outside $Y \subset X_1$). Therefore, we may assume $X$ has a presentation $\{X_h\}$ such that it can induce a presentation $\{Y_h\}$ of $Y$ such that $Y_h \subset X_h$ is closed immersions.
			
			We can then construct $X'_h=\operatorname{Bl}_{Y_h}(X_h)$, the blowup of $X_h$ along $Y_h$, and denote by $Y'_h=X'_h \times_{X_h} Y_h$. We note that $X'_h=X'_1 \times _{X_1} X_h$ for any $h$. According to the definition of pro-\'{e}tale cohomology of dagger varieties, $\rg (X_{\rm pro\acute{e}t},\Q_p)=\colim_h \rg (X_{h, \rm pro\acute{e}t},\Q_p), \rg (Y_{\rm pro\acute{e}t},\Q_p)\simeq\colim_h \rg (Y_{h, \rm pro\acute{e}t},\Q_p)$, we want to prove that $$\rg (X'_{\rm pro\acute{e}t},\Q_p)\simeq\colim_h \rg (X'_{h, \rm pro\acute{e}t},\Q_p), \rg (Y'_{\rm pro\acute{e}t},\Q_p)=\colim_h \rg (Y'_{h, \rm pro\acute{e}t},\Q_p).$$
			
			According to the construction of blowups, if we assume $X_h=\operatorname{Sp}(A_h)$ and $X=\operatorname{Sp}^{\dagger}(A)$, we can find a finite dagger affinoid hypercovering of $X'$ by $U^{\bullet}$ such that locally they are the disjoint union of $\Spec(A_{(a)})$ (relative Spec, see \cite{conard2006relative}), and we can construct hypercoverings of $X_h'$ by $U_h^{\bullet}$ in a compatible way (which means locally they are the disjoint union of $\Spec(A_{h,(a)})$). Therefore,
			$$\rg (X'_{\rm pro\acute{e}t},\Q_p)\simeq\rg (U^{\bullet}_{\rm pro\acute{e}t},\Q_p)\simeq\colim_h \rg (U^{\bullet}_{h, \rm pro\acute{e}t},\Q_p)\simeq\colim_h \rg (X'_{h, \rm pro\acute{e}t},\Q_p),$$ 
			where the middle isomorphism follows from the fact that colimits commute with finite limits. The same argument also shows that $\rg (Y'_{\rm pro\acute{e}t},\Q_p)\simeq\colim_h \rg (Y'_{h, \rm pro\acute{e}t},\Q_p)$.
			\item Hyodo-Kato cohomology: We prove the claim for $L=C.$ We know that $\operatorname{R\Gamma}_{\rm HK}(X) \otimes_F^{\blacksquare} C \simeq \operatorname{R\Gamma}_{\rm dR}(X).$ Write $C \simeq F \oplus W$ for some $F$-vector space $W$. Since $\operatorname{R\Gamma}_{\rm dR}(X)$ has the Mayer-Vietoris property, $\operatorname{R\Gamma}_{\rm HK}(X)$, as a direct summand of $\operatorname{R\Gamma}_{\rm dR}(X)$ in $D(\Mod^{\solid}_{\Q_p})$, must also have the Mayer-Vietoris property. 
			
			When $L=K,$ the claim follows from the case for $L=C,$ finiteness of Hyodo-Kato cohomology Proposition \ref{finitenesshk}, and Galois descent.
		\end{itemize}
	\end{proof}
	
	The above proof for dagger varieties also shows we have Mayer-Vietoris properties for cohomology with compact support for dagger varieties.
	
	\begin{prop} \label{mvc}
		The above proposition also holds for $\mathcal{C}=\operatorname{Rig}_{L}^{\dag}$, $L=K$ or $C$, and $F=\operatorname{R\Gamma}_{\rm pro\acute{e}t,\rm c}(-,\Q_p)$ or $\operatorname{R\Gamma}_{\mathrm{dR,c}}(-/B^+_{\operatorname{dR}})$ or $\operatorname{R\Gamma}_{{\rm HK,c}}(-)$.
	\end{prop}
	
	\begin{proof}
		By assuming $X$ is dagger affinoid and cover $X'$ with dagger affinoids as above, the proposition follows from the cosheaf condition and the above Mayer-Vietoris property for the corresponding cohomology (without compact support).
	\end{proof}

	\subsection{The open-closed exact sequence}
	
	A characteristic property of cohomology with compact support is the existence of open-closed long exact sequence. For example, for \'etale cohomology with compact support, let $X$ be a partially proper rigid analytic variety, $D \hookrightarrow X$ is a Zariski closed immersion, and $U:=X-D \hookrightarrow X$ is an open immersion. We have the following proposition, which is implied by \cite[Remark 5.1.1]{huber2013etale}.
	
	\begin{prop}
		We have $$\etc(U,\Q_p)\simeq[\etc(X,\Q_p)\to \etc(D,\Q_p)].$$
	\end{prop}
	
	In this section, we will prove our definition of cohomology with compact support for rigid analytic varieties have the open-closed exact sequence, when $X$ is proper smooth or smooth with strict normal crossings divisor $Z$. 
	
	Let $L=K$ or $C$, suppose $X$ is a smooth dagger variety over $L$. Let $i: Z \hookrightarrow X$ be a Zariski closed immersion, $U=X-Z$.
	
	\begin{prop}\label{open-closed}
		Assume moreover $Z$ is smooth, or $Z$ is a strict normal crossing divisor. Then we have
		$$\drc(U)\simeq  [\drc(X) \xrightarrow{} \drc(Z)],$$
		$$\hkc(U)\simeq  [\hkc(X) \xrightarrow{} \hkc(Z)].$$
	\end{prop}
	
	\begin{proof}
		This problem is local, so we can assume $X=\Sp(A)$ is a smooth dagger affinoid variety. We prove the claim for de Rham cohomology first.
		
		Choose a representation $A=L\left\langle x_1,...,x_n\right\rangle^\dagger /I,$ and assume $Z=\Sp(L\left\langle x_1,...,x_n\right\rangle^\dagger /J)$. By definition we can find a $\delta>1$ such that $I \subset J \subset T_n(\delta)=L\left\langle \delta^{-1}x_1,...,\delta^{-1}x_n\right\rangle.$ By shrinking $\delta$ we can assume $I \subset J \subset W_n(\delta)=L\left\langle \delta^{-1}x_1,...,\delta^{-1}x_n\right\rangle^\dagger.$ Take a decreasing sequence $\delta_1=\delta, \delta_h \to 1$. Define $X_h=\Sp(W_n(\delta_h)/IW_n(\delta_h)), Z_h=\Sp(W_n(\delta_h)/JW_n(\delta_h))$. Moreover, if we take a generator of $J=(f_1,...,f_m),$ for any $\rho\geq 1$ write $$B_k(\rho)=L\left\langle \rho^{-1}x_1,...,\rho^{-1}x_n,\omega^{-k}f_1,...,\omega^{-k}f_m\right\rangle^\dagger,$$ we can define a decreasing sequence of neighborhood of $Z_h$ (resp. $Z$) by $V_{k,h}=\Sp(B_k(\delta_h)/JB_k(\delta_h))$ (resp. $V_{k}=\Sp(B_k(1)/JB_k(1))=V_{k,h}\cap X$).
		
		\begin{lem}\label{problemlemma}
			We have natural isomorphisms
			$$\dr(Z_1)\simeq \colim_k \dr (V_{k,1}),$$
			$$\dr(Z_1-Z)\simeq \colim_k \dr (V_{k,1}-(V_{k,1} \cap X)).$$
		\end{lem}
		
		\begin{proof}
			The claims are given by the definition of de Rham cohomology by Elmar Grosse-Kl\"onne, see Appendix \ref{appa}. For the first one, note that $V_{k,1}$ is a cofinal system of the set of admissible open neighborhood $U \subset X_1$ of $Z_1$, then by Theorem \ref{drgk} (denote by $j_k:V_{k,1} \hookrightarrow X_1$),
			$$\colim_k \dr (V_{k,1}) \simeq \colim_k \rg (X_1, Rj_{k,*}\Omega^{\bullet}|_{V_{k,1}}) \simeq \colim_k \rg (X_1, j_{k,*}\Omega^{\bullet}|_{V_{k,1}}) \simeq \dr(Z_1).$$
			
			The second argument is more subtle, but in the case $Z$ is smooth, or $Z$ is a strict normal crossing divisor, by \cite[Theorem 1.18]{kiehl1967derham}, we may further assume $$X \simeq S \times \mathbb{B}^r \simeq S \times \Sp(L\langle x_1,...,x_r \rangle),$$ $$X_h \simeq S_h \times \mathbb{B}(\delta_h)^r \simeq S_1 \times \Sp(L\langle \delta_h^{-1}x_1,...,\delta_h^{-1}x_r \rangle),$$ where $S,S_h$ are smooth affinoids, and$$Z_h \simeq S_h \simeq S_h \times \Sp(L\langle \delta_h^{-1}x_1,...,\delta_h^{-1}x_r \rangle/(x_1,...,x_r)) \hookrightarrow X$$ when $Z$ is smooth, or $$Z_h \simeq S_h \times \Sp(L\langle \delta_h^{-1}x_1,...,\delta_h^{-1}x_r \rangle/(x_1 x_2... x_r)) \hookrightarrow X_1$$ when $Z$ is a strict normal crossing divisor. Let $j_{k,h}:V_{k,h}-(V_{k,h} \cap X) \hookrightarrow X_h-X.$ When $Z$ is smooth, we can let $V_{k,h} \simeq S_h \times \mathbb{B}(\omega^k)^r$ and $j_{k,h}$ can be described as $(S_h-S) \times \mathbb{B}(\omega^k)^r.$ By Theorem \ref{drgk},
			\begin{align*}
				\dr (Z_1-Z) & \simeq \rg(X_1-X,\colim_k j_{k,1,*} \Omega_{V_{k,1}-V_K}^{\bullet}) \\ & \simeq R{\lim}_h \rg(X_1-X_h^{\circ},\colim_k j_{k,1,*}\Omega_{V_{k,1}-V_K}^{\bullet}|_{X_1-X_h^{\circ}}) \\ & \simeq R{\lim}_h\colim_k \dr(V_{k,1}-V_{k,h}^{\circ}) \\ & \simeq R{\lim}_h \dr(V_{k,1}-V_{k,h}^{\circ}) \simeq \dr (V_{k,1}-(V_{k,1} \cap X)).
			\end{align*}
			When $Z$ is a strict normal crossings divisor, one can use the \'eh-descent to reduce to the smooth case.	
		\end{proof}
		
		By definition, we have
		$$\drc (X) \simeq \colim_k [\dr(X_1) \xrightarrow{} \dr((X_1-X))].$$
		We also have 
		\begin{align*}
			\drc (U) & \simeq \colim_k\drc(X-V_k)\\ & \simeq \colim_k [\dr(X_1) \xrightarrow{} \dr((X_1-X)\cup V_{k,1})]
		\end{align*}
		
		    Using the lemma above, we have
			\begin{align*}
				\drc (Z) & \simeq [\dr(Z_1) \xrightarrow{} \dr((Z_1-Z))] \\ & \simeq \colim_k [\dr (V_{k,1}) \xrightarrow{} \dr (V_{k,1}-(V_{k,1} \cap X))] \\ & \simeq \colim_k [\dr((X_1-X)\cup V_{k,1}) \xrightarrow{} \dr((X_1-X))].
			\end{align*}	    
			The last isomorphism follows from the excision square:    
			$$\begin{tikzcd}
				V_{k,1}-(V_{k,1} \cap X) \arrow[r] \arrow[d] & V_{k,1} \arrow[d] \\
				X_1-X \arrow[r] & (X_1-X)\cup V_{k,1}.
			\end{tikzcd}$$
			By Combining the above three isomorphisms we get the desired proposition for de Rham cohomology.
			
			If $X$ is defined over $L=C$, the proof for Hyodo-Kato cohomology follows from the Hyodo-Kato isomorphism. Then for $L=K$ we use Galois descent. This concludes the proof.		
		\end{proof}
		
		
		\begin{rem}
			If $Z$ is not Zariski closed, it seems one could not even give a reasonable way to define $\drc(X)\to\drc(Z).$
		\end{rem}
		
		\begin{rem}
			One could begin by assuming the existence of a tubular neighborhood and proving the proposition through direct computation. Initially, we aimed to establish the proposition for any \( Z \). According to the proof, the primary issue for a general \( Z \) is that \( Z_1 - Z \) is not quasi-compact, and we do not have a straightforward characterization of the de Rham cohomology as described by Grosse-Klönne. It remains uncertain whether the requirement for \( Z \) can be relaxed. However, if \( X \) is proper, the definition of cohomology with compact support used in rigid analytic varieties allows us to remove the requirement for \( Z \) through a different approach, as presented below.
		\end{rem}
		
		\begin{prop} \label{properopenclosed}
			If $X$ is proper, we have
			$$\drc(U)\simeq  [\dr(X) \xrightarrow{} \dr(Z)],$$
			$$\hkc(U)\simeq  [\hk(X) \xrightarrow{} \hk(Z)].$$
		\end{prop}
		
		\begin{proof}
			By excision,
			\begin{align*}
				\drc(U) & \simeq \colim_{W \in \Phi(U)}[\dr(U) \xrightarrow{} \dr(U-W)] \\ & \simeq \colim_{W \in \Phi(U)}[\dr(X) \xrightarrow{} \dr(X-W)] .
			\end{align*}
			It is suffices to show
			$$\colim_{W \in \Phi(U)} \dr(X-W) \xrightarrow{\simeq} \dr(Z).$$
			Cover $X$ by finite affinoid dagger varieties $X_1,...,X_m,$ and let $Z_i=X_i \cap Z$. Write 
			$$X_i=\Sp(A_i)=\Sp(L\left\langle x_1,...,x_n\right\rangle^\dagger /I_i),$$ 
			$$X_i^{\circ}=\bigcup_{\delta \to 1}\Sp(L\left\langle \delta^{-1}x_1,...,\delta^{-1}x_n \right\rangle^\dagger /I_i),$$
			and assume $Z_i$ is cut out by an ideal $I_i=(f_{i1},...,f_{il}) \subset A_i$ (we can fix a large $l$ and a large $n$ for all $i$). For any small $\rho>0$, define 
			$$V_i(\rho)=\{|f_{ij}(x)|\leq\rho, j=1,2,...,l\}=\Sp\left(A\left\langle \rho^{-1} f_{i1},...,\rho^{-1} f_{il}\right\rangle^\dagger \right).$$
			Denote by $V(\rho)=V_1(\rho)\cup...\cup V_m(\rho)$.
			
			The proposition follows from the next lemma. 
			
			\begin{lem}
				$V(\rho)$ is an admissible open subset of $X$. Let $j(\rho):V(\rho) \hookrightarrow X$, then for any sheaf $\f$ on $V(\rho)$ locally free of finite rank, $R^ij(\rho)_*\f=0$ for $i>0$. Moreover, $\{V(\rho)\}_{\rho}$ is a cofinal system of the set of admissible open neighborhood $U \subset X$ of $Z$, and 
				$$\colim_{W \in \Phi(U)} \dr(X-W) \simeq \colim_{\rho \to 0} \dr(X-V(\rho)).$$
			\end{lem}
			
			\begin{proof}
				To prove $V(\rho)$ is admissible, it suffices to show $V(\rho) \cap X_i$ is admissible for each $i$. For this we use \cite[9.1.4, Corolly 4]{bosch1984nonarchimedean}. The second claim is clear. \cite[Lemma 2.3]{kisin1999localconstancy} implies $\{V(\rho)\}_{\rho}$ is a cofinal system of the set of admissible open neighborhood $U \subset X$ of $Z$. For the last statement, for any small $\rho$ define $$W_i(\rho)=\bigcup_j \Sp\left(A\left\langle \rho f_{ij}^{-1}\right\rangle^\dagger \right) \in X_i$$ and $W(\rho)=W_1(\rho)\cup...\cup W_m(\rho)$. Use $loc.$ $cit.$ $\{W(\rho)\}_{\rho}$ is a cofinal system of $\Phi(U)$. It is suffices to show for any $\rho$, there exist $\rho''\leq\rho'\leq\rho,$ such that $$X-V(\rho) \subset W(\rho') \subset X-V(\rho'').$$
				For the left side, we can just take $\rho=\rho'$ by the construction of $W(\rho)$. For the left side, $loc.$ $cit.$ ensures the existence of $\rho''$.
			\end{proof} 
			
			To complete the proof, since $X$ is proper,
			\begin{align*}
				\colim_{W \in \Phi(U)} \dr(X-W) & \simeq \colim_{\rho \to 0} \dr(X-V(\rho)) \\ & \simeq  \colim_{\rho \to 0} \rg(X,j(\rho)_*\Omega_{X-V(\rho)}^{\bullet}) \\ & \simeq \rg(X,\colim_{\rho \to 0}j(\rho)_*\Omega_{X-V(\rho)}^{\bullet}) \\ & \simeq \dr(Z).
			\end{align*}	
			The last isomorphism follows from the above lemma and Theorem \ref{drgk} from the Appendix \ref{appa}. Finally, the claim for Hyodo-Kato cohomology follows from the Hyodo-Kato isomorphism (for $L=C$, then for $L=K$ we use Galois descent).
		\end{proof}

		\subsection{The Poincar\'e duality}
		
		In this section, following \cite{AGNcompact}, we will review the Poincar\'e duality for de Rham cohomology and Hyodo-Kato cohomology of smooth dagger varieties. We will compare the Poincar\'e duality for de Rham cohomology with the one constructed in \cite{llz2023derham}. We will then deduce the Gysin sequence from the open-closed long exact sequence and Poincar\'e duality.
		
		The Poincar\'e duality for de Rham cohomology is formulated as follows.
		
		\begin{thm}\cite[Theorem 5.29]{AGNcompact} \label{logpoincaredr}
			Let $Y$ be a partially proper smooth rigid analytic variety or a quasi-compact smooth dagger variety over $L=K$ or $C$ of dimension $d$. Then there is a natural trace map $${\operatorname{tr_{dR}}}:\drc(Y)[2d] \to L,$$ that induce a perfect pairing $$\dr(Y) \otimes^{L_\blacksquare}_L\drc(Y)[2d] \xrightarrow{} \drc(Y)[2d]\to L.$$ Moreover, the trace map ${\operatorname{tr_{dR}}}$ is compatible with restrictions to open dagger subvarieties.
		\end{thm}
		
		In \cite{llz2023derham}, they also constructed the Poincar\'e duality for \'etale cohomolgy of an almost proper smooth rigid analytic variety $Y$ over $K$, i.e., $Y$ can be written as $X-Z$, where $X$ is a proper smooth rigid analytic variety defined over $K$ and $Z$ is Zariski closed in $X$. The main goal of this section is to compare these two construction, which will be used when proving GAGA in the next section.
		
		\begin{rem}
			In \cite{llz2023derham}, they also constructed the Poincar\'e duality for \'etale cohomolgy of almost proper smooth rigid analytic varieties. We refer the reader to \cite[Theorem 4.4.1]{llz2023derham} for the statements.
		\end{rem}
		
		We begin with compactly supported de Rham cohomology defined in \cite{llz2023derham}. Let $Y=X-Z,$ where $X$ is a proper smooth rigid analytic variety defined over $L=K$ or $C$ and $Z$ is strictly normal crossing divisor in $X$. We endow $X$ with the log structure induced by $Z$. Denote by $\mathcal I$ be the invertible ideal sheaf of $\mathcal O_X$ associated to the closed immersion $Z \hookrightarrow X.$
		
		\begin{defn} \cite[Definition 3.1.1]{llz2023derham} \label{drcllz}
			The compactly supported de Rham cohomology of $Y$ (by Lan-Liu-Zhu) is defined to be $$\rg_{\operatorname{dr,LLZ,c}}(Y):=\rg(X,\mathcal I \otimes_{O_X} \Omega_{X/L}^{\bullet,\log})$$ in $D(\Mod(L^{\solid})).$
		\end{defn}
		
		\begin{rem}
			This definition also applies to algebraic varieties.
		\end{rem}
		
		\begin{lem}
			We have a natural filtered quasi-isomorphism $$\rg_{\operatorname{dr,LLZ,c}}(Y) \xrightarrow{\simeq} \drc(Y).$$
		\end{lem}
		
		\begin{proof}
			This follows from Proposition \ref{properopenclosed} and \cite[Proposition 4.3.4]{llz2023derham}: write $Z$ as the union of irreducible components $Z=Z_1 \cup ... \cup Z_m$ where $Z_1,...,Z_m$ are smooth. Denote by $Z^{\bullet}$ the \v{C}ech nerve of the map $\coprod Z_i \to Z$, then both definitions can be written as $$\rg_{\operatorname{dR,*,c}}(U)\simeq [\rg_{\operatorname{dR}}(X)\to \rg_{\operatorname{dR}}(Z^{\bullet})],$$ where $*=\varnothing$ or LLZ.
		\end{proof}

		Now suppose $X$ is defined over $L=K$. Denote by $$\operatorname{tr_{dR,LLZ}}:\rg_{\operatorname{dr,LLZ,c}}(Y)[2d] \to K$$ the trace map defined in \cite[Theorem 4.2.1]{llz2023derham}. We want to compare $\operatorname{tr_{dR,LLZ}}$ with $\operatorname{tr_{dR}}$. Both constructions rely on the Serre duality for coherent sheaves on rigid analytic varieties, which was studied in \cite{beyer1997serreduality}. We begin with the definition of algebraic local cohomology for local rings. Let $(R,\mathfrak m)$ be a noetherian local ring, $M$ be a finitely generated $R$-module, and $\widetilde M$ be its associated coherent $\Spec (R)$-module. Denote by $X:=\Spec (R)$ and $U:=X-\{\mathfrak m\}.$ Consider the functor $$\Gamma_{\mathfrak m}(M):=\{a\in M|\exists t \in \N \text{ such that } \mathfrak m^t a=0\}.$$ Then the local cohomology of $M$ with support in $\mathfrak m$ is defined to be $$H^i_{\mathfrak m}(M):=H^i(R\Gamma_{\mathfrak m}(M)).$$
		
		Now we define the canonical residue map. Let $R=L[[x_1,...,x_d]]$ be the ring of formal power series over a $p$-adic field $L$, and $\mathfrak m=(x_1,...,x_d).$ Consider $\Omega_{R/L}^{\bullet}$ the algebra of differential forms of $R/L$, then $\Omega_{R/L}^d \simeq L[[x_1,...,x_d]]dx_1\wedge dx_2 \wedge ... \wedge dx_d.$ By \cite{hartshorne1966residue}, we have $$H^d_{\mathfrak m}(\Omega_{R/L}^d)\simeq \left\{\sum_{\alpha}a_{\alpha}x^{\alpha}dx_1\wedge dx_2 \wedge ... \wedge dx_d, \alpha=(\alpha_1,...,\alpha_d)\in \Z_{< 0}^d,a_{\alpha}\in L\right\},$$ where the sum is finite. The residue map $$\operatorname{res}_{(x)}:H^d_{\mathfrak m}(\Omega_{R/L}^d) \to L$$ is then given by $$\sum_{\alpha}a_{\alpha}x^{\alpha}dx_1\wedge dx_2 \wedge ... \wedge dx_d \mapsto a_{(-1,-1,...,-1)}.$$ By \cite[Proposition 3.2.1]{beyer1997serreduality}, the residue map is independent of the choice of coordinate $(x).$ 
		
		In general, if $R$ is a complete regular local ring over $L$ of dimension $d$, with maximal ideal $\mathfrak m,$ such that $L'=R/\mathfrak m$ is a finite field extension of $L$. By choosing a regular system of parameters $(x)$ of $R,$ we have $R\simeq L'[[x_1,...,x_d]].$ We can define the residue map $$\operatorname{res}_{(x)}:H^d_{\mathfrak m}(\Omega_{R/L}^d) \to L$$ by composing the residue map above with the canonical trace map of finite field extension $L'/L.$ This is also independent of the choice of coordinate $(x).$ Therefore, we will denote it by $\operatorname{res}_{\mathfrak m}.$
		
		Now let $Y$ be a rigid analytic variety over $L$, $\f$ a coherent sheaf on $Y$. Let $y\in Y$ be a classical point, i.e. defined by a finite field extension of $L$, and $\mathfrak m_y$ be the maximal ideal of the stalk $\mathcal O_{Y,y}$. Denote by $(R,\mathfrak m, M)$ be the $\mathfrak m_y$-adic completion of $(\mathcal O_{Y,y},\mathfrak m_y,\f_y).$ According to \cite[Lemma 4.1.1]{llz2023derham}, there exists a canonical map 
		\begin{equation} \label{pretrace}
			H^d_{\mathfrak m}(\Omega_{R/L}^d) \to H^d(Y,\f).
		\end{equation}
		If moreover $Y$ is proper smooth, then there exists a coherent trace map $$\operatorname{tr_{coh}}:H^d(Y,\f)\to L,$$ whose pre-composition with (\ref{pretrace}), is the canonical residue map for any classical point $y$. $\operatorname{tr_{coh}}$ is an isomorphism if $Y$ is geometrically connected. The coherent trace map allows us to deduce the Serre duality \cite{beyer1997serreduality} for coherent complex on proper smooth rigid analytic varieties. 
		
		${\operatorname{tr_{dR,LLZ}}}$ is then defined by the composition $$H^{2d}_{\operatorname{dR,LLZ,c}}(Y)\simeq H^{d,d}_{\operatorname{Hodge,c}}(Y,\mathcal O_Y):=H^d(X,\mathcal I \otimes_{O_X}\Omega^{d,\log}_X) \simeq H^d(X,\Omega^d_X)\to L,$$ where the first isomorphism follows from the degeneration of the Hodge–de Rham spectral sequence $$E_1^{a,i-a}:=H^{a,i-a}_{\operatorname{Hodge,c}}(Y,\mathcal O_Y)\Longrightarrow H^a_{\operatorname{dr,LLZ,c}}(Y)$$ on the $E_1$ page by \cite[Theorem 3.1.10]{llz2023derham}. ${\operatorname{tr_{dR,LLZ}}}$ can be extended on derived level.
		
		In \cite{AGNcompact}, ${\operatorname{tr_{dR}}}$ is defined for $Y$ by the cosheaf condition: they construct ${\operatorname{tr_{dR}}}$ for Stein spaces first, by using the coherent trace map for Stein spaces constructed in \cite{van1992serreduality}. Note that the coherent trace map for Stein spaces in \cite{van1992serreduality} is compatible with the one in \cite{beyer1997serreduality}: it suffices to check it for open disc, which follows from \cite[Theorem 3.7]{van1992serreduality} and a direct computation. In particular, we have

		\begin{lem}
			Via the natural quasi-isomorphism $\rg_{\operatorname{dr,LLZ,c}}(Y) \xrightarrow{\simeq} \drc(Y),$ we have $$\operatorname{tr_{dR,LLZ}}=\operatorname{tr_{dR}},$$ which is also Galois equivariant.
		\end{lem}
		
		\begin{proof}
			Since $\operatorname{tr_{dR}}$ is compatible with open immersions, and $\operatorname{tr_{dR,LLZ}}$ has similar compatibility by \cite[Lemma 4.2.12]{llz2023derham}, we may assume $Y$ is proper. In this case, both are just the coherent trace map $\operatorname{tr_{coh}}:H^d(Y,\Omega^d_{Y/K})\to K$ via the canonical isomorphism $H^d(Y,\Omega^d_{Y/K}) \xrightarrow{\simeq} H^d_{\operatorname{dR}}(Y).$
			
			To see that the de Rham trace map is Galois equivariant, we need to show the square
			$$\begin{tikzcd}
				H^{d} (Y_C,\Omega^d_{Y_C}) \arrow[r,"\operatorname{tr_{coh}}"] \arrow[d] & C \arrow[d,"\sigma"] \\
				H^{d} (Y_C^{\sigma},\Omega^d_{Y_C^{\sigma}}) \arrow[r,"\operatorname{tr_{coh}}"] & C.
			\end{tikzcd}$$
			is commutative, where $\sigma \in \Gal(C/K)$. In fact, according to \cite{beyer1997serreduality}, we can take a point $y \in Y$ and the trace map $\operatorname{tr}:H^d_y(\Omega^d_{Y_C}) \to C$ is surjective which can be factorized as $$H^d_y(\Omega^d_{Y_C})  \to H^{d} (Y_C,\Omega^d_{Y_C}) \xrightarrow{\operatorname{tr_{coh}}} C.$$ Since $\operatorname{tr_{coh}}$ is an isomorphism, we are reduced to showing that the square
			$$\begin{tikzcd}
				H^d_y(\Omega^d_{Y_C}) \arrow[r,"\operatorname{tr}"] \arrow[d] & L \arrow[d,"\sigma"] \\
				H^{d}_{y^{\sigma}} (\Omega^d_{Y_C^{\sigma}}) \arrow[r,"\operatorname{tr}"] & L
			\end{tikzcd}$$
			is commutative, where $y^{\sigma}$ is the preimage of $y$ of $Y_C^{\sigma} \to Y_C$. This can be deduced from the construction of the residue map for local cohomology.
		\end{proof}
		
		In particular, we may safely use all propositions in \cite{llz2023derham} when $Y$ is an almost proper rigid analytic variety over $K$.

		The Poincar\'e duality for Hyodo-Kato cohomology of open rigid analytic varieties is formulated as follows, which can be deduced from the Poincar\'e duality for de Rham cohomolgy.
		
		\begin{theorem} \cite[Theorem 5.34]{AGNcompact} \label{hkrigidpoincare}
			Let $Y$ be 
			
			(1) a partially proper smooth rigid analytic variety or a quasi-compact smooth dagger variety over $L=C$ of dimension $d$, and denote $L_F=F^{\operatorname{nr}}$; or
			
			(2) a proper smooth rigid analytic variety that is defined over $L=K$ and admits a proper semistable formal model over $\Spf(\mathcal O_K)$, and denote $L_F=F$. 
			
			Then:
			
			(i) There is a natural trace map in $(\phi,N)$ solid $L_F$-vector spaces: $${\operatorname{tr_{HK}}}:\hkc(Y)[2d] \to L_F\{-d\},$$ compatible with the Hyodo-Kato morphism.
			
			(ii) The pairing $$\hk(Y) \otimes^{L_{\blacksquare}}\hkc(Y)[2d] \xrightarrow{} \hkc(Y)[2d]\to L_F\{-d\}$$ is a perfect duality of $(\phi,N)$ solid $L_F$-vector spaces, i.e., we have induced isomorphisms in $D(\Mod_{L_F}^{\operatorname{solid}}):$
			$$H^i_{\operatorname{HK}}(Y) \simeq \operatorname{Hom}_{L_F}(H^{2d-i}_{\operatorname{HK,c}}(Y),L_F\{-d\}),$$
			$$H^i_{\operatorname{HK,c}}(Y) \simeq \operatorname{Hom}_{L_F}(H^{2d-i}_{\operatorname{HK}}(Y),L_F\{-d\}).$$
		\end{theorem}

		\begin{rem} \label{remtrace}
			Similar to the Poincar\'e duality for de Rham cohomology, the trace map ${\operatorname{tr_{HK}}}$ is compatible with restrictions to open dagger subvarieties.
		\end{rem}
		
		\begin{proof}
			See \cite[Theorem 5.34]{AGNcompact}. The proof for case (2) is the same, since we have the Hyodo-Kato isomorphism in this case. 
		\end{proof}
	
		The open-closed exact sequence, combined with the Poincar\'e duality, gives us the exact Gysin sequence.
		
		\begin{prop} \label{gysin}
			If $X$ is a smooth dagger variety over $K$ or $C$, and $Z \subset X$ is a smooth closed subvariety of pure codimension $i$ which is nowhere dense, denote by $U:=X-Z.$ Then we have $$\dr(Z)[-2i]\simeq [\dr(X) \to \dr(U)],$$$$\hk(Z)\{-i\}[-2i]\simeq [\hk(X) \to \hk(U)].$$
		\end{prop}
		
		\begin{proof}
			We prove the Gysin sequence for Hyodo-Kato cohomology, the proof for de Rham cohomology is the same (also proved in \cite[Proposition 2.5]{grosse2004derham}). We assume $X$ is defined over $C$ first. We may assume that $X$ is quasi-compact. Applying the Poincar\'e duality (Theorem \ref{hkrigidpoincare}) to the open-closed exact sequence: $$\hkc(U)\simeq  [\hk(X) \xrightarrow{} \hk(Z)],$$ we get $$\hk(Z)\{-i\}[-2i]\simeq [\hk(X) \to \hk(U)].$$ Remark \ref{remtrace} ensures the induced map $\hk(X) \to \hk(U)$ is exactly the natural map induced by the open immersion $U \hookrightarrow X$. 
			
			The case that $X$ is defined over $K$ then follows by Galois descent.
		\end{proof}
		
		\begin{rem}
			The Gysin sequence is also true for \'etale cohomology, which is implied by cohomological purity \cite[3.9]{huber2013etale}.
		\end{rem}

		
		\section{Applications}
		
		In this section, we prove the comparison between algebraic and analytic Hyodo-Kato cohomology and the semistable conjecture for \'etale cohomology of almost proper rigid analytic varieties. We will also show that Tsuji’s compactly supported log-crystalline cohomology (after inverting $p$) agrees with our definition of compactly supported Hyodo-Kato cohomology.
		
		\subsection{Algebraic and analytic Hyodo-Kato cohomology}
		
		Recall that for an algebraic variety, we can compare the algebraic and analytic de Rham cohomology. This follows from \cite[Theorem 2.3]{grosse2004derham} and Theorem \ref{drgk}.
		
		\begin{thm}
			Let $X$ be an algebraic variety over $L$, and $X^{\operatorname{an}}$ be its analytification. Then there exist a natural quasi-isomorphism $$\dr (X) \xrightarrow{\simeq} \dr(X^{\an}).$$
		\end{thm}	
		
		In this section we will prove the same result for Hyodo-Kato cohomology. Note that for de Rham cohomology, we can construct the comparison map between algebraic and analytic de Rham cohomology immediately from the projection $X_{\an} \to X_{\operatorname{zar}},$ and prove it is an isomorphism by resolution of singularities and the existence of tubular neighborhood after desingularization. For Hyodo-Kato cohomology it is unclear how to define the comparison map. On the other hand, once we are able to construct a map $\hk (X) \xrightarrow{\simeq} \hk(X^{\an}),$ which is compatibility via the Hyodo-Kato isomorphism with its de Rham version, the naturality of this map can be easily deduced from the naturality of its de Rham version. If $X$ is defined over $K$, we can base change it to $C$, and then we use the Galois descent.
		
		\begin{thm} \label{hkgaga}
			Let $X$ be an algebraic variety over $K$, and $X^{\operatorname{an}}$ be its analytification. Then there exist a natural quasi-isomorphism
			$$\hk (X) \xrightarrow{\simeq} \hk(X^{\an}),$$
			which is compatible with Frobenius, monodromy, and the GAGA morphism for de Rham cohomology, i.e. we have the following commutative square:
			$$\begin{tikzcd}
				\hk(X) \arrow[r] \arrow[d,"\simeq"] & \dr(X) \arrow[d,"\simeq"] \\ 
				\hk(X^{\an}) \arrow[r] & \dr(X^{\an}),
			\end{tikzcd}$$
			where the horizontal maps are Hyodo-Kato morphisms.		
		\end{thm}
		
		In geometric case, one needs to moreover take base change over $ \overline{K}\to C $.
		
		\begin{thm}
			Let $X$ be an algebraic variety over $\overline K$, and $X_C^{\operatorname{an}}$ be the analytification of $X_C:=X \times_{\overline K}C$. Then there exist a natural quasi-isomorphism
			$$\hk (X) \xrightarrow{\simeq} \hk(X_C^{\an}),$$
			which is compatible with Galois action (when $X$ is the base change of some $X_0$ over $K$), Frobenius, monodromy, and the GAGA morphism for de Rham cohomology.
		\end{thm}
		
		\subsubsection{The local case}
		
		We prove the case for Beilinson basis first, i.e. we are going to show the following proposition.
		
		\begin{prop}
			Suppose $(U, \overline{U})$ is a strict semistable pair over $K$, as in Definition \ref{modelalgebraic}. Then we have natural (with respect to all strict semistable pairs) quasi-isomorphisms $$\hk (U) \xrightarrow{\simeq} \hk(U^{\an}),$$$$\hk (U_{\overline{K}}) \xrightarrow{\simeq} \hk(U_C^{\an}).$$ Moreover, both are compatible with Galois action, Frobenius, monodromy, and the GAGA morphism for de Rham cohomology.
		\end{prop}
		
		\begin{proof}
			Consider a strict ss-pair $(U, \overline{U})$. Let $Z=\overline{U}_K-U$. The local-global compatibility for algebraic varieties (Proposition \ref{hklocalglobal}) and rigid analytic varieties (Proposition \ref{hkanalyticlocalglobal}) imply 
			
			\begin{lem} \label{gagaprop}
				We have natural quasi-isomorphisms $$\hk(\overline{U}_K) \xrightarrow{\simeq}  \hk(\overline{U}^{\an}_K),  \quad \hk(Z) \xrightarrow{\simeq}  \hk(Z^{\an}),$$
				which are compatible with Galois action, Frobenius, monodromy, and the GAGA morphism for de Rham cohomology.
			\end{lem}
			
			\begin{proof}
				Since the $p$-adic completion of $\overline{U}$ is a semistable formal model of $\overline{U}_K^{\an}$, the first quasi-isomorphism can be defined as the identity. 
				
				For the second isomorphism, according to Lemma \ref{nclemma}, the $p$-adic completion of $\overline{Z}$ is a semistable formal model of $Z^{\an}$ if $Z$ is smooth, therefore we have the natural quasi-isomorphism. 
				
				In general, write $Z=D_1 \cup ... \cup D_n$, where $D_i$ are irreducible components of $Z$. The h-descent for algebraic varieties and \'eh-descent for rigid analytic varieties allow us to reduce to the case that $Z$ is smooth.
				
				To show the compatibility with the GAGA morphism for de Rham cohomology, note that for $\overline{U}$ we have the following diagram
				$$\begin{tikzcd}
					\hk (\overline{U}_{1}^0) \arrow[r, "\simeq"] \arrow[d, equal] & \hk (\overline{U}_K) \arrow[d,"\simeq"] \arrow[r] & \dr(\overline{U}_K) \arrow[d, "\simeq"] \\
					\hk (\overline{U}_{1}^0) \arrow[r, "\simeq"] & \hk (\overline{U}^{\an}_K) \arrow[r] & \dr(\overline{U}^{\an}_K),
				\end{tikzcd}$$
				where the left square is commutative. To show the right square is commutative, it suffices to prove that the big square is commutative. The horizonal maps are given by
				$$\begin{tikzcd}
					\hk (\overline{U}_{1}^0) \arrow[r] \arrow[d, equal] & \hk (\overline{U}_{1}^0) \otimes_FK \arrow[d, equal] \arrow[r, "\iota_{\mathrm{HK}}"] & \cris(\overline{U}_{1}/O_K^{\times})_{\Q_p} \arrow[d, equal] \arrow[r,"\simeq"] & \dr(\overline{U}_K) \arrow[d, "\simeq"] \\
					\hk (\overline{U}_{1}^0) \arrow[r] & \hk (\overline{U}_{1}^0)\otimes_FK \arrow[r, "\iota_{\mathrm{HK}}"] & \cris(\overline{U}_{1}/O_K^{\times})_{\Q_p} \arrow[r,"\simeq"] &  \dr(\overline{U}^{\an}_K),
				\end{tikzcd}$$
				and we are reduced to showing the rightmost square is commutative. In fact, we can prove the integral version of the commutativity of the square, i.e., that the diagram 
				$$\begin{tikzcd}[row sep=tiny]
					& \dr(\overline{U}/O_K^{\times}) \arrow[dd] \\
					\cris(\overline{U}_{1}/O_K^{\times})=\lim_n \cris(\overline{U}_{1}/O_{K,n}^{\times}) \simeq \lim_n \dr(\overline{U}_{1}/O_{K,n}^{\times}) \arrow[ur] \arrow[dr] & \\
					& \dr(\overline{\mathcal{U}}/O_K^{\times})
				\end{tikzcd}$$
				is commutative, where $\overline{\mathcal{U}}$ is the $p$-adic completion of $\overline{U}$. Here we use the fact that the rigid generic fiber of $\overline{\mathcal{U}}$ coincidence with $\overline{U}^{\an}_K$. But this is exactly the construction for formal functions theorem, as described in \cite[section 4]{EGA3I}. 
				
				By reducing to the smooth case, the same argument also shows the compatibility for $Z$.
			\end{proof}
			
			By the definition of compactly supported Hyodo-Kato cohomology defined in Definition \ref{hkcalgebraic}, we have 
			$$\hkc(U)\simeq  [\hk(\overline{U}_K) \xrightarrow{} \hk(Z)].$$
			By Proposition \ref{properopenclosed}, we also have
			$$\hkc(U^{\an})\simeq  [\hk(\overline{U}^{\an}_K) \xrightarrow{} \hk(Z^{\an})].$$
			
			Therefore, the above lemma shows we have the natural quasi-isomorphism
			$$\hkc(U) \xrightarrow{\simeq} \hkc(U^{\an}),$$
			the same argument also shows $\hkc(U_{\overline{K}}) \xrightarrow{\simeq} \hkc(U^{\an}_C)$, which is compatible with the Hyodo-Kato morphism for compactly supported cohomology. Since $U$ is smooth, we can construct the quasi-isomorphism $\hk(U_{\overline{K}}) \xrightarrow{\simeq} \hk(U^{\an}_C)$ as the dotted arrow as follows, by using the Poincar\'e duality:
			$$\begin{tikzcd}
				\hk(U_{\overline{K}}) \arrow[r,"\simeq"] \arrow[d, dashed] & \hkc(U_{\overline{K}})^\vee  \\
				\hk(U^{\an}_C) \arrow[r,"\simeq"]& \hkc(U^{\an}_C)^\vee \arrow[u,"\simeq"].
			\end{tikzcd}$$		
			Here the right vertical map makes sense as both the algebraic and analytic Hyodo-Kato trace map are compatible with the de Rham trace map (under the Hyodo-Kato morphism), and we know that the algebraic and analytic de Rham trace maps are compatible with GAGA morphism (\cite[Lemma 5.4]{llz2023derham}). 
			
			This also implies the isomorphism $\hk(U_{\overline{K}}) \xrightarrow{\simeq} \hk(U^{\an}_C)$ is compatible with the Hyodo-Kato morphism, as we have the follow commutative diagram:
			$$\begin{tikzcd}
				\hk(U_{\overline{K}})\otimes_{F^{\mathrm{nr}}}C \arrow[r,"\simeq"] \arrow[d, "\simeq"] & \hkc(U_{\overline{K}})^\vee\otimes_{F^{\mathrm{nr}}}C & \drc(U_C)^\vee  \arrow[l,"\simeq"] & \dr(U_C) \arrow[d, "\simeq"] \arrow[l,"\simeq"]\\
				\hk(U^{\an}_C)\otimes_{F^{\mathrm{nr}}}C \arrow[r,"\simeq"]& \hkc(U^{\an}_C)^\vee \otimes_{F^{\mathrm{nr}}}C \arrow[u,"\simeq"] &\drc(U^{\an}_C)^\vee \arrow[l,"\simeq"]\arrow[u,"\simeq"] & \dr(U^{\an}_C) \arrow[l,"\simeq"].
			\end{tikzcd}$$	
			
			Since the Hyodo-Kato cohomology satisfies Galois
			descent, after taking Galois invariants, we have
			$$\hk(U) \simeq \hk(U^{\an}),$$
			which is also compatible with the Hyodo-Kato morphism.
			
			Finally, our construction is functorial, i.e., for a morphism of semistable pairs $(U', \overline{U'}) \to (U, \overline{U})$, where $(U, \overline{U})$ is a semistable pair over $K$ and $(U', \overline{U'})$ is a semistable pair over $K'$, we have a commutative square
			$$\begin{tikzcd}
				\hk(U) \arrow[r,"\simeq"] \arrow[d] & \hk(U^{\an}) \arrow[d] \\
				\hk(U') \arrow[r,"\simeq"]& \hk(U'^{\an}),
			\end{tikzcd}$$
			and a similar commutative square over $\overline{K}.$ To see this, we may assume $K'$ is a finite extension of $K$ with residue field $k'$, and $F'$ is the fraction field of $W(k').$ Then after tensoring $K'$ over $F$, by using Hyodo-Kato isomorphism it suffices to check that the square 
			$$\begin{tikzcd}
				\dr(U_{K'}) \arrow[r,"\simeq"] \arrow[d] & \dr(U^{\an}_{K'}) \arrow[d] \\
				\dr(U') \otimes_{F}F' \arrow[r,"\simeq"]& \dr(U'^{\an}) \otimes_{F}F',
			\end{tikzcd}$$
			commutes, which is clear as the de Rham GAGA morphism is functorial.
		\end{proof}	
		
		\subsubsection{The global case}
		
		For a general algebraic variety $X$, take a h-cover $(U^{\bullet},\overline{U}^{\bullet}) \to X.$ Since we have proved the GAGA for a local model, we have 
		$$\begin{tikzcd}
			\hk(X) \arrow[r,"\simeq"] \arrow[d, dashed] & \hk(U^{\bullet}) \arrow[d,"\simeq"] \\
			\hk(X^{\an}) \arrow[r,"\psi"]& \hk(U^{\an,\bullet}).
		\end{tikzcd}$$
		In order to constructed the dotted arrow, it suffices to show that the bottom arrow is a quasi-isomorphism. In general we don't know if we have h-descent for rigid analytic varieties. But thanks to the Hyodo-Kato morphism and GAGA for de Rham cohomology, this is not a problem for us: if $X$ is defined over $\overline K$, after applying $\otimes_{F^{\operatorname{nr}}}C$ to the above square, we get a commutative diagram 
		$$\begin{tikzcd}
			\dr(X_C) \arrow[r,"\simeq"] \arrow[d, "\simeq"] & \dr(U_C^{\bullet}) \arrow[d,"\simeq"] \\
			\dr(X_C^{\an}) \arrow[r,"\phi"]& \dr(U_C^{\an,\bullet}),
		\end{tikzcd}$$
		where the vertical maps are given by the GAGA for algebraic and analytic de Rham cohomology. Therefore $\phi$ is a quasi-isomorphism, since $\psi\otimes_{F^{\operatorname{nr}}}C$ is also a quasi-isomorphism. Therefore we can construct a quasi-isomorphism $\hk (X) \xrightarrow{\simeq} \hk(X_C^{\an}).$ This quasi-isomorphism is compatible with Galois action, Frobenius, monodromy, and the GAGA morphism for de Rham cohomology by our construction and Proposition \ref{gagaprop}.
		
		When $X$ is defined over $K$, we can use the same construction as above: with the same notations, note that $\psi$ is still a quasi-isomorphism by Galois descent.
		
		
		\subsubsection{Algebraic and analytic cohomology with compact support}
		
		We prove the GAGA for de Rham and Hyodo-Kato cohomology with compact support.
		
		\begin{thm} \label{gagacdr}
			Let $X$ be an algebraic variety over $C$, and $X^{\operatorname{an}}$ be its analytification. Then we have a natural quasi-isomorphism
			$$\drc (X) \xrightarrow{\simeq} \drc(X^{\an}),$$
			which is compatible with Galois action. Moreover, the morphism makes the following square commutes:
			\begin{equation} \label{gagadrc}
				\begin{tikzcd}
					\drc(X) \arrow[r,"\simeq"] \arrow[d] & \drc(X^{\an}) \arrow[d] \\
					\dr(X) \arrow[r,"\simeq"]& \dr(X^{\an}).
				\end{tikzcd}
			\end{equation}			 
		\end{thm}
		
		\begin{proof}
			When $X$ is smooth, the theorem follows from GAGA for usual cohomology and Poincaré duality, i.e., $\drc (X) \xrightarrow{\simeq} \drc(X^{\an})$ is defined as the dotted arrow of the following commutative square:
			$$\begin{tikzcd}
				\drc(X) \arrow[r,"\simeq"] \arrow[d, dashed] & \dr(X)^\vee  \\
				\drc(X^{\an}) \arrow[r,"\simeq"]& \dr(X^{\an})^\vee \arrow[u,"\simeq"].
			\end{tikzcd}$$
			To see the square (\ref{gagadrc}) is commutative, by Hironaka's resolution of singularities (\cite{hironaka1964resolution}), we may assume $X$ admits a smooth compactification $\overline{X}$, such that $\overline{X}-X$ is a strictly normal crossing divisor. Then the commutativity of (\ref{gagadrc}) follows directly from the GAGA for coherent sheaves (\cite{kopf1974gaga}), and another interpretation of compactly supported de Rham cohomology as presented in \cite[Definition 3.1.1]{llz2023derham}.
			
			In general, we use induction on the dimension of $X$: when the dimension of $X$ is 0, the theorem is clear. In general, by Hironaka's resolution of singularities (\cite{hironaka1964resolution}), there exist a composition of finitely many smooth blowups $X_n \to X_{n-1} \to \cdots \to X_1 \to X_0=X,$ such that $X_n$ is smooth. We now show the theorem is true for $X_{n-1}$. Let $Y_{n-1} \hookrightarrow X_{n-1}$ be the blowup center, and $Y_n \hookrightarrow X_n$ be the inverse image. Then by Proposition \ref{mvc} we have $$\drc(X_{n-1})\simeq [\drc(Y_{n-1})\oplus \drc(X_{n})\to \drc(Y_n)].$$ By the induction hypothesis and the fact that $X_n$ is smooth, the theorem holds for $Y_n,Y_{n-1}$ and $X_n$, therefore an easy diagram-chasing shows the theorem also holds for $X_{n-1}$. Then the same argument shows that all $X_i$ are smooth successively, which conclude the proof.
		\end{proof}
		
		\begin{thm} \label{gagachk}
			Let $X$ be an algebraic variety over $\overline{K}$, and $X^{\operatorname{an}}$ be its analytification. Then we have a natural quasi-isomorphism
			$$\hkc (X) \xrightarrow{\simeq} \hkc(X_C^{\an}),$$
			which is compatible with Galois action, Frobenius, monodromy, and the GAGA morphism for de Rham cohomology with compact support. Moreover, the morphism makes the following square commutes:
			\begin{equation} \label{gagac}
				\begin{tikzcd}
					\hkc(X) \arrow[r,"\simeq"] \arrow[d] & \hkc(X_C^{\an}) \arrow[d] \\
					\hk(X) \arrow[r,"\simeq"]& \hk(X_C^{\an}).
				\end{tikzcd}
			\end{equation}			 
		\end{thm}
		
		\begin{proof}
			The same proof as the proof of compactly supported de Rham GAGA. The square \ref{gagac} is commutative, as it is commutative after tensoring $C$ over $F^{\mathrm{nr}}$.
		\end{proof}
		
		\begin{rem}
			Note that our constructions are automatically compatible with the Poincar\'e duality, i.e. when $X$ is smooth, they fit into a commutative diagram
			$$\begin{tikzcd}
				\hk(X) \arrow[d, "\simeq"] \arrow[r,phantom,"\otimes^{L_\blacksquare}_{L_F}" description]  & \hkc(X)[2d]  \arrow[r] \arrow[d, "\simeq"] & \hkc(X)[2d] \arrow[r,"\operatorname{tr}"] \arrow[d,"\simeq"] & L_F \arrow[d, equal] \\
				\hk(X^{\an}) \arrow[r,phantom,"\otimes^{L_\blacksquare}_{L_F}" description] & \hkc(X^{\an})[2d] \arrow[r]  & \hkc(X)[2d] \arrow[r,"\operatorname{tr}"]  & L_F.
			\end{tikzcd}$$
		\end{rem}
		
		\subsection{Semistable conjecture for open varieties}
		
		In this section we will prove the following theorem, which extends the semistable conjecture from proper smooth rigid analytic varieties to almost proper smooth rigid analytic varieties.
		
		\begin{theorem} \label{ss}
			Suppose $X$ is a proper smooth rigid analytic variety over $C$, $Z \subset X$ is a strictly normal crossing divisor, and $U=X-Z$.
			
			(1) We have a $B_{\mathrm{st}}$-linear functorial isomorphism commuting with $\varphi$ and $N$
			$$\alpha^i_{\operatorname{st}}(U):H^i_{\operatorname{\acute{e}t}}(U,\Q_p)\otimes_{\Q_p}{B_{{\rm st}}}\simeq H^i_{\operatorname{HK}}(U)\otimes_{F^{\operatorname{nr}}}{B_{{\rm st}}},$$ that induces a $B_{\mathrm{dR}}$-linear filtered isomorphism $$H^i_{\et}(U,\Q_p) \otimes_{\Q_p}B_{\operatorname{dR}} \simeq H^i_{B_{\operatorname{dR}}^+}(X) \otimes_{B_{\operatorname{dR}}^+}B_{\operatorname{dR}}.$$ Here, the filtration on $H^i_{B_{\operatorname{dR}}^+}(X)$ is defined by $$\fil^{\star}H^i_{B_{\operatorname{dR}}^+}(X):=\operatorname{Im}(H^i(\fil^{\star}\rg_{B_{\operatorname{dR}}^+}(X))\to H^i_{B_{\operatorname{dR}}^+}(X)).$$
			
			(2) Let $i \leq r$. Then we have an exact sequence
			$$0 \xrightarrow{}H^i_{\et}(U,\Q_p(r))\xrightarrow{}(H_{\mathrm{HK}}^i(U) {\otimes}_{F^{\operatorname{nr}}} {B^+_{{\rm st}}})^{N=0, \varphi=p^r}\xrightarrow{}H^i_{B_{\operatorname{dR}}^+}(X)/F^r\xrightarrow{}0.$$ 
			
			Moreover, when $X$ descends to a rigid analytic variety over $K$, statements in (1) and (2) are Galois equivariant.
		\end{theorem}			
		
		\begin{proof}
			
			We will prove the above theorem in following steps. 
			
			(1) Construction of the comparison map. We can define the period morphism $$\alpha_{\operatorname{st}}(U):\rg_{\operatorname{\acute{e}t}}(U,\Q_p)\otimes_{\Q_p}{B_{{\rm st}}}\to \hk(U)\otimes_{F^{\operatorname{nr}}}{B_{{\rm st}}},$$ 
			as follows. For $r >2d$, consider the following composition
			\begin{align*}
				\rg_{\operatorname{\acute{e}t}}(U,\Q_p(r)) & \to \tau^{\leq 2d}\rg_{\proet}(U,\Q_p(r)) \xrightarrow{\simeq} \tau^{\leq 2d}\rg_{\operatorname{syn}}(U,r)\\ & \to [\hk(U) {\otimes}_{F^{\operatorname{nr}}} {B_{{\rm st}}}]^{N=0, \varphi=p^r} \\ & \xrightarrow{p^{-r}} \hk(U) {\otimes}_{F^{\operatorname{nr}}} {B_{{\rm st}}}.
			\end{align*}
			We set $$\alpha_{\operatorname{st}}(U):=t^{-r}\alpha_{\operatorname{st}}(r)\varepsilon^r,$$ where $\varepsilon$ is the generator of $\Z_p(1)$ corresponding to $t$.
			
			Note that $\alpha_{\operatorname{st}}$ is natural with respect to any morphism $V \to U$, as the same claims hold for each maps of the construction of $\alpha_{\operatorname{st}}$.
			
			
			(2) The compatibility with Gysin isomorphism: suppose $Z$ is smooth, we check that the diagram
			\begin{equation}\label{diagram} 
				\begin{tikzcd}
					\rg_{\et} (Z,\Q_p(-1))[-2 ]\otimes_{\Q_p}B_{\operatorname{st}} \arrow[r,"g_{\et}"] \arrow[d,"\alpha_{\operatorname{st}}(Z)(-1){[-2]}"] & \rg_{\et}(X,\Q_p)\otimes_{\Q_p}B_{\operatorname{st}} \arrow[r] \arrow[d,"\alpha_{\operatorname{st}}(X)"] & \rg_{\et}(U,\Q_p)\otimes_{\Q_p}B_{\operatorname{st}} \arrow[d,"\alpha_{\operatorname{st}}(U)"] \\
					\hk (Z)\{-1\}[-2] \otimes_{F^{\operatorname{nr}}}B_{\operatorname{st}} \arrow[r,"g_{\operatorname{HK}}"] & \hk (X) \otimes_{F^{\operatorname{nr}}}B_{\operatorname{st}} \arrow[r] & \hk (U) \otimes_{F^{\operatorname{nr}}}B_{\operatorname{st}},
				\end{tikzcd}
			\end{equation}
				commutes. The right square is commutative, as $\alpha_{\operatorname{st}}$ is natural with respect to the morphism $X \to U$. For the left square, by Theorem \ref{ehcomparison}, we know that $\alpha_{\operatorname{st}}(Z)$ and $\alpha_{\operatorname{st}}(X)$ are quasi-isomorphisms. 
				
				Denote by $PD_{\et}(X)$ and $PD_{\operatorname{HK}}(X)$ the Poincar\'e duality morphisms $$PD_{\et}(X):\rg_{\et}(X,\Q_p)\otimes_{\Q_p}B_{\operatorname{st}} \to R\operatorname{Hom}_{B_{\operatorname{st}}}(\rg_{\et}(X,\Q_p(d))[2d]\otimes_{\Q_p}B_{\operatorname{st}},B_{\operatorname{st}}),$$ $$PD_{\operatorname{HK}}(X):\hk(X)\otimes_{F^{\operatorname{nr}}}B_{\operatorname{st}} \to R\operatorname{Hom}_{B_{\operatorname{st}}}(\hk(X)\{d\}[2d]\otimes_{F^{\operatorname{nr}}}B_{\operatorname{st}},B_{\operatorname{st}}).$$ Here the trace map for \'etale cohomology is just the trace map for pro-\'etale cohomology: this indeed gives a perfect pairing, for example, by the main theorem of \cite{zavyalov2024modppoincaredualitypadic}. $PD_{\et}(X)$ is a quasi-isomorphism, as for any $i$, $$H^i(\rg_{\et}(X,\Q_p(d))[2d]\otimes_{\Q_p}B_{\operatorname{st}})\simeq H^{i-2d}_{\et}(X,\Q_p(d))\otimes_{\Q_p}B_{\operatorname{st}}$$ is a finite free $B_{\operatorname{st}}$-module, therefore for any $j >0,$ $$\operatorname{Ext}^j_{B_{\operatorname{st}}}(H^{i-2d}_{\et}(X,\Q_p(d))\otimes_{\Q_p}B_{\operatorname{st}},B_{\operatorname{st}})=0.$$ The same argument also shows $PD_{\operatorname{HK}}(X)$ is a quasi-isomorphism. Define $PD_{\et}(Z)$ and $PD_{\operatorname{HK}}(Z)$ similarly. Let $i_{\et}$ and $i_{\operatorname{HK}}$ be the canonical maps
				$$i_{\et}:\rg_{\et}(X,\Q_p)\otimes_{\Q_p}B_{\operatorname{st}} \to \rg_{\et}(Z,\Q_p)\otimes_{\Q_p}B_{\operatorname{st}},$$ $$i_{\operatorname{HK}}:\hk(X)\otimes_{F^{\operatorname{nr}}}B_{\operatorname{st}} \to \hk(Z)\otimes_{F^{\operatorname{nr}}}B_{\operatorname{st}}.$$ Then for $*=$HK or \'et, the Gysin morphism $g_*$ is $$g_*:=PD_*(X)^{-1}\circ R\operatorname{Hom}_{B_{\operatorname{st}}}(i_*[2d],B_{\operatorname{st}}) \circ PD_*(Z)[-2].$$ We need to show that $$\alpha_{\operatorname{st}}(X)\circ g_{\et}=g_{\operatorname{HK}} \circ\alpha_{\operatorname{st}}(Z)(-1)[-2].$$ We have the following lemma. 
				
				\begin{lem}
					(1) $\alpha_{\operatorname{st}}(Z)\circ i_{\et}=i_{\operatorname{HK}}\circ \alpha_{\operatorname{st}}(Z).$
					
					(2) $R\operatorname{Hom}_{B_{\operatorname{st}}}(\alpha_{\operatorname{st}}(X)(d)[2d],B_{\operatorname{st}})\circ PD_{\operatorname{HK}}(X) \circ\alpha_{\operatorname{st}}(X)= PD_{\et}(X)$
					
					(3)
					$R\operatorname{Hom}_{B_{\operatorname{st}}}(\alpha_{\operatorname{st}}(Z)(d-1)[2d-2],B_{\operatorname{st}})\circ PD_{\operatorname{HK}}(Z) \circ\alpha_{\operatorname{st}}(Z)= PD_{\et}(Z)$
				\end{lem}
				
				\begin{proof}
					The first lemma follows from the fact that $\alpha_{\operatorname{st}}$ is natural with respect to the morphism $Z \to X$.

					We prove the claim (2) for $X$. The proof of the claim (3) for $Z$ is the same. Recalling the construction of $\alpha_{\operatorname{st}}$, after choosing $r,r'>2d,$ denote $s:=r+r'-d,$ after twisting we need to show the following diagram
					
					$$\begin{tikzcd}
						\rg_{\et}(X,\Q_p(r)) \arrow[d] \arrow[r,phantom,"\otimes^{L_\blacksquare}_{\Q_p}" description]  & \rg_{\et}(X,\Q_p(r'))[2d]  \arrow[r] \arrow[d] & \rg_{\et}(X,\Q_p(r+r'))[2d] \arrow[r,"\operatorname{tr}_{\et}"] \arrow[d] &  \Q_p(s) \arrow[d, equal] \\
						\rg_{\proet}(X,\Q_p(r)) \arrow[r,phantom,"\otimes_{\Q_p}^{L_\blacksquare}" description] & \rg_{\proet}(X,\Q_p(r'))[2d] \arrow[r]  & \rg_{\proet}(X,\Q_p(r+r'))[2d]  \arrow[r,"\operatorname{tr}_{\proet}"]  &  \Q_p(s) \arrow[d,equal] \\ 
						\syn(X,r) \arrow[u,"\alpha_{r}"]\arrow[d]\arrow[r,phantom,"\otimes_{\Q_p}^{L_\blacksquare}" description] & \syn(X,r')[2d] \arrow[u,"\alpha_{r'}"]\arrow[d]\arrow[r]  & \syn(X,r+r')[2d] \arrow[u,"\alpha_{r'+r}"]\arrow[d]\arrow[r,"\operatorname{tr_{syn}}"] &  \Q_p(s)\arrow[d] \\
						\hk(X)\{r\}\otimes B_{\operatorname{st}} \arrow[r,phantom,"\otimes_{B_{\mathrm{st}}}^{L_\blacksquare}" description] & \hk(X)\{r'\}\otimes B_{\operatorname{st}}[2d] \arrow[r]  & \hk(X)\{r+r'\}\otimes B_{\operatorname{st}}[2d] \arrow[r,"\operatorname{tr_{HK}}\otimes{B_{\mathrm{st}}}"] &  B_{\operatorname{st}}\{s\}.
					\end{tikzcd}$$
					is commutative, after taking truncation $\tau^{\leq 2d}$. Here the $\alpha$'s are period morphisms from \cite[6.9]{CN4.3}.
					
					We  only need to consider the pairings. 
					The above two rows are naturally compatible. The bottom two rows are commutative because the product of syntomic cohomology is defined by the mapping fiber. The middle two rows are commutes because the classical comparison maps are compatible with products. 
				\end{proof}
				
				Now we can calculate
				\begin{align*}
					\alpha_{\operatorname{st}}(X)\circ g_{\et} &= \alpha_{\operatorname{st}}(X)\circ PD_{\et}(X)^{-1}\circ R\operatorname{Hom}_{B_{\operatorname{st}}}(i_{\et}[2d],B_{\operatorname{st}}) \circ PD_{\et}(Z)[-2] \\
					&= PD^{-1}_{\operatorname{HK}}(X) \circ R\operatorname{Hom}_{B_{\operatorname{st}}}(\alpha_{\operatorname{st}}^{-1}(X)(d)[2d],B_{\operatorname{st}})\circ  R\operatorname{Hom}_{B_{\operatorname{st}}}(i_{\et}[2d],B_{\operatorname{st}}) \circ PD_{\et}(Z)[-2] \\
					&=PD^{-1}_{\operatorname{HK}}(X) \circ R\operatorname{Hom}_{B_{\operatorname{st}}}(i_{\et}[2d] \circ \alpha_{\operatorname{st}}^{-1}(X)(d)[2d],B_{\operatorname{st}}) \circ PD_{\et}(Z)[-2] \\
					&=PD^{-1}_{\operatorname{HK}}(X) \circ R\operatorname{Hom}_{B_{\operatorname{st}}}( \alpha_{\operatorname{st}}^{-1}(Z)(d)[2d]\circ i_{\operatorname{HK}}[2d] ,B_{\operatorname{st}}) \circ PD_{\et}(Z)[-2] \\
					&=PD^{-1}_{\operatorname{HK}}(X) \circ R\operatorname{Hom}_{B_{\operatorname{st}}}(i_{\operatorname{HK}}[2d],B_{\operatorname{st}}) \circ R\operatorname{Hom}_{B_{\operatorname{st}}}(\alpha_{\operatorname{st}}^{-1}(Z)(d)[2d],B_{\operatorname{st}})\circ    PD_{\et}(Z)[-2] \\
					&= PD^{-1}_{\operatorname{HK}}(X) \circ R\operatorname{Hom}_{B_{\operatorname{st}}}(i_{\operatorname{HK}}[2d],B_{\operatorname{st}}) \circ PD_{\operatorname{HK}}(Z)[-2] \circ\alpha_{\operatorname{st}}(Z)(-1)[-2] \\
					&=g_{\operatorname{HK}} \circ\alpha_{\operatorname{st}}(Z)(-1)[-2].
				\end{align*}
				This shows the compatibility with Gysin isomorphism.	
				
				(3) $\alpha_{\operatorname{st}}$ is a natural quasi-isomorphism compatible with Gysin morphisms. We assume first that $Z$ is smooth. From (2) we know that $\alpha_{\operatorname{st}}(Z)$ and $\alpha_{\operatorname{st}}(X)$ are isomorphisms compatible with Gysin morphisms.
				Therefore the commutative diagram in step (2) shows that we have a natural quasi-isomorphism:
				$$\alpha_{\mathrm{st}}(U):\rg_{\et} (U,\Q_p)\otimes_{\Q_p}B_{\operatorname{st}} \xrightarrow{\simeq} \hk (U) \otimes_{F^{\operatorname{nr}}}B_{\operatorname{st}}.$$
				
				In general, we can write $Z$ as the union of irreducible components $Z=Z_1 \cup ... \cup Z_m$ where $Z_1,...,Z_m$ are smooth. Write $D_j:=\bigcup_{i \leq j}Z_j.$ We have a map of distinguished triangles:	
				$$\begin{tikzcd}
					\rg_{\et} (Z_m-Z_m\cap D_{m-1},\Q_p(-1))[-2]\otimes_{\Q_p}B_{\operatorname{st}} \arrow[r] \arrow[d, "\alpha_{\operatorname{st}}(Z-Z\cap D_{m-1})(-1){[-2]}"] & \rg_{\et}(X-D_{m-1},\Q_p)\otimes_{\Q_p}B_{\operatorname{st}} \arrow[r] \arrow[d, "\alpha_{\operatorname{st}}(X-D_{m-1})"] & \rg_{\et}(U,\Q_p)\otimes_{\Q_p}B_{\operatorname{st}} \arrow[d,"\alpha_{\operatorname{st}}(U)"] \\
					\hk (Z_m-Z_m\cap D_{m-1})\{-1\}[-2] \otimes_{F^{\operatorname{nr}}}B_{\operatorname{st}} \arrow[r] & \hk (X-D_{m-1}) \otimes_{F^{\operatorname{nr}}}B_{\operatorname{st}} \arrow[r] & \hk (U) \otimes_{F^{\operatorname{nr}}}B_{\operatorname{st}},
				\end{tikzcd}$$
				We need to check that the left square is commutative. We use induction on the number of irreducible divisors $m$. Suppose that the the left square is commutative for any proper smooth rigid analytic varieties over $C$ with at most $m-1$ irreducible divisors. Since the Gysin map is functorial, we have the commutative diagram
				$$\begin{tikzcd}
					\rg_{\et} (Z_{m}\cap Z_{m-1}-Z_{m}\cap Z_{m-1}\cap D_{m-2},\Q_p(-2))[-4] \arrow[r] \arrow[d] & \rg_{\et}(Z_{m-1}-Z_{m-1}\cap D_{m-2},\Q_p(-1))[-2]  \arrow[d ]  \\
					\rg_{\et}(Z_{m}-Z_{m-1}\cap D_{m-2},\Q_p(-1))[-2]  \arrow[r] \arrow[d] & \rg_{\et}(X-D_{m-2},\Q_p) \arrow[d]  \\
					\rg_{\et} (Z_m-Z_m\cap D_{m-1},\Q_p(-1))[-2] \arrow[r] &\rg_{\et}(X-D_{m-1},\Q_p)
					,
				\end{tikzcd}$$
				where the horizontal rows (and the first vertical map) are the Gysin maps. We also have the similar commutative diagram for Hyodo-Kato cohomology. Therefore it suffices to show that the diagrams
				$$\begin{tikzcd}
					\rg_{\et} (Z_{m-1}-Z_{m-1}\cap D_{m-2},\Q_p(-1))[-2]\otimes_{\Q_p}B_{\operatorname{st}} \arrow[r] \arrow[d] & \rg_{\et}(X-D_{m-2},\Q_p)\otimes_{\Q_p}B_{\operatorname{st}}  \arrow[d]  \\
					\hk (Z_{m-1}-Z_{m-1}\cap D_{m-2})\{-1\}[-2] \otimes_{F^{\operatorname{nr}}}B_{\operatorname{st}} \arrow[r] & \hk (X-D_{m-2}) \otimes_{F^{\operatorname{nr}}}B_{\operatorname{st}} 
				\end{tikzcd}$$
				and
				$$\begin{tikzcd}
					\rg_{\et} (Z_{m}\cap Z_{m-1}-Z_{m}\cap Z_{m-1}\cap D_{m-2},\Q_p(-1))[-2]\otimes_{\Q_p}B_{\operatorname{st}} \arrow[r] \arrow[d] & \rg_{\et}(Z_{m-1}-Z_{m-1}\cap D_{m-2},\Q_p)\otimes_{\Q_p}B_{\operatorname{st}} \arrow[d] \\
					\hk (Z_{m}\cap Z_{m-1}-Z_{m}\cap Z_{m-1}\cap D_{m-2})\{-1\}[-2] \otimes_{F^{\operatorname{nr}}}B_{\operatorname{st}} \arrow[r] & \hk (Z_{m-1}-Z_{m-1}\cap D_{m-2}) \otimes_{F^{\operatorname{nr}}}B_{\operatorname{st}}
				\end{tikzcd}$$
				are commutative. In both cases we reduce the number of irreducible divisors to $m-1$, then the induction hypothesis applies. Therefore we have a natural quasi-isomorphism:
				$$\alpha_{\mathrm{st}}(U):\rg_{\et} (U,\Q_p)\otimes_{\Q_p}B_{\operatorname{st}} \xrightarrow{\simeq} \hk (U) \otimes_{F^{\operatorname{nr}}}B_{\operatorname{st}}.$$	
				
				(4) The short exact sequence. The following theorem (see \cite[Remark 5.16]{CN5}) is an extension of the standard results for admissible filtered $(\varphi,N)$-modules, e.g., \cite[Proposition 5.3]{colmezfontaine2000ss} or \cite[Chapter 10]{fargues2018courbe}.
				
				\begin{thm}
					Suppose that $(D,D^+_{\operatorname{dR}})$ is an acyclic filtered $(\varphi,N)$-module over $C$ (see \cite[Definition 5.3.1]{CN5} for the definition) with $\varphi$-slopes in $[0,r]$, and $D\otimes t^rB^+_{\operatorname{dR}} \subset t^rD^+_{\operatorname{dR}} \subset D\otimes B^+_{\operatorname{dR}}.$ Then we have the following short exact sequence: $$0 \to t^rV_{\operatorname{st}}(D,D^+_{\operatorname{dR}}) \to (D \otimes_{F^{\operatorname{nr}}}B^+_{\operatorname{st}})^{N=0,\varphi=p^r} \to (D \otimes_{F^{\operatorname{nr}}}B^+_{\operatorname{dR}})/t^rD^+_{\operatorname{dR}} \to 0,$$
					where $$V_{\operatorname{st}}(D,D^+_{\operatorname{dR}}):=\operatorname{Ker}((D \otimes_{F^{\operatorname{nr}}}B_{\operatorname{st}})^{N=0,\varphi=1} \to (D \otimes_{F^{\operatorname{nr}}}B_{\operatorname{dR}})/D^+_{\operatorname{dR}}).$$
				\end{thm}
				
				Here we take $(D,D^+_{\operatorname{dR}}):=(H^i_{\operatorname{HK}}(U),H^i_{\et}(U,\Q_p)\otimes_{\Q_p}B_{\operatorname{dR}}^+)$. Then for $i \leq r$ we have $$D\otimes_{F^{\operatorname{nr}}}B_{\operatorname{dR}}^+\simeq H^i_{\operatorname{dR}}(U/B^+_{\operatorname{dR}}), \quad t^rD^+_{\operatorname{dR}}=\fil^r(H^i_{\et}(U,\Q_p)\otimes_{\Q_p}B_{\operatorname{dR}}) \simeq \fil^r H^i_{B_{\operatorname{dR}}^+}(X),$$ where the last isomorphism follows from Theorem \ref{logcdr}, and $$(D \otimes_{F^{\operatorname{nr}}}B^+_{\operatorname{dR}})/t^rD^+_{\operatorname{dR}} \simeq H^i_{\operatorname{dR}}(U/B^+_{\operatorname{dR}})/t^rD^+_{\operatorname{dR}}\simeq H^i_{B_{\operatorname{dR}}^+}(X)/F^r,$$ where the last isomorphism follows from \cite[Proposition 5.17]{xsperiod}. By using the semistable and geometric log de Rham comparison (Theorem \ref{logcdr}), and the fundamental exact sequence, we have $$V_{\operatorname{st}}(D,D^+_{\operatorname{dR}})\simeq H^i_{\et}(U,\Q_p),$$ and $(D,D^+_{\operatorname{dR}})$ is a weakly admissible filtered $(\varphi,N)$-module over $C$ (see \cite[Remark 5.14]{CN5}). The $\varphi$-slopes of $H^i_{\operatorname{HK}}(U)$ are in $[0,i]$, since this is true for quasi-compact dagger varieties, and we have the Gysin sequence. Then for $i \leq r$ the above theorem shows that we have the following short exact sequence: $$0 \xrightarrow{}H^i_{\et}(U,\Q_p(r))\xrightarrow{}(H_{\mathrm{HK}}^i(U) {\otimes}_{F^{\operatorname{nr}}} {B^+_{{\rm st}}})^{N=0, \varphi=p^r}\xrightarrow{}H^i_{B_{\operatorname{dR}}^+}(X)/F^r\xrightarrow{}0.$$
				
				(5) The short exact sequence is Galois equivariant. Suppose that $X,D,U$ descend to $X_0,Z_0,U_0$ over $K$ respectively. We need to show that the morphism $\alpha_{\operatorname{st}}(U)\otimes_{B_{\operatorname{st}}}B_{\operatorname{dR}}$ is the same as the period morphism in Theorem \ref{logcdr}. Suppose first that $Z_0$ is smooth. After tensoring $B_{\operatorname{dR}}$ to the commutative diagram \ref{diagram}, we have the commutative diagram 
				$$\begin{tikzcd}
					\rg_{\et} (Z,\Q_p(-1))[-2]\otimes_{\Q_p}B_{\operatorname{dR}} \arrow[r] \arrow[d,"\alpha_{\operatorname{st}}(Z)(-1){[-2]}\otimes_{B_{\operatorname{st}}}B_{\operatorname{dR}}"] & \rg_{\et}(X,\Q_p)\otimes_{\Q_p}B_{\operatorname{dR}} \arrow[r] \arrow[d, "\alpha_{\operatorname{st}}(X)\otimes_{B_{\operatorname{st}}}B_{\operatorname{dR}}"] & \rg_{\et}(U,\Q_p)\otimes_{\Q_p}B_{\operatorname{dR}} \arrow[d, "\alpha_{\operatorname{st}}(U)\otimes_{B_{\operatorname{st}}}B_{\operatorname{dR}}"] \\
					\rg_{\operatorname{dR}}(Z_0){[-2]} \otimes_KB_{\operatorname{dR}} \arrow[r] & \rg_{\operatorname{dR}}(X_0) \otimes_KB_{\operatorname{dR}} \arrow[r] & \rg_{\operatorname{dR}}(U_0) \otimes_KB_{\operatorname{dR}}.
				\end{tikzcd}$$
				Since the period morphism in Theorem \ref{logcdr} also satisfies the above diagram by \cite[Proposition 4.3.17]{llz2023derham}, it suffices to check for $X$ (and $Z$) proper, $\alpha_{\operatorname{st}}(X)\otimes_{B_{\operatorname{st}}}B_{\operatorname{dR}}$ coincidences with the period morphism constructed in \cite{scholze2013p}, which follows from the local computation by Sally Gilles in \cite{gilles2023period}. In general, the same argument as in step (3) allows us to reduce the number of irreducible divisors in $X$.						
			\end{proof}
			
			\begin{rem}
				The proof was originally obtained while the author was working on the GAGA problem for Hyodo-Kato cohomology, and realized that it is also possible to deduce the semistable conjecture for $U$ from the proper case. However, the proof intertwines different constructions of period morphisms and depends on Poincaré duality, which itself presents significant challenges. As mentioned in the introduction, a forthcoming article \cite{xslogsyn} will provide a more conceptual approach by introducing logarithmic syntomic cohomology and extending cohomology groups to the category of vector spaces. This perspective aligns with the methodologies developed in \cite{CNsytntomic} and \cite{CN5}.
			\end{rem}

			\subsection{Comparison with Tsuji’s compactly supported log-crystalline cohomology}
			
			In this section we prove that the compactly supported Hyodo-Kato cohomology defined above agrees with the one previously defined by Tsuji in \cite{tsuji1999poincare}. We first start by recalling the construction of Tsuji.
			
			\subsubsection{Definitions}
			
			We start with the geometry of monoids, which follows from \cite{kato1994toric} and \cite{tsuji1999poincare}.
			
			Let $P$ be a monoid. An ideal $I$ of $P$ is a subset of $P$ satisfying $PI \in I.$ A prime ideal $\mathfrak{p}$ of $P$ is an ideal of $P$ such that $P-\mathfrak{p}$ is a submonoid of $P$. Let $\Spec(P)$ be the set of all primes of $P$. The dimension $\dim(P)$ of $P$ is the maximal length $r$ (if it exists) of a sequence $\mathfrak{p}_0\supsetneq\mathfrak{p}_1\supsetneq\cdot\cdot\mathfrak{p}_r$ (if such a sequence does not exist, we define $\dim(P)=\infty$). The height $\operatorname{ht}(\mathfrak{p})$ of a prime ideal $\mathfrak{p}$ is the maximal length $r$ (if it exists) of a sequence $\mathfrak{p}=\mathfrak{p}_0\supsetneq\mathfrak{p}_1\supsetneq\cdot\cdot\mathfrak{p}_r$ (if such a sequence does not exist, we define $\operatorname{ht}(\mathfrak{p})=\infty$).
			
			For $h:Q\to P$ a morphism of monoids, we say a prime $\mathfrak{p}$ of $P$ is horizontal with respect to $h$ if $h(Q) \subset P-\mathfrak{p}$. For a morphism $f:(X,\mathcal{M})\to (S,\mathcal{N})$ of log schemes and $x \in X$, we say a prime $\mathfrak{p}$ of $\mathcal{M}_{\overline{x}}$ is horizontal with respect to $f$ if it is horizontal with respect to $f^*_{\overline{x}}:\mathcal{N}_{\overline{f(x)}} \to \mathcal{M}_{\overline{x}}.$
			
			Let $f:(X,\mathcal{M}) \to (S,\mathcal{N})$ be a smooth morphism of fs log schemes. Define the sheaf of ideals $\mathcal{I}_f$ of the sheaf of monoids $M$ by
			\begin{equation*}
				\begin{split}
					\Gamma(U,\mathcal{I}_f):=\{a\in\Gamma(U,\mathcal M)|a\in \mathfrak{p} \text{ for all } \mathfrak{p} \in \Spec(\mathcal M_{\overline{x}}) \text{ of height } 1 \\ \text{horizontal with respect to } f \text{ and all } y\in U \}.
				\end{split}
			\end{equation*} 
			The ideal sheaf $\mathcal{I}_f\mathcal{O}_X$ of $\mathcal{O}_X$ is quasi-coherent. If $f:(X,\mathcal{M}) \to (S,\mathcal{N})$ is a log-smooth morphism of fs log schemes with $(S,N)$ log-regular, then we also have $$\Gamma(U,\mathcal{I}_f)=\{a\in\Gamma(U,\mathcal{M})|a\notin \mathcal{O}^*_{X,\overline{x}}\text{ for all } x\in S_U\},$$ where $S_U:=\{x \text{ of codimension 1 such that } \mathcal{M}_{\overline{x}}\neq \mathcal{O}^*_{X,\overline{x}} \text{ and } \mathcal{N}_{\overline{f(x)}}=\mathcal{O}^*_{S,\overline{f(x)}}\}.$ In particular, we have the following description:
			
			\begin{prop}
				Let $\mathcal{I}$ be the ideal sheaf of $\mathcal{O}_X$ associated to the closed subset $\bigcup_{y\in S_X}\overline{\{x\}}$ of $X$. Then $\mathcal{I}\mathcal{O}_X=\mathcal{I}_f.$
			\end{prop}
			
			Let $f:(X,\mathcal{M}) \to (S,\mathcal{N})$ be a log-smooth morphism of fs log-schemes. Let $\Omega_{X/S}^{\bullet}$ be the associated complex of log-differentials. Then the compactly supported (log-)de Rham cohomology of $Y$ is defined by the complex $$\rg_{\operatorname{dR,Tsu,c}}(X/S):=\rg(X,\mathcal{I}_f\mathcal{O}_X\otimes_{\mathcal{O}_X}\Omega_{X/S}^{\bullet}).$$ The compactly supported (log-)de Rham cohomology satisfies Poincar\'e duality.
			
			\begin{thm}\cite[Theorem 3.4]{tsuji1999poincare} \label{tsujidrpoincare}
				Suppose moreover that $X$ is proper and $S$ is  the log scheme $\Spec(\mathcal O_{F,n})$ equipped with the log structure induced by $\N \to \mathcal O_F,1\mapsto a$ for some $a \in \mathcal O_{F,n}$. Then 
				
				(i) There is a natural trace map $${\operatorname{tr_{dR}}}:\rg_{\operatorname{dR,Tsu,c}}(X/S) \to \mathcal O_{F,n}.$$
				
				(ii) The pairing $$H^i_{\operatorname{logdR}}(X/S) \otimes H^{2d-i}_{\operatorname{dR,Tsu,c}}(X/S) \xrightarrow{} H^{2d}_{\operatorname{dR,Tsu,c}}(X/S) \to \mathcal O_{F,n}$$ is perfect.
			\end{thm}
			
			We now review the definition of compactly supported Hyodo-Kato (log-crystalline) cohomology of Tsuji. Let $(S,\mathcal{N})$ (respectively $(S_n,\mathcal{N}_n)$) be the log scheme $\Spec(\mathcal O_F)$ (respectively $\Spec(\mathcal O_{F,n})$) equipped with the log structure induced by $\N \to \mathcal O_F,1\mapsto 0.$ Let $f:(X,\mathcal{M}) \to (S_0,\mathcal{N}_0)$ be a log-smooth and universally saturated morphism of fs log-schemes of relative dimension $d$. Denote by $\gamma$ the canonical PD-structure on the ideal $p\mathcal O_S.$ We will freely use the notations and properties of log-crystalline sites in \cite{kato1989log}.
			
			We write $\mathcal M_{X/S}$ the sheaf on the log-crystalline site $(X/S)_{\operatorname{cris}}$ given by $$\Gamma((U,(T,\mathcal M_T)),\mathcal M_{X/S})=\Gamma(T,\mathcal M_T).$$ Denote by $u:(X/S)_{\operatorname{cris}}\to X_{\et}$ the canonical projection. Define the sheaf of ideals $\mathcal{I}_{X/S}$ of the sheaf of monoids $M$ by
			\begin{equation*}
				\begin{split}
					\Gamma((U,(T,\mathcal M_T)),\mathcal{I}_{X/S}):=\{a\in\Gamma((U,T),\mathcal M)|a\in \mathfrak{p} \text{ for all } x \in T \text{ and all } \mathfrak{p} \in \Spec(\mathcal M_{T,\overline{x}}/O^*_{T,\overline{x}}) \\ \text{ of height } 1  \text{ horizontal with respect to }\mathcal N_{\overline{f(x)}}/\mathcal O^*_{S,\overline{f(x)}}\to \mathcal M_{T,\overline{x}}/O^*_{T,\overline{x}} \}.
				\end{split}
			\end{equation*} 
			We write $\mathcal K_{X/S}$ the ideal $\mathcal{I}_{X/S}\mathcal O_{X/S}$ of $\mathcal O_{X/S}$, it is a crystal by \cite[Lemma 5.3]{tsuji1999poincare}. In particular, if $(X,\mathcal M) \to (Z, \mathcal M_Z)$ is a closed immersion to a smooth scheme over $(S_0,\mathcal{N}_0),$ and $(X^{\operatorname{PD}},\mathcal M^{\operatorname{PD}})$ is the PD-envelop of $(X,\mathcal M)$ in $(Z, \mathcal M_Z)$, then there is a quasi-isomorphism $$Ru_*\mathcal K_{X/S}\simeq \mathcal{I}_{f^{\operatorname{PD}}}\mathcal O_{X^{\operatorname{PD}}} \otimes_{\mathcal O_{X^{\operatorname{PD}}}} \Omega^{\bullet}_{X^{\operatorname{PD}}/S} ,$$ where $f^{\operatorname{PD}}$ is the map $(X^{\operatorname{PD}},\mathcal M^{\operatorname{PD}}) \to (S_0,\mathcal{N}_0).$
			
			The compactly supported crystalline cohomology is defined by
			$$\rg_{\operatorname{cris,Tsu,c}}(X/S):=\varinjlim_n\rg((X/S)_{\operatorname{cris}},\mathcal K_{X/S_n})=\varinjlim_n\rg_{\et}(X,Ru_*\mathcal K_{X/S_n}),$$
			and the compactly supported Hyodo-Kato cohomology is then defined by
			$$\rg_{\operatorname{HK,Tsu,c}}(X/S):=\rg_{\operatorname{cris,Tsu,c}}(X/S)\otimes_{\mathcal O_F}F.$$
			If $(X,\mathcal M) \to (Z, \mathcal M_Z)$ is a closed immersion to a smooth scheme over $(S_0,\mathcal{N}_0),$ and $(X^{\operatorname{PD}},\mathcal M^{\operatorname{PD}})$ is the PD-envelop of $(X,\mathcal M)$ in $(Z, \mathcal M_Z)$, then we have a quasi-isomorphism $$\rg_{\operatorname{cris,Tsu,c}}(X/S)\simeq \rg_{\operatorname{dR,Tsu,c}}(X^{\operatorname{PD}}/S).$$
			
			We have the log-crystalline Poincar\'e duality by Tsuji \cite{tsuji1999poincare}.
			
			\begin{thm}\cite[Proposition 5.4 and Theorem 5.5]{tsuji1999poincare} \label{tsujipoincare}
				Suppose moreover that $X$ is proper. Then 
				
				(i) There is a natural trace map $${\operatorname{tr_{HK}}}:\rg_{\operatorname{HK,Tsu,c}}(X/S) \to F.$$
				
				(ii) The pairing $$H^i_{\operatorname{HK}}(X) \otimes H^{2d-i}_{\operatorname{HK,Tsu,c}}(X/S) \xrightarrow{} H^{2d}_{\operatorname{HK,Tsu,c}}(X/S) \to F$$ is perfect.
				\end{thm}
				
				We have a similar description of Tsuji's compactly supported cohomology as last subsection.
				
				\begin{lem}
					Suppose $(U, \overline{U})\in \operatorname{Var}^{\operatorname{ss}}_K$ is a strict semistable pair, which is defined in Definition \ref{modelalgebraic}. Moreover assume that $D:=\overline{U}_{\eta}-U$ is a smooth irreducible divisor, denote by $\mathscr D:=\overline{D}.$ Endow $(U, \overline{U})$ with the log structure defined by the compacting log structure of the open immersion $U\hookrightarrow \overline U.$ Then we have $$\rg_{\operatorname{HK,Tsu,c}}(\overline{U}_0) \simeq [\rg_{\operatorname{HK}}(\overline{U}_0) \to \rg_{\operatorname{HK}}(\mathscr D_0) ].$$
				\end{lem}
				
				\begin{proof}
					Write $i$ for the closed immersion $\mathscr D_0 \to \overline U_0.$ It suffices to show that we have a distinguished triangle $$Ru_*\mathcal K_{\overline U_0/S}\to Ru_*\mathcal O_{\overline U_0/S} \to Ru_*(i_*\mathcal O_{\mathscr D_0/S}).$$
					
					Denote by $(\overline U_0^{\operatorname{PD}},\mathcal M^{\operatorname{PD}})$ be the log PD-envelop of $(U, \overline{U})_0$, and let $(\mathscr D_0^{\operatorname{PD}},\mathcal M^{\operatorname{PD}}_{\mathscr D_0})$ be the log PD-envelop of the closed immersion $\mathscr D_0 \hookrightarrow \overline U_0$. Since $(U,\overline U_0)$ is log-smooth over $(S,\mathcal N),$ the above triangle equals $$\mathcal{I}_{f^{\operatorname{PD}}}\mathcal O_{\overline U_0^{\operatorname{PD}}} \otimes_{\mathcal O_{\overline{U}^{\operatorname{PD}}}} \Omega^{\bullet}_{\overline U_0^{\operatorname{PD}}/S}\to \Omega^{\bullet}_{\overline U_0^{\operatorname{PD}}/S} \to i_*\Omega^{\bullet}_{{\mathscr D}_0^{\operatorname{PD}}/S}.$$ Since the log structure of $\overline U_0^{\operatorname{PD}}$ is induced by $\mathscr D_0^{\operatorname{PD}}$, the sheaf $\mathcal{I}_{f^{\operatorname{PD}}}\mathcal O_{\overline U_0^{\operatorname{PD}}}$ corresponds to the closed immersion of $\mathscr D_0 \hookrightarrow \overline U_0$. This concludes the proof.
				\end{proof}
								
				\subsubsection{Comparison with compactly supported algebraic Hyodo-Kato cohomology}
				
				\begin{theorem} \label{tsujialg}
					Suppose that $(U, \overline{U})\in \operatorname{Var}^{\operatorname{ss}}_K$ is a strict semistable pair, which is defined in Definition \ref{modelalgebraic}. Then we have a quasi-isomorphism:
					$$\rg_{\operatorname{HK,Tsu,c}}(\overline{U}_0) \xrightarrow{\simeq} \hkc (U),$$ which is compatible with Frobenius, monodromy, and Hyodo-Kato isomorphism.
				\end{theorem}
				
				\begin{proof}
					By using the Poincar\'e duality (Theorem \ref{tsujipoincare} and \ref{algpoincare}), the quasi-isomorphism is constructed by the following dotted arrow:
					$$\begin{tikzcd}
						\rg_{\operatorname{HK,Tsu,c}}(\overline{U}_0) \arrow[r,"\simeq"] \arrow[d, dashed] & \hk(U)^\vee \arrow[d,"\simeq"] \\
						\hkc (U) \arrow[r,"\simeq"]& \hk(U)^\vee.
					\end{tikzcd}$$	
				\end{proof}
				
				\subsubsection{Comparison with compactly supported analytic Hyodo-Kato cohomology}
				
				We keep most of notations as above, but now let $\mathscr X$ be a proper semistable formal scheme over $\operatorname{Spf}(\mathcal O_K)$. Locally $\mathscr{X}$ can be written as $\operatorname{Spf}(R)$ with $R$ the completion of an \'etale algebra over $$\mathcal O_K\langle x_1,x_2,...,x_d,\dfrac{1}{x_1...x_a},\dfrac{1}{x_{a+1}...x_{a+b}}\rangle,$$ and we equip $\mathscr{X}$ with the log-structure induced by the divisors $\mathscr{D}:=(x_{a+1}...x_{d}=0)$. Let $Y$ be the special fiber $\mathscr X_0$ of $\mathscr X$. Denote by $\mathcal{M}_Y$ the log-structure on $Y$ induced by $\mathscr M.$ Let $X$ be the rigid analytic variety over $K$ associated to $\mathscr X$ and denote by $X_{\operatorname{tr}}$ its trivial locus. Locally $X_{\operatorname{tr}}$ is equal to $$\Sp(R\left[\dfrac{1}{p}\right])\setminus(x_{a+1}...x_{d}=0).$$
				
				Since we also have the Poincar\'e duality for almost proper rigid analytic varieties (Theorem \ref{hkrigidpoincare}), the same argument as in the last subsection gives us the comparison with compactly supported analytic Hyodo-Kato cohomology.
				
				\begin{theorem} \label{tsujian}
					We have a quasi-isomorphism:
					$$\rg_{\operatorname{HK,Tsu,c}}(\mathscr{X}_0) \xrightarrow{\simeq} \hkc (X_{\operatorname{tr}}),$$  which is compatible with Frobenius, monodromy, and Hyodo-Kato isomorphism.
				\end{theorem}
				
				\begin{proof}
					This follows from Theorem \ref{hkrigidpoincare}, Theorem \ref{tsujipoincare} and the same construction as Theorem \ref{tsujialg}.
				\end{proof}

				\appendix	
				
				\section{de Rham cohomology of Elmar Grosse-Kl\"onne} \label{appa}
				
				In this appendix we review the definition of de Rham cohomology defined by Elmar Grosse-Kl\"onne in \cite{grosse2004derham}. Elmar Grosse-Kl\"onne introduced a definition of de Rham cohomology for rigid analytic variety $X$, by locally embedding $X$ into a rigid analytic variety which has a dagger structure. In this appendix we will only study the dagger variety. We will show that this definition agrees with the overconvergent de Rham cohomology defined in \cite{CNderham} and \cite{bosco2023rational}. We need this definition to prove the open-closed long exact sequence for cohomology with compact support in Proposition \ref{open-closed}.
				
				Let $X$ be a dagger variety over $L$, $Y$ be a smooth dagger over $L$ and let $\phi:X \hookrightarrow Y$ be a closed immersion. Those $X$ form a site $\operatorname{EmbRig}_{L,\an}^{\dagger}$ equipped with the analytic topology, due to the following lemma. In fact, the following lemma implies we can equip $\operatorname{EmbRig}_{L}^{\dagger}$ with any topology coming from $\operatorname{Rig}_{L}^{\dagger}$.
				
				\begin{lem}
					Suppose $X,X',X'' \in \operatorname{EmbRig}_{L}^{\dagger}$ and $f:X' \xrightarrow{} X, g:X'' \xrightarrow{} X$ be two morphisms. Then $X' \times_X X'' \in \operatorname{EmbRig}_{L}^{\dagger}.$ 
				\end{lem}
				
				\begin{proof}
					Take closed immersions $X \hookrightarrow Z, X' \hookrightarrow Z', X'' \hookrightarrow Z''$ such that $Z,Z',Z''$ are smooth. We can construct closed immersions by taking the diagonal embedding (also known as the graph) $X' \hookrightarrow Z \times Z', X'' \hookrightarrow Z \times Z''$. Therefore we have a close immersion $X' \times_X X''=X' \times_Z X'' \hookrightarrow (Z \times Z')\times_Z (Z \times Z'')=Z\times Z'\times Z''.$
				\end{proof}
				
				It is clear $\operatorname{EmbRig}_{L}^{\dagger}$ form a basis of $\operatorname{Rig}_{L,\an}^{\dagger}$. We define a sheaf $$\mathscr{A}_{\operatorname{dR,GK}}:\operatorname{Rig}_{L,\an}^{\dagger} \xrightarrow{} \mathcal{D}(\Mod^{\operatorname{cond}}_L)$$
				by the following way: For $X \in \operatorname{EmbRig}_{L}^{\dagger}$, let $\phi:X \hookrightarrow Y$ be a closed immersion into a smooth dagger variety $Y$. We denote by $\Phi(\phi,T)$ the set of admissible open neighborhood $U \subset Y$ of $X$. Denote by $j_U$ the open immersion $U \to Y$. For a abelian sheaf $\f$ on $Y$, we define $$\f^{\phi}=\colim_{U \in \Phi(\phi,T)}j_{U,*}\f|_U.$$ We define $\mathscr{A}_{\operatorname{dR,GK}}$ to be the sheaf associated to
				by
				$$X \mapsto \rg(X,\phi^{-1}(\Omega_{Y/L}^{\bullet}))=\rg(Y,(\Omega_{Y/L}^{\bullet})^{\phi}).$$
				The following proposition guarantees the above definition makes sense.
				
				\begin{prop}[\cite{grosse2004derham}, Proposition 1.6] $\rg(X,\phi^{-1}(\Omega_{Y/L}^{\bullet}))$ is independent of $Y$ and $\phi$, it depends only on the reduced structure of $X$. Moreover,  $\rg(X,\phi^{-1}(\Omega_{Y/L}^{\bullet}))$ is a contravariant functor in $X$.
				\end{prop}
				
				For $X \in \operatorname{Rig}_{L}^{\dagger},$ we define the de Rham cohomology (of Elmar Grosse-Kl\"onne) of $X$ by
				$$\operatorname{R\Gamma_{dR}^{GK}}(X):=\operatorname{R\Gamma}(X_{\an},\mathscr{A}_{\operatorname{dR,GK}}).$$
				We have local-global compatibility: 
				
				\begin{prop}
					If $X \in \operatorname{EmbRig}_{L}^{\dagger}$ and $\phi:X \hookrightarrow Y$ is a closed immersion into a smooth dagger variety $Y$, then
					$$\operatorname{R\Gamma_{dR}^{GK}}(X)=R\Gamma(X,\phi^{-1}(\Omega_{Y/L}^{\bullet})).$$
				\end{prop}
				
				\begin{proof}
					The proof is similar to \cite[2.1]{hartshorne1975derham}. For any embedded open covering $\mathcal{U}:=\{X_i,Y_i\}_{i\in I}$, where $$X^1:=\coprod_{i \in I}X_i \to X, X_i \hookrightarrow Y_i,$$ where $X_i \hookrightarrow Y_i$ are closed immersions into smooth dagger varieties $Y_i$. Denote by $$X^n:=(X^1)^{\times n}=\coprod_{(i) \in I^n}X_{(i)},X_{(i)} \hookrightarrow Y_{(i)},$$ where for $(i)=(i_1,...,i_n) \in I^n,$ $X_{(i)}=X_{i_1}\times_{X}...\times_{X} X_{i_n}$, and $\phi_{(i)}:X_{(i)} \hookrightarrow Y_{(i)} :=Y_{i_1}\times...\times Y_{i_n}$ are given by the diagonal embedding. Therefore we can defined simplicial objects $X^{\bullet}$ and $Y^{\bullet}$. For any $(i) \in I^n$ define the sheaf $C_{(i)}$ on $X$ by $C_{(i)}:=j_*\phi_{(i)}^{-1}\Omega_{Y_{(i)}/L}^{\bullet}.$ We can define the \v{C}ech complex $\mathcal{C}(\mathcal{U})$ associated to $\{X_i,Y_i\}_{i\in I}$ by $\mathcal{C}^n(\mathcal{U})=\prod_{(i) \in I^n}C_{(i)}.$ The boundary maps are given naturally from the projection of $Y_{(i)}.$
					
					We need to prove that
					$$R\Gamma(X,\phi^{-1}(\Omega_{Y/L}^{\bullet}))\simeq R\Gamma(X^{\bullet},\phi^{-1}(\Omega_{Y^{\bullet}/L}^{\bullet}))=R\Gamma(X,\mathcal{C}(\mathcal{U})),$$
					where the last equality comes from the cohomological descent. It suffices to show that $R\Gamma(X,\mathcal{C}(\mathcal{U}))$ is independent of the embedded covering $\mathcal{U}:=\{X_i,Y_i\}_{i\in I}.$ For two embedded covering $\mathcal{U}:=\{X_i,Y_i\}_{i\in I}$ and $\mathcal{U}':=\{X'_i,Y'_i\}_{j\in J},$ we call $\mathcal{U}'$ is a refinement of $\mathcal{U}$ if each $X_i$ is a open subset of $X'_{\lambda(i)}$ together with a smooth morphisms $Y_i \to Y'_{\lambda(i)}$ compatible with the inclusion. For $\mathcal{U}'$ a refinement of $\mathcal{U}$, we have the natural map
					$$u:\mathcal{C}(\mathcal{U})\to \mathcal{C}(\mathcal{U}').$$ Since any two systems of local embedding clearly have a common refinement, we are reduced to showing that $\rg(X,{u})$ is an isomorphism.
					
					The problem is local, so we may assume $X_1=X'_1=X.$ Denote by $\mathcal{V}$ and $\mathcal{V}'$ be the embedded open covering $\{X_1,Y_1\}$ and $\{X'_1,Y'_1\}$ respectively. By the above proposition, we have $$R\Gamma(X,\mathcal{C}(\mathcal{V})) \simeq R\Gamma(X,\mathcal{C}(\mathcal{V}')).$$ Since we have the commutative diagram $$\begin{tikzcd}
						\mathcal{C}(\mathcal{V}) \arrow[r] \arrow[d] &\mathcal{C}(\mathcal{V}') \arrow[d] \\
						\mathcal{C}(\mathcal{U}) \arrow[r] & \mathcal{C}(\mathcal{U}'),
					\end{tikzcd}$$ it suffices to show the map $\mathcal{C}(\mathcal{V})\to \mathcal{C}(\mathcal{U})$ is a quasi-isomorphism.
					
					Write $$\mathcal{C}(\mathcal{U})=\mathcal{C}'(\mathcal{U})+\mathcal{C}''(\mathcal{U}),$$ where for each $n$, $$\mathcal{C}'^n(\mathcal{U}):=\prod_{(i) \in I^n, i_1 \neq 1}C_{(i)},\mathcal{C}''^n(\mathcal{U}):=\prod_{(i) \in I^n, i_1 = 1}C_{(i)}.$$ Now for each $(i)=(i_1,...,i_n) \in I^n,$ the boundary map $\delta:C_{(i)} \to C_{(1,i_1,...,i_n)}$ is an isomorphism, since $X_{(i)}$ and $X_{(1,i_1,...,i_n)}$ are equal. Hence the map $\delta$  gives a quasi-isomorphism of $\mathcal{C}'(\mathcal{U})$ onto its image in $\mathcal{C}''(\mathcal{U})$, which is everything except $\mathcal{C}(\mathcal{V})$. Thus the map $\mathcal{C}(\mathcal{V})\to \mathcal{C}(\mathcal{U})$ is a quasi-isomorphism. This concludes the proof.
				\end{proof}
				
				\begin{rem}
					We only consider the overconvergent setup. It is clear that the definition of $\operatorname{R\Gamma_{dR}^{GK}}(X)$ in this article agrees with the definition of $\operatorname{R\Gamma_{dR}}(\widehat{X})$ in \cite{grosse2004derham}, and therefore all propositions in \cite{grosse2004derham} apply here.
				\end{rem}
				
				The main result of this appendix is the follow.
				
				\begin{thm}\label{drgk}
					For $X \in \operatorname{Rig}_{L}^{\dagger},$ there is a natural isomorphism in $D(\Mod^{\operatorname{cond}}_L)$:
					$$\operatorname{R\Gamma_{dR}^{GK}}(X) \xrightarrow{\simeq} \dr(X).$$
				\end{thm}
				
				\begin{proof}
					The question is local, so we can assume $X \in \operatorname{EmbRig}_{L}^{\dagger}$ and $\phi:X \hookrightarrow Y$ be a closed immersion into a quasi compact smooth dagger variety $Y$.
					
					The map can be constructed as the composition of
					$$\operatorname{R\Gamma_{dR}^{GK}}(X)=\rg(X,\phi^{-1}(\Omega_{Y/L}^{\bullet})) \xrightarrow{} \rg(X,\Omega_{X/L}^{\bullet}) \xrightarrow{} \dr(X).$$
					The first map is canonical. For the second one, denote by $\pi: \operatorname{Rig}_{L,\eh}^\dagger \xrightarrow{} \operatorname{Rig}_{L,\an}^\dagger$ the natural map between big sites. Recall that the sheaf $\Omega^i_{\eh}$ is the $\eh$-sheafification of the continuous differential, i.e. $\Omega^i_{\eh}=\pi^{-1}\Omega^i_{/K}.$ We can construct the second map by the counit:
					$$\Omega^{\bullet}_{/K} \xrightarrow{} R\pi_*\pi^{-1}\Omega^{\bullet}_{/K}=R\pi_*\Omega^{\bullet}_{\eh}.$$
					
					We proceed by induction on the dimension of $X$. By \cite[Proposition 2.1]{grosse2004derham} and Proposition \ref{mvmain}, both sides satisfy Mayer-Vietoris properties for a closed cover. Hence by induction on the number of irreducible components we can reduce to the case that $X$ is irreducible. If $X$ is smooth, by \cite[Proposition 1.8]{grosse2004derham} and Theorem \ref{ehdescentrigid} both sides agree with the usual de Rham cohomology defined by continuous differentials, and the map defined above is clearly an isomorphism. In general, we perform a resolution of singularities: use \cite[Proposition 2.2]{grosse2004derham} and Proposition \ref{mvmain} to reduce to the smooth case.
				\end{proof}
				
				The proof of the above theorem relies on the embedded desingularization for dagger varieties, which is an easy consequence of \cite{temkin2018embeddedresolution}.
				
				\begin{thm}[Embedded desingularization]
					Let $X$ be a quasi-compact smooth dagger variety over $L$ and $Y \hookrightarrow X$ be a closed immersion. Then there exist a finite sequence of blowups $X'=X_n \xrightarrow{} X_{n-1} \xrightarrow{} \cdots \xrightarrow{} X_0=X$ such that the strict transform of $Y$ is smooth, and each center of blowups is smooth and contained in the preimage of $Y$. 
				\end{thm}
				
				\begin{proof}
					Any dagger $L$-algebra is an excellent ring: It suffices to show that the Washnitzer algebra $W_n$ is excellent. By \cite[Proposition 1.5]{gross2000overconvergent}, the Washnitzer algebra $W_n$ is regular and equidimensional of dimension $n$, then it is clear that $W_n$ satisfies the conditions of \cite[Theorem 102]{matsumura1970commutative}. Then the same argument as in \cite[Sec. 5]{temkin2012resolution} gives the analogy result of \cite[Theorem 1.1.9]{temkin2018embeddedresolution}.
				\end{proof}
				
				It's worth to mention that we have therefore the finiteness results in our settings.
				
				\begin{prop} \label{finitenesshk}
					Let $X$ be a quasi-compact smooth dagger variety over $L=K$ or $C$, $U \subset X$ a quasi-compact open subset, $Z \hookrightarrow X$ a closed immersion, $T=X-(U\cup Z)$.  Let $i\ge 0$. Then the condensed cohomology group $H_{\operatorname{HK}}^i(T)$ (resp. $H_{\operatorname{dR}}^i(T)$) is a finite-dimensional condensed vector space over $L_F=F$ or $F^{\operatorname{nr}}$ (resp. over $L$). 
				\end{prop}
				
				\begin{proof}
					When $X$ is defined over $C$, the proposition follows from the Hyodo-Kato isomorphism, together with \cite[Corollary 3.5]{grosse2004derham} and \cite[Theorem 3.6]{grosse2004derham}; it only remains to pass to the framework of condensed mathematics. We only need to do it for de Rham cohomology: in fact, similar to \cite{grosse2004derham}, it suffices to establish the proposition in two cases: for $X-U$ where $X$ is smooth dagger affinoid and $U$ is a rational subdomain of $X$, and for $X-Z$ where $X$ is smooth dagger affinoid and $Z$ is a strictly normal crossing divisor. The first case follows from \cite[Theorem 3.6]{grosse2004derham} and the fact that we can write $\dr(X-U)$ as filtered colimits of de Rham cohomology of Stein spaces (with finite de Rham cohomological groups), and the second case follows from \cite[Corollary 3.5]{grosse2004derham}, \cite[Theorem 2.8]{xsperiod} and \cite[Proposition 2.10]{xsperiod}. The result over $K$ then follows by Galois descent.
				\end{proof}

				\bibliography{ref}

\begin{thebibliography}{DLLZ23b}

\bibitem[AGN25]{AGNcompact}
Piotr Achinger, Sally Gilles, and Wiesława Nizioł.
\newblock Compactly supported $p$-adic pro-\'etale cohomology of analytic
  varieties, 2025, 2501.13651.

\bibitem[Bei12]{beilinson2012derham}
A.~Beilinson.
\newblock {$p$}-adic periods and derived de {R}ham cohomology.
\newblock {\em J. Amer. Math. Soc.}, 25(3):715--738, 2012.

\bibitem[Bei13]{beilinson2013crys}
A.~Beilinson.
\newblock On the crystalline period map.
\newblock {\em Camb. J. Math.}, 1(1):1--51, 2013.

\bibitem[Bey97]{beyer1997serreduality}
Peter Beyer.
\newblock On {S}erre-duality for coherent sheaves on rigid-analytic spaces.
\newblock {\em Manuscripta Math.}, 93(2):219--245, 1997.

\bibitem[BGR84]{bosch1984nonarchimedean}
S.~Bosch, U.~G\"{u}ntzer, and R.~Remmert.
\newblock {\em Non-{A}rchimedean analysis}, volume 261 of {\em Grundlehren der
  mathematischen Wissenschaften [Fundamental Principles of Mathematical
  Sciences]}.
\newblock Springer-Verlag, Berlin, 1984.
\newblock A systematic approach to rigid analytic geometry.

\bibitem[BMS18]{bms1}
Bhargav Bhatt, Matthew Morrow, and Peter Scholze.
\newblock Integral {$p$}-adic {H}odge theory.
\newblock {\em Publ. Math. Inst. Hautes \'Etudes Sci.}, 128:219--397, 2018.

\bibitem[Bos23]{bosco2023rational}
Guido Bosco.
\newblock Rational $p$-adic hodge theory for rigid-analytic varieties, 2023,
  2306.06100.

\bibitem[CD19]{deglise2019motive}
Denis-Charles Cisinski and Fr\'ed\'eric D\'eglise.
\newblock {\em Triangulated categories of mixed motives}.
\newblock Springer Monographs in Mathematics. Springer, Cham, [2019] \copyright
  2019.

\bibitem[CDN20]{CDNstein}
Pierre Colmez, Gabriel Dospinescu, and Wiesława Nizioł.
\newblock Cohomology of {$p$}-adic {S}tein spaces.
\newblock {\em Invent. Math.}, 219(3):873--985, 2020.

\bibitem[CF00]{colmezfontaine2000ss}
Pierre Colmez and Jean-Marc Fontaine.
\newblock Construction des repr\'{e}sentations {$p$}-adiques semi-stables.
\newblock {\em Invent. Math.}, 140(1):1--43, 2000.

\bibitem[CN17]{CNsytntomic}
Pierre Colmez and Wiesława Nizioł.
\newblock Syntomic complexes and {$p$}-adic nearby cycles.
\newblock {\em Invent. Math.}, 208(1):1--108, 2017.

\bibitem[CN20]{CNderham}
Pierre Colmez and Wiesława Nizioł.
\newblock On {$p$}-adic comparison theorems for rigid analytic varieties, {I}.
\newblock {\em M\"{u}nster J. Math.}, 13(2):445--507, 2020.

\bibitem[CN24]{CN5}
Pierre Colmez and Wiesława Nizioł.
\newblock On the cohomology of p-adic analytic spaces, {II}: The $c_{\rm
  st}$-conjecture, 2024, 2108.12785.

\bibitem[CN25]{CN4.3}
Pierre Colmez and Wiesława Nizioł.
\newblock On the cohomology of {$p$}-adic analytic spaces, {I}: {T}he basic
  comparison theorem.
\newblock {\em J. Algebraic Geom.}, 34(1):1--108, 2025.

\bibitem[Con06]{conard2006relative}
Brian Conrad.
\newblock Relative ampleness in rigid geometry.
\newblock {\em Ann. Inst. Fourier (Grenoble)}, 56(4):1049--1126, 2006.

\bibitem[CS19]{cscondensed}
Dustin Clausen and Peter Scholze.
\newblock Lectures on condensed mathematics.
\newblock 2019.

\bibitem[DLLZ23a]{dllz2023logrh}
Hansheng Diao, Kai-Wen Lan, Ruochuan Liu, and Xinwen Zhu.
\newblock Logarithmic {R}iemann-{H}ilbert correspondences for rigid varieties.
\newblock {\em J. Amer. Math. Soc.}, 36(2):483--562, 2023.

\bibitem[DLLZ23b]{dllz2023logadic}
Hansheng Diao, Kai-Wen Lan, Ruochuan Liu, and Xinwen Zhu.
\newblock Logarithmic adic spaces: some foundational results.
\newblock In {\em {$p$}-adic {H}odge theory, singular varieties, and
  non-abelian aspects}, Simons Symp., pages 65--182. Springer, Cham, [2023]
  \copyright2023.

\bibitem[DN18]{nizioldeglise2018syntomic}
Fr\'{e}d\'{e}ric D\'{e}glise and Wiesława Nizioł.
\newblock On {$p$}-adic absolute {H}odge cohomology and syntomic coefficients.
  {I}.
\newblock {\em Comment. Math. Helv.}, 93(1):71--131, 2018.

\bibitem[EGN24]{ertl2024vpicardgroupsteinspaces}
Veronika Ertl, Sally Gilles, and Wiesława Nizioł.
\newblock On the {$v$}-{P}icard group of {S}tein spaces.
\newblock {\em Int. Math. Res. Not. IMRN}, (20):13352--13379, 2024.

\bibitem[FF18]{fargues2018courbe}
Laurent Fargues and Jean-Marc Fontaine.
\newblock Courbes et fibr\'{e}s vectoriels en th\'{e}orie de {H}odge
  {$p$}-adique.
\newblock {\em Ast\'{e}risque}, (406):xiii+382, 2018.
\newblock With a preface by Pierre Colmez.

\bibitem[Fon94]{fontaine1994semistable}
Jean-Marc Fontaine.
\newblock Repr\'{e}sentations {$p$}-adiques semi-stables.
\newblock Number 223, pages 113--184. 1994.
\newblock With an appendix by Pierre Colmez, P\'{e}riodes $p$-adiques
  (Bures-sur-Yvette, 1988).

\bibitem[Gei06]{thomas2006eh}
Thomas Geisser.
\newblock Arithmetic cohomology over finite fields and special values of
  {$\zeta$}-functions.
\newblock {\em Duke Math. J.}, 133(1):27--57, 2006.

\bibitem[Gil23]{gilles2023period}
Sally Gilles.
\newblock Morphismes de p\'eriodes et cohomologie syntomique.
\newblock {\em Algebra Number Theory}, 17(3):603--666, 2023.

\bibitem[GK00]{gross2000overconvergent}
Elmar Grosse-Kl\"{o}nne.
\newblock Rigid analytic spaces with overconvergent structure sheaf.
\newblock {\em J. Reine Angew. Math.}, 519:73--95, 2000.

\bibitem[GK04]{grosse2004derham}
Elmar Grosse-Kl\"{o}nne.
\newblock De {R}ham cohomology of rigid spaces.
\newblock {\em Math. Z.}, 247(2):223--240, 2004.

\bibitem[Gro61]{EGA3I}
A.~Grothendieck.
\newblock \'{E}l\'{e}ments de g\'{e}om\'{e}trie alg\'{e}brique. {III}.
  \'{E}tude cohomologique des faisceaux coh\'{e}rents. {I}.
\newblock {\em Inst. Hautes \'{E}tudes Sci. Publ. Math.}, (11):167, 1961.

\bibitem[Guo23]{guo2019hodgetate}
Haoyang Guo.
\newblock Hodge-{T}ate decomposition for non-smooth spaces.
\newblock {\em J. Eur. Math. Soc. (JEMS)}, 25(4):1553--1625, 2023.

\bibitem[Har66]{hartshorne1966residue}
Robin Hartshorne.
\newblock {\em Residues and duality}, volume No. 20 of {\em Lecture Notes in
  Mathematics}.
\newblock Springer-Verlag, Berlin-New York, 1966.
\newblock Lecture notes of a seminar on the work of A. Grothendieck, given at
  Harvard 1963/64, With an appendix by P. Deligne.

\bibitem[Har75]{hartshorne1975derham}
Robin Hartshorne.
\newblock On the de {Rham} cohomology of algebraic varieties.
\newblock {\em Publications Math\'ematiques de l'IH\'ES}, 45:5--99, 1975.

\bibitem[Hir64]{hironaka1964resolution}
Heisuke Hironaka.
\newblock Resolution of singularities of an algebraic variety over a field of
  characteristic zero. {I}, {II}.
\newblock {\em Ann. of Math. (2)}, 79:109--203; 79 (1964), 205--326, 1964.

\bibitem[HK94]{hyodokato1994}
Osamu Hyodo and Kazuya Kato.
\newblock Semi-stable reduction and crystalline cohomology with logarithmic
  poles.
\newblock Number 223, pages 221--268. 1994.
\newblock P\'{e}riodes $p$-adiques (Bures-sur-Yvette, 1988).

\bibitem[Hub96]{huber2013etale}
Roland Huber.
\newblock {\em \'{E}tale cohomology of rigid analytic varieties and adic
  spaces}.
\newblock Aspects of Mathematics, E30. Friedr. Vieweg \& Sohn, Braunschweig,
  1996.

\bibitem[Hub98]{huber1998comparison}
R.~Huber.
\newblock A comparison theorem for {$l$}-adic cohomology.
\newblock {\em Compositio Math.}, 112(2):217--235, 1998.

\bibitem[Hyo91]{hyodo1991derhamwitt}
Osamu Hyodo.
\newblock On the de {R}ham-{W}itt complex attached to a semi-stable family.
\newblock {\em Compositio Math.}, 78(3):241--260, 1991.

\bibitem[Jan89]{jannsen1989cohomology}
Uwe Jannsen.
\newblock On the {$l$}-adic cohomology of varieties over number fields and its
  {G}alois cohomology.
\newblock In {\em Galois groups over {${\bf Q}$} ({B}erkeley, {CA}, 1987)},
  volume~16 of {\em Math. Sci. Res. Inst. Publ.}, pages 315--360. Springer, New
  York, 1989.

\bibitem[K\"74]{kopf1974gaga}
Ursula K\"opf.
\newblock \"uber eigentliche {F}amilien algebraischer {V}ariet\"aten \"uber
  affinoiden {R}\"aumen.
\newblock {\em Schr. Math. Inst. Univ. M\"unster (2)}, pages iv+72, 1974.

\bibitem[Kat89]{kato1989log}
Kazuya Kato.
\newblock Logarithmic structures of {F}ontaine-{I}llusie.
\newblock In {\em Algebraic analysis, geometry, and number theory ({B}altimore,
  {MD}, 1988)}, pages 191--224. Johns Hopkins Univ. Press, Baltimore, MD, 1989.

\bibitem[Kat94]{kato1994toric}
Kazuya Kato.
\newblock Toric singularities.
\newblock {\em Amer. J. Math.}, 116(5):1073--1099, 1994.

\bibitem[Kie67]{kiehl1967derham}
Reinhardt Kiehl.
\newblock Die de {R}ham {K}ohomologie algebraischer {M}annigfaltigkeiten
  \"{u}ber einem bewerteten {K}\"{o}rper.
\newblock {\em Inst. Hautes \'{E}tudes Sci. Publ. Math.}, (33):5--20, 1967.

\bibitem[Kis99]{kisin1999localconstancy}
Mark Kisin.
\newblock Local constancy in {$p$}-adic families of {G}alois representations.
\newblock {\em Math. Z.}, 230(3):569--593, 1999.

\bibitem[LLZ23]{llz2023derham}
Kai-Wen Lan, Ruochuan Liu, and Xinwen Zhu.
\newblock de {R}ham comparison and {P}oincar\'{e} duality for rigid varieties.
\newblock {\em Peking Math. J.}, 6(1):143--216, 2023.

\bibitem[Mat70]{matsumura1970commutative}
Hideyuki Matsumura.
\newblock {\em Commutative algebra}, volume 120.
\newblock WA Benjamin New York, 1970.

\bibitem[NN16]{janniziol2016hk}
Jan Nekovář and Wiesława Nizioł.
\newblock Syntomic cohomology and {$p$}-adic regulators for varieties over
  {$p$}-adic fields.
\newblock {\em Algebra Number Theory}, 10(8):1695--1790, 2016.
\newblock With appendices by Laurent Berger and Fr\'ed\'eric D\'eglise.

\bibitem[Sch13]{scholze2013p}
Peter Scholze.
\newblock {$p$}-adic {H}odge theory for rigid-analytic varieties.
\newblock {\em Forum Math. Pi}, 1:e1, 77, 2013.

\bibitem[SGA72]{SGA4I}
{\em Th\'{e}orie des topos et cohomologie \'{e}tale des sch\'{e}mas. {T}ome 1:
  {T}h\'{e}orie des topos}.
\newblock Lecture Notes in Mathematics, Vol. 269. Springer-Verlag, Berlin-New
  York, 1972.
\newblock S\'{e}minaire de G\'{e}om\'{e}trie Alg\'{e}brique du Bois-Marie
  1963--1964 (SGA 4), Dirig\'{e} par M. Artin, A. Grothendieck, et J. L.
  Verdier. Avec la collaboration de N. Bourbaki, P. Deligne et B. Saint-Donat.

\bibitem[Sha25a]{xslogsyn}
Xinyu Shao.
\newblock Logarithmic syntomic cohomology for analytic varieties via period
  sheaves.
\newblock 2025.

\bibitem[Sha25b]{xsperiod}
Xinyu Shao.
\newblock On the kummer pro-\'etale cohomology of {$\mathbb
  B_{\operatorname{dR}}$}, 2025, 2501.16916.

\bibitem[SV96]{voevodsky1996h}
Andrei Suslin and Vladimir Voevodsky.
\newblock Singular homology of abstract algebraic varieties.
\newblock {\em Invent. Math.}, 123(1):61--94, 1996.

\bibitem[Tem12]{temkin2012resolution}
Michael Temkin.
\newblock Functorial desingularization of quasi-excellent schemes in
  characteristic zero: the nonembedded case.
\newblock {\em Duke Math. J.}, 161(11):2207--2254, 2012.

\bibitem[Tem17]{temkin2017alteration}
Michael Temkin.
\newblock Altered local uniformization of {B}erkovich spaces.
\newblock {\em Israel J. Math.}, 221(2):585--603, 2017.

\bibitem[Tem18]{temkin2018embeddedresolution}
Michael Temkin.
\newblock Functorial desingularization over {$\bf Q$}: boundaries and the
  embedded case.
\newblock {\em Israel J. Math.}, 224(1):455--504, 2018.

\bibitem[Tsu99]{tsuji1999poincare}
Takeshi Tsuji.
\newblock Poincar\'{e} duality for logarithmic crystalline cohomology.
\newblock {\em Compositio Math.}, 118(1):11--41, 1999.

\bibitem[vdP92]{van1992serreduality}
Marius van~der Put.
\newblock Serre duality for rigid analytic spaces.
\newblock {\em Indag. Math. (N.S.)}, 3(2):219--235, 1992.

\bibitem[Vez18]{vezzani2018monskywashnitzer}
Alberto Vezzani.
\newblock The {M}onsky-{W}ashnitzer and the overconvergent realizations.
\newblock {\em Int. Math. Res. Not. IMRN}, (11):3443--3489, 2018.

\bibitem[Voe]{voevodsky2000motive}
Vladimir Voevodsky.
\newblock Triangulated categories of motives over a field.
\newblock In {\em Cycles, transfers, and motivic homology theories}, volume 143
  of {\em Ann. of Math. Stud.}, pages 188--238.

\bibitem[Zav24]{zavyalov2024modppoincaredualitypadic}
Bogdan Zavyalov.
\newblock Mod-p poincar\'e duality in p-adic analytic geometry, 2024,
  2111.01830.

\end{thebibliography}
				
				\bibliographystyle{halpha}	
				
			\end{document}